\documentclass[11pt]{article}
\usepackage{amsmath,amsfonts,amssymb,amsthm,empheq}
\usepackage{subfigure,graphics,graphicx,psfrag,verbatim} 
\usepackage{mathrsfs}

\usepackage[preprint]{jmlr2e-arxiv}
\usepackage{microtype}
\usepackage{standard-arxiv}

\usepackage[normalem]{ulem} 
\usepackage{framed}

\usepackage[latin1]{inputenc}
\usepackage[T1]{fontenc}
\usepackage[all]{xy}
\usepackage{amsmath}
\usepackage{tikz-cd}
\tikzcdset{scale cd/.style={every label/.append style={scale=#1},
    cells={nodes={scale=#1}}}}
    
\usepackage[margin=1in]{geometry}

\usepackage{url}
\usepackage{enumerate}
\usepackage{hyperref}
\usepackage{cleveref}

\usepackage{makeidx}
\makeindex

\newtheorem{theo}{Theorem}[section]
\newtheorem{defi}{Definition}[section]

\newtheorem{axiom1}{Axiom}
\newtheorem{coro}{Corollary}
\newtheorem{axioms}{Axiom}

\crefname{axiom1}{Axiom}{Axioms}
\crefname{axioms}{Axiom}{Axioms}

\newcommand{\pow}{\mathrm{pow}}

\title{ {An axiomatic treatment of Persistent Homology}  }

\author{
    \name S. T. Ura
    \email{sergio.ura@gmail.com}, \\
    \addr{S\~ao Paulo State University (UNESP), Department of Mathematics, Brazil}
    \AND 
    \name M.\ A. \ F.\ Contessoto\
    \email{marco.contessoto@unesp.br}\\
    \addr{S\~ao Paulo State University (UNESP), Department of Mathematics,  Brazil}
    \AND
    \name Alice K.\ M.\ Libardi 
    \email{alice.libardi@unesp.br}, \\
    \addr{S\~ao Paulo State University (UNESP), Department of Mathematics,  Brazil}
  }

\begin{document}

\maketitle

\begin{abstract}
We develop an axiomatic framework for persistent homology in any  degree. We prove the existence and uniqueness for both a persistent version of the Eilenberg-Steenrod axioms for classical homology and a reduced version of this set of axioms. \footnote{This work is partially supported by the Projeto Tem\'atico: Topologia Alg\'ebrica, Geom\'etrica e Diferencial, FAPESP Process Numbers 2016/24707-4 and 2022/16455-6}
\end{abstract}

\tableofcontents

\section{Introduction}

In different branches of mathematics,  questions often fall into two main types: those about existence and those about classification. As soon as the existence question has a positive answer, we are faced with the classification question. The classification problem is typically attacked by introducing  notions of equivalence between objects and by identifying quantities that   behave well under  these notions of equivalence.

\paragraph{Algebraic topology.} In terms of studying topological spaces, Algebraic and Differential Topology attack these questions by associating "invariant objects" to the topological spaces. These objects  are often integers but  richer information can sometimes be extracted from invariant algebraic structures such as groups and rings.
The first results in Algebraic Topology were announced by Poincar\'e in  Comptes-Rendus Note of $1892$.  In $1895$ he published the theory developed in a paper entitled ``Analysis Situs" \cite{Poincare}, and  between $1899$ and $1905$ he complemented  the theory through the papers called ``Compl\'ements \`a l\'a Analysis Situs" \cite{Poincare}.

Many mathematicians before Poincar\'e tried to attach \emph{numbers} invariant under homeomorphism. 
Poincar\'e was the first to introduce the idea of 
computing with topological
objects, not only with numbers, by defining
the concepts of homology and of fundamental group 
as topological invariants. Homology and Homotopy were in fact first introduced for  applications to problems of topology, but many  applications to other parts of mathematics  such  categories and functors, homological algebra,  K-theory were subsequently found.
One  central concept in homology is the category of generalized chain complexes in an abelian category, mentioned by  Poincar\'e, in $1900$, after  a frustrated attempt to give a mathematical formulation to the intuitive ideas of Riemann and Betti, in $1895$. 

The following is account of some of the main developments that took place in the early history  of algebraic topology. 

\begin{itemize}

\item Initially, Poincar\'e restricted his attention to compact triangulated spaces (or finite cell complexes) and proved that a subdivision of every cell into smaller cells gives the same homology for the subdivided triangulation. So one can restrict attention to  euclidean simplicial complexes, in which every cell is a rectilinear simplex contained in some $\mathbb{R}^N$ with large $N$.

\item The primary goal of homology theory is attach to any triangulable space $X$ a system of "homology groups" $H_{n}(X)$, for $n\geq 0,$ invariant under homeomorphism. For compact spaces such as $S^n, P_n(\mathbb{R}), T^n$, it is not difficult define triangulations.    Poincar\'e conjectured that manifolds in general and algebraic varieties (with singularities) can be triangulated. In 1930, Cairns and van der Waerden  proved  this conjecture.

\item Around $1925$, several topologists tried to  define homology groups for spaces for which no triangulation is known, or even possible. The idea of singular homology arose from these attempts  and was eventually  provided  by Eilenberg. Other methods were  developed by Vietoris \cite{Vietoris}, Alexandroff \cite{Alexandroff}, Lefschetz \cite{Lefschetz} and \v{C}ech \cite{Cech}. In all cases,  the coefficient of  the homology group considered 
  was a free $\mathbb{Z}$-module. For singular homology, the dimension of $ H_p(X;Q)$  is called the $p$-th Betti number $b_p(X)$ of the space $X$.

\item The idea of duality was considered by Poincar\'e in \cite{Poincare} section $9$, but only in 1935, under the influence of Pontrjagin's duality theorem, the general definition of cohomology was given. Alexander and Spanier  gave a direct definition of cohomology for an arbitrary space. Then ``duality" in  algebraic topology was better understood as a consequence of this.

\item Around  1949, through comparing the various definitions of homology and cohomology it was proved that  \v{C}ech cohomology and Alexander--Spanier cohomology give isomorphic cohomology groups, whereas the singular cohomology of a space may differ from its \v{C}ech cohomology.

\item 
In $1945$, a new approach was considered by Eilenberg and Steenrod \cite{Eilenberg}. They selected a small number of  properties shared by the various theories and adopted them as axioms for a theory of homology (and cohomology). Also they showed that on the category of compact triangulable spaces all theories verifying the axioms give isomorphic groups (uniqueness); in other words, there is only one notion of homology (and cohomology) in that category \cite{Eilenberg}.

\item In $1946,$ Leray introduced the concept of  sheaf cohomology which may be considered as a general machinery applicable to problems dealing with "passage from local to global properties".  Problems of this type arose  in the 1880's in the study of analytic functions of several complex variables by Poincar\'e, Cousin and later Cartan and Oka. The machinery used in 1946 in the study of such  problem was gradually refined and generalized by Leray and other mathematicians until the present day.\footnote{Text extracted from the book of J. Dieudonn\'e \cite{J-D}.}

\item {Morse theory in its classical form provides a relationship between the critical points of a given function on a manifold and the topology of the manifold. Between   $1980$ to 1990, Morse theory was extended to infinite dimensions. In the book \cite {Morse}, in $1934$,  Morse  exposed his idea, one the most important and influential works of the twentieth century. }

\end{itemize}

\paragraph{The advent of persistent homology.} Ideas related to algebraic topology began to be considered in  applications in the early 1990s eventually leading to the notion of \emph{persistent homology} and to the field of \emph{topological data analysis} (TDA).

\begin{itemize}
\item 
Frosini was the first to introduce the concept of size function \cite{Frosini1990} in the early 1990s -- a precursor of what came to be known as persistent homology.\footnote{Size functions coincide with degree-0 persistent homology.} Size  functions were used to compare shapes \cite{Frosini-Landi,Ferri} with applications in in computer vision and pattern recognition
\cite{vuff,frosini5,frosini4,frosini6,frosini7,frosini8,frosini9,frosini10}.

\item In 1999 Robins \cite{Robins} defined the concept of persistent Betti number, which she used to study fractal sets related to the structure of specific materials and also in image processing. A possible algorithm for computer implementation was suggested and tested, proving that it is possible to obtain topological information about structures of a compact space from
finite approximations of it.

\item A few years later, Edelsbrunner et al. \cite{E-L-Z}  gave  
optimized algorithms to compute persistent homology in the setting of finite simplicial complexes.

\item Other manifestations of the idea of the idea of persistence were a posteriori found in the work of Morse \cite{morse} and
 Barannikov \cite{S-B}.

\item Once consolidated, the persistence theory has seen many subsequent developments, including the elucidation of its stability \cite{Cohen2007,Chazal},
multiparameter \cite{gunnarZomo}, its connection to the notion of M\"obius inversion 
\cite{Amit-Generalized}, cohomological persistent invariants \cite{Marco-F-L-A,Facundo-L-A}, and to the filling radius \cite{Sunhyuk-Facundo}.

 \item Hierarchical clustering and persistent homology are deeply connected, since applying the $\pi_0$ functor to a simplicial filtration we obtain a dendrogram describing the evolution of the connected components during the filtration. Concerned about this connection, Carlsson and M\'emoli \cite{Gunnar-Facundo1, Gunnar-Facundo2}  provided an axiomatic framework for classifying different clustering functors.

 \end{itemize}
\paragraph{Our contribution.} In this paper, we develop an axiomatic framework for persistent homology in any  degree. 

\paragraph{Structure of the paper.}

In Section \ref{sec:prelim}, we  establish basic definitions and terms used throughout the paper. This section includes a definition of persistent homology for filtered spaces.

Section \ref{axioms-section} not only presents a persistent version of the classical Eilenberg-Steenrod axioms for homology, but introduces also a reduced version to the set of axioms is identified. 
These apply to the $\mathcal{RF}$ category of filtered spaces (see Appendix \ref{apdxb}). 
In this section we show that these two sets of axioms are equivalent. Roughly speaking, it is shown that the Axiom \ref{ax5} is a consequence of the of the other axioms.

In Section \ref{sec:uniq}, we prove a uniqueness theorem. That is, we prove that there is only one persistent homology theory that satisfies the persistent version of the Eilenberg-Steenrod axioms, up to  isomorphism.

 The persistent homology theory defined in Section \ref{sec:prelim} satisfies all the axioms for persistent homology that we put forward.

In Appendix \ref{ap1} we introduce the theoretical elements needed in order to prove the uniqueness theorem. These include a series of theorems and definitions from the classical homology theory adapted to the context of persistent homology and filtered spaces. 
In Appendix \ref{apdxb} we present the definition of the category {$\mathscr{RF}$}, and the definition of a simplicial complex in the context of finite sets.

 \section{Preliminaries}\label{sec:prelim} \quad

Throughout this paper, let us fix 
$M$, a  large {integer} in the extended real number line $\overline{\mathbb{R}} = \mathbb{R}\cup\{-\infty, +\infty\}$ 
 (one can {set} $M=+\infty$). Consider the poset  ${\mathrm{Int}(\mathbb{R})} := \{ [\varepsilon, \varepsilon'], \, \varepsilon \leq \varepsilon']\},$ with the order $ I \leq  J$ if and only if, $ I \subseteq J$. {Whenever} we write $\varepsilon\leq \varepsilon'$, {we mean} $0\leq\varepsilon\leq \varepsilon'\leq M.$

 Given a  set $X$, we recall that $\pow(X)$ is the set of all subsets of $X$, and a function  $F_X: \pow(X)\longrightarrow\overline{\mathbb{R}}$ is a monotone function  if $F_X(\sigma)\leq F_X(\tau),$ whenever $\sigma\subseteq\tau$, for $\sigma,\tau\in \pow(X).$

 A \textbf{filtration} over a finite set $X$ is a monotone function 
 $F_X: \pow(X)\longrightarrow (-\infty, M]$ that gives rise to a family of subsets {$K^\varepsilon:=F_{X}^{-1}((-\infty, \varepsilon])$}, 
 {$\varepsilon\leq M$}, and inclusions maps $\{K^{\varepsilon}\stackrel{k^{I}}{\longrightarrow}K^{\varepsilon'}\}$, 
 {$\varepsilon\leq\varepsilon'\leq M$}. Each $K^\varepsilon$ of $\mathbb{X}$ is a simplicial complex and each $\sigma\in K^\varepsilon$ with $\#\{\sigma\}=n+1$ is an $n$-simplex of $K^\varepsilon$ (see Appendix \ref{simpcomp}).

 A \textbf{filtered set} is a pair $\mathbb{X}=(X, F_X)$, where  $F_X$ is a filtration over $X$. Given $\mathbb{X}=(X, F_X)$ and $\mathbb{Y}=(Y, F_Y)$ filtered sets, a map $f: \mathbb{X}\longrightarrow \mathbb{Y}$ is a \textbf{filtration preserving map} if $F_{X}\geq F_{Y}\circ f$.

\begin{defi}
A \textbf{relative filtered set} $(\mathbb{X},\mathbb{A}):=((X,A),(F_X,F_A))$ is a pair $(X,A)$ of finite sets, where $A\subset X$,  equipped with a pair of maps $(F_X,F_A)$ defined in the following way: the maps $F_X: \pow(X)\longrightarrow\mathbb{R}$ and $F_A:\pow(A)\longrightarrow \mathbb{R}$ are filtrations over ${X}$ and over ${A}$, respectively, and they satisfy $F_A(\sigma)\geq F_X(\sigma)$ for all $\sigma\in \pow(A)$.
A \textbf{filtration preserving map of relative filtered sets} is any map $f: (\mathbb{X},\mathbb{A}) \longrightarrow (\mathbb{Y},\mathbb{B})$ such that 
    \begin{itemize}
        \item $f$ is a map of pairs $f:(X, A)\longrightarrow (Y, B)$, that is, $f:X\longrightarrow Y$ such that $f(A)\subset B$; 
        \item $f:(X, F_X) \longrightarrow (Y, F_Y)$ is a filtration preserving map and
        \item the restriction $f|_A:(A,F_A) \longrightarrow (B,F_B)$ is also a filtration preserving map.
\end{itemize}
 \end{defi}

\begin{remark}\label{Remark-1}
It is important to note that the word filtration in the previous definition is related to the fact that, from the finite set $X$ with the map $F_X$, the following filtrations of simplicial complexes are obtained 
$$
K^{\varepsilon} = F_X^{-1}((-\infty,\varepsilon]) \text{ and }A^{\varepsilon} = F_A^{-1}((- \infty, \varepsilon]).
$$
\end{remark}

\begin{defi}
Let $\mathbb{X} = (X, F_X)$ and $\mathbb{Y} = (Y, F_Y)$ be filtered sets. The \textbf{union} of {these} filtered sets is the filtered set  $\mathbb{X} \cup \mathbb{Y} = (X\cup Y, \mathcal{F})$, where $\mathcal{F}$ is given by
\[
\mathcal{F} (\sigma) :=  \begin{cases}

\min\{ F_X(\sigma) , F_Y(\sigma)\}  & \mbox{if } \sigma\subset X\cap Y \\

F_X(\sigma) & \mbox{if } \sigma\subset X\mbox{ and } \sigma\not\subset Y \\

F_Y(\sigma) & \mbox{if } \sigma\subset Y\mbox{ and } \sigma\not\subset X  \\

M & \mbox{otherwise}.
\end{cases} 
\]

The \textbf{intersection} of  filtered sets is the filtered set $\mathbb{X} \cap \mathbb{Y} = (X\cap Y, \mathcal{F'})$, where $\mathcal{F'}$ is given by
\[
\mathcal{F'} (\sigma) :=  
\max\{ F_X(\sigma) , F_Y(\sigma)\}.  
\]
\end{defi}

 For stating the next results, we need to define the functor persistent homology of relative filtered pairs with the respective filtration preserving homomorphism.

\begin{defi}\label{defhomrel}

For every $n\in\mathbb{Z}, \,\, n > 0$,
 and each $\varepsilon >0$ fixed we  have from the classical theory that $C_n^{\varepsilon}(\mathbb{X},\mathbb{A}) := C_n(K^\varepsilon, A^\varepsilon) =  {C_n(K^\varepsilon)}/{C_n(A^\varepsilon)}$ and 
the map
$\partial^{\varepsilon}_{n}:C_{n}(K^{\varepsilon})\longrightarrow C_{n-1}(K^\varepsilon)$ induces
$\partial^{\varepsilon}_{n}:C_{n}(K^{\varepsilon}, A^\varepsilon)\longrightarrow C_{n-1}(K^\varepsilon, A^\varepsilon).$ Then we obtain the following \textbf{relative chain complex}, denoted by $(C(K^{\varepsilon}, A^{\varepsilon}), \partial^{\varepsilon})$:
\[
\mathcal{C}^\varepsilon: \cdots C_{n}(K^{\varepsilon}, A^\varepsilon) 
\stackrel{\partial^{\varepsilon}_{n}}{\longrightarrow}  
C_{n-1}(K^{\varepsilon}, A^\varepsilon) 
\stackrel{\partial^{\varepsilon}_{n-1}}{\longrightarrow} \cdots
\]
\noindent Consider     
$Z^{\varepsilon}_{n}(\mathbb{X},\mathbb{A})  :={\ker} (\partial^{\varepsilon}_{n})$ and $B^{\varepsilon}_{n}(\mathbb{X},\mathbb{A}) : = \mathrm{Im} (\partial^{\varepsilon}_{n+1})$   the sets of  \textbf{relative $n$-cycles}  and \textbf{relative $n$-boundaries}, respectively.

For  $ I $ in the poset  $\mathrm{Int}(\mathbb{R})$ and $n\in\mathbb{Z}, \,\, n > 0$, the chain map $k^{I}_n:C_n(K^{\varepsilon})\longrightarrow C_n(K^{\varepsilon'})$, induced by the inclusion map $ k^{I}:{K^{\varepsilon}}\longrightarrow K^{\varepsilon'}$, also induces a chain map in the relative chain complex and for simplicity it also will  be denoted by $k^{I}_n$.

The \textbf{relative $n$-persistent homology group} of $(\mathbb{X},\mathbb{A})$ is defined as the quotient
\[
\mathrm{H}_{n}^I(\mathbb{X},\mathbb{A}):= 
\dfrac{{k}_{n}^{I}({Z}_{n}^{\varepsilon}(\mathbb{X},\mathbb{A}))}{{B}_{n}^{\varepsilon'}(\mathbb{X},\mathbb{A})\cap {k}_{n}^{I}({Z}_{n}^{\varepsilon}(\mathbb{X},\mathbb{A}))}  \cdot
\]

\end{defi}

Given a {filtration} preserving map of relative filtered sets
$f: (\mathbb{X},\mathbb{A}) \longrightarrow (\mathbb{Y},\mathbb{B})$  the \emph{relative chain map induced by $f$} is induced by 
$f_{\#}^{I}:
  k^{I}_{n}(Z_{n}^{\varepsilon}(\mathbb{X})) 
 \longrightarrow
 k^{I}_{n}(Z_{n}^{\varepsilon}(\mathbb{Y}))$
 previously defined and we will give the same notation
 $f_{\#}^{I}:
  k^{I}_{n}(Z_{n}^{\varepsilon}(\mathbb{X}, \mathbb{A})) 
 \longrightarrow
 k^{I}_{n}(Z_{n}^{\varepsilon}(\mathbb{Y}, \mathbb{B}))$.
This induces a natural homomorphism in relative persistent homology
$f^{I}_*:  
\mathrm{H}_{n}^{I}(\mathbb{X}, \mathbb{A})\longrightarrow  \mathrm{H}_{n}^{I}(\mathbb{Y},\mathbb{B})$ given by $f^{I}_*([d]) := [f^{\varepsilon'}_{\#}(d)], $ where $[d] \in \mathrm{H}_{n}^{I}(\mathbb{X}, \mathbb{A})$, for  $d\in k_n^{I}(Z_n^{\varepsilon}(\mathbb{X},\mathbb{A}))$.

\begin{defi}

The \textbf{persistent homology boundary homomorphism} 
$$ \partial:\mathrm{H}^{I}_n(\mathbb{X},\mathbb{A})\longrightarrow \mathrm{H}^{I}_{n-1}(\mathbb{A})$$ 

\noindent is the homomorphism induced by the boundary map from the group of $n$-chains of $(K^{\varepsilon'}, A^{\varepsilon'})$ to the group of $(n-1)$-chains of $A^{\varepsilon'}$.
\end{defi}

\begin{defi}\label{contiguous}
\label{seqcontiguous}
Let $(\mathbb{X},\mathbb{A})$ and $(\mathbb{Y},\mathbb{B})$ be relative filtered sets and let $f,g:(\mathbb{X},\mathbb{A})\longrightarrow(\mathbb{Y},\mathbb{B})$ be two filtration preserving maps. 
We  say that $f$ and $g$ are \textbf{contiguous maps in $I$} if for all $\alpha\in I$,
$f_\alpha,g_\alpha:(K^\alpha,L^\alpha)\longrightarrow(M^\alpha,N^\alpha)$
satisfy the property that for every $\sigma\in K^\alpha$, there is $\tau \in M^\alpha$ such that $f_\alpha(\sigma)\subset\tau$ and $g_\alpha(\sigma)\subset \tau$ and if $\sigma\in L^\alpha$, then the aforementioned $\tau$ must be a subset of $N^\alpha$, where  
$K^\alpha = F_X^{-1}((-\infty,\alpha]), \,
L^\alpha = F_A^{-1}((-\infty,\alpha]), \,
M^\alpha = F_Y^{-1}((-\infty,\alpha])$ and $
N^\alpha = F_B^{-1}((-\infty,\alpha]).$ We will say that $f$ and $g$ are \textbf{contiguous maps} if they are contiguous in $I = \mathbb{R}$.

\label{eqcontiguous}

The maps  $f$ and $g$ are \textbf{contiguously equivalent maps} if $g\circ f$ and $\mathop{id}_{(\mathbb{X},\mathbb{A})}$ are contiguous maps and also $f\circ g$ and  $\mathop{id}_{(\mathbb{Y},\mathbb{B})}$ are contiguous maps. In this case,  $(\mathbb{X},\mathbb{A})$ and $(\mathbb{Y},\mathbb{B})$ are \textbf{equi-contiguous sets} and the map $f$ (respect. $g$) is a \textbf{contiguity}.
\end{defi}

\section{Relative Persistent Homology Axioms and  Equivalence with  Simplified Axioms}\label{axioms-section}

In this section we introduce the relative persistent homology axioms and establish the equivalence with a set of simplified axioms.

\begin{theo} 
{
There is a unique pair $(\mathrm{H},\partial)$ such that $\mathrm{H}$ is a functor from the category $\mathscr{RF}$ (See Appendix \ref{rf}) of filtered relative sets to the category $\mathscr{A}b$ of abelian groups
and
$\partial$ is a natural transformation which reduces the degree by 1, $\partial^{I} : \mathrm{H}^I (\mathbb{X}, \mathbb{A}) \longrightarrow \mathrm{H}^I(\mathbb{A},\emptyset)$, that satisfies the following axioms:
}
\end{theo}

\begin{axiom1}\label{ax1} If $\mathop{id}: (\mathbb{X},\mathbb{A})\longrightarrow (\mathbb{X}, ,\mathbb{A})$ is the  filtration preserving identity  map, then   $ \mathop{id}^{I}_*: \mathrm{H}_{*}^{I}(\mathbb{X},\mathbb{A})\longrightarrow \mathrm{H}_{*}^{I}(\mathbb{X},\mathbb{A})$ is the identity homomorphism.
\end{axiom1}

\begin{axiom1}\label{ax2} Let $f: (\mathbb{X}, \mathbb{A})\longrightarrow (\mathbb{Y},\mathbb{B})$ and $g:(\mathbb{Z}, \mathbb{C})\longrightarrow (\mathbb{X}, \mathbb{A})$ be filtration preserving maps. Then  $(f\circ g)^{I}_*=f^{I}_*\circ g^{I}_*$. In other words, the following diagram  commutes:

\[
\xymatrix
{
	\mathrm{H}_{*}^I(\mathbb{Z},\mathbb{C})\ar[dr]^{g^{I}_*} \ar[rr]^{(f\circ g)^{I}_*} &  & \mathrm{H}_{*}^I(\mathbb{Y},\mathbb{B})  \\
	 & \mathrm{H}_{*}^I(\mathbb{X},\mathbb{A}) \ar[ur]^{f^{I}_*} &  \\
}
\]
\end{axiom1}

\begin{axiom1}\label{ax3} Let $\mathbb{(X,A)}$ and  $\mathbb{(Y,B)}$ be relative filtered sets and let $f: \mathbb{(X,A)}\longrightarrow \mathbb{(Y,B)}$ be  a filtration preserving map. Then the following diagram commutes:
\[
\xymatrix
{
	\mathrm{H}_{n}^I(\mathbb{X},\mathbb{A})     \ar[d]_{\partial^{I}_{A}} \ar[rr]^{f^{I}_*} & & \mathrm{H}_{n}^I(\mathbb{Y},\mathbb{B}) \ar[d]^{\partial^{I}_{B}}  \\
	\mathrm{H}_{n-1}^I(\mathbb{A}) \ar[rr]^{(f|_{A})^I_*} & &  \mathrm{H}_{n-1}^I(\mathbb{B})
}
\]
\end{axiom1}

\begin{axiom1}\label{ax4}\label{seq-axiom}
Let $\mathbb{(X,A)}$  be a relative filtered set and  
 $i: \mathbb{A}\longrightarrow \mathbb{X}$ and $j :\mathbb{X}\longrightarrow (\mathbb{X}, \mathbb{A})$   inclusion maps. The following sequence is exact:
$$ \cdots\longrightarrow \mathrm{H}_{n}^I(\mathbb{A})\stackrel{i^{I}_*}{\longrightarrow} \mathrm{H}_{n}^I(\mathbb{X})\stackrel{j^{I}_*}{\longrightarrow} \mathrm{H}_{n}^I(\mathbb{X},\mathbb{A})\stackrel{\partial^{I}_{A}}{\longrightarrow} \mathrm{H}_{n-1}^I(\mathbb{A})\longrightarrow\cdots $$
\end{axiom1}

\begin{axiom1} \label{ax5}
Let $f,g:(\mathbb{X}, \mathbb{A})\longrightarrow (\mathbb{Y},\mathbb{B})$ be filtration preserving maps. If $f$ and $g$ are contiguous maps 
then $f_{*}^{I}=g_{*}^{I}$.
\end{axiom1}

\begin{axiom1}(Dimension axiom)\label{ax6}
Let $\mathbb{X}=(X,F_X)$ be a filtered set, where  $X$  a one-point set and  $F_{X}(X) := \alpha$ for some $\alpha \in \mathbb{R}$. Then

$\mathrm{H}^I_{n}(\mathbb{X}) = \begin{cases}

\mathbb{F}, & \mbox{if } n=0 \text{ and } \varepsilon\geq\alpha; \\

\mathbf{0},& \text{ otherwise. }

\end{cases}$
\end{axiom1}

\begin{axiom1}(Excision axiom)\label{exision axiom}\label{ax7}
Let $(\mathbb{X},\mathbb{A})$  and $(\mathbb{X}',\mathbb{A}')$ be relative filtered sets
 such that $ \mathbb{X} = \mathbb{X'}\cup \mathbb{A} \text{ and } \mathbb{A'} = \mathbb{X'\cap A}$.
 Assume that the inclusion map $i: (\mathbb{X}',\mathbb{A}')\longrightarrow (\mathbb{X},\mathbb{A})$ is a filtration preserving map. Then the induced map:
\[ i_*^{I}: \mathrm{H}^I_n(\mathbb{X}',\mathbb{A}')\longrightarrow \mathrm{H}^I_n(\mathbb{X},\mathbb{A})\]

\noindent is   an isomorphism. 
\end{axiom1}

 Although it is not stated in the following theorem \emph{itself}, what is established in its proof is to show that the simplified version of the \cref{as1,as2,as3} implies \cref{ax1,ax2,ax3,ax4,ax5,ax6,ax7}. Therefore, the equivalence between the two persistent homology theories follows.

\begin{theo} \label{thrm-equivalence}
There is a unique pair $(\mathrm{H},\partial)$ such that $\mathrm{H}$ is a functor from the category $\mathscr{RF}$ (See Appendix \ref{rf}) of filtered relative sets to the category $\mathscr{A}b$ of abelian groups and $\partial$ is a natural transformation which reduces the degree by $1$, $\partial^I: \mathrm{H}^I(\mathbb{X},\mathbb{A})\to \mathrm{H}^I(\mathbb{A},\emptyset)$ that satisfies the following axioms:

\begin{axioms} \label{as1}
The inclusion map
$
(\mathbb{X}, \mathbb{X}\cap \mathbb{Y}) 
 \stackrel{I}{\hookrightarrow}  
  (\mathbb{X}\cup \mathbb{Y}, \mathbb{Y})
$
 induces an isomorphism in persistent homology.   
\end{axioms}

\begin{axioms}\label{as2}
    The following triangle is exact,  where $i:\mathbb{A}\hookrightarrow \mathbb{X}$ and $j:(\mathbb{X},\emptyset)\hookrightarrow (\mathbb{X},\mathbb{A})$ are inclusion maps:
\[
\xymatrix
{
	\mathrm{H}^I(\mathbb{X})\ar[dr]^{j_{\ast}}  &  & \mathrm{H}^I(\mathbb{A})\ar[ll]_{i_{\ast}}  \\
	 & \mathrm{H}^I(\mathbb{X},\mathbb{A}) \ar[ur]^{\partial} &  
} 
\]
\end{axioms}

\begin{axioms}
     \label{as3}
For every $(q,\alpha)$-simplex $\mathbb{S}^q_\alpha$ (see \cref{def-simp-bord}), 
\[
\mathrm{H}_{k}^I(\mathbb{S}^q_\alpha) = 
\begin{cases}

\mathbb{Z}\, & \mbox{if } k=0 \mbox{ and } \varepsilon \geq \alpha \\

\mathbf{0},& \mbox{otherwise. } 

\end{cases}
\] 
\end{axioms}

\end{theo}
\smallskip

We will show that \cref{ax1,ax2,ax3,ax4,ax5,ax6,ax7} for persistent homology can be obtained from the simplified axioms.
In pure theoretical sense, the relevance of this equivalence is to show that \cref{ax5} is a consequence of the other axioms. Another point of the simplified axioms for a persistent homology theory  is that: given a candidate $\mathrm{H}^{I}$ for a persistent homology theory, it is sufficient to check whether it satisfies the simplified ones.

\begin{theo}\label{theoequiv}
    \cref{ax1,ax2,ax3,ax4,ax5,ax6,ax7} are equivalent to \cref{as1,as2,as3}.
\end{theo}

\smallskip

For the proof of this theorem the following results are required:

\smallskip

\begin{lemma} \cref{as1,as2} imply 
  that the Mayer-Vietoris sequence
\[
\cdots \rightarrow 
\mathrm{H}_n^I (\mathbb{X}_1\cap \mathbb{X}_2)
\rightarrow  
\mathrm{H}_n^I (\mathbb{X}_1) \oplus \mathrm{H}_n^I(\mathbb{X}_2) 
\rightarrow 
\mathrm{H}_n^I(\mathbb{X}_1\cup \mathbb{X}_2) 
\rightarrow 
\mathrm{H}_{n-1}^I(\mathbb{X}_1\cap \mathbb{X}_2) 
\rightarrow \cdots 
\]
is exact for any two filtered sets  $\mathbb{X}_1$ and $\mathbb{X}_2$.
\end{lemma}

\proof 
{\cref{as1,as2} imply that $(\mathbb{X}_1\cup \mathbb{X}_2; \mathbb{X}_1, \mathbb{X}_2)$ is a proper triad relative to the persistent homology $\mathrm{H}_\ast^I$. Then by \cite[I.15.3]{foundations} the Mayer-Vietoris sequence is exact.}
\qed

\bigskip

Recall that a simplicial complex $K$ is star shaped with respect to a vertex $v$ if every simplex of $K$ is a face of another simplex which has $v$ as one of its vertices.

\begin{defi}
We will say that a filtered set $\mathbb{X}$ is \emph{star shaped} at $a\in X$ in $I$ if for each $\alpha\in I$, $K_\alpha$ is {either} star shaped at $a$ in the {standard} sense or it is the empty set.
\end{defi}

\begin{defi}
Let $\mathbb{X}= (X, F_X)$ be a filtered set. 
A \emph{filtered subset} of $\mathbb{X}$ is a filtered set $\mathbb{Y}= (Y, F_Y)$,
with $Y\subset X$
 and $F_Y(\sigma)\geq F_X(\sigma)$ for all $\sigma \subset Y$.
A filtered subset $\mathbb{Y}$ is said \emph{proper} in  $I$ if $\mathbb{X}$ contains a simplex in  $I$ that does not occur on $\mathbb{Y}$ in $I$.

Let $\mathbb{X} = (X, F_X)$ be a filtered set and let $a\in X$. In this context, the filtered set (filtered subset) $(\{a\}, F|_{a})$ {will be denoted by
$\mathbb{Q}_{a}$}.
\end{defi}

\begin{lemma}\label{lemma2.2}
 \cref{as1,as2,as3} imply that the inclusion of a vertex {$\mathbb{Q}_{a}$} into a filtered set $\mathbb{X}$ which is star shaped at $a$ in $I$ induces an isomorphism in persistent homology
\[
\mathrm{H}^I(\mathbb{Q}_a)\longrightarrow \mathrm{H}^I(\mathbb{X}).
\]
\end{lemma}

\begin{proof}

Let $\mathbb{X}$ be a filtered set that is star shaped at $a$ in $I$.

The lemma is trivially true for $X= \{a\}$.

The filtered subsets of $\mathbb{X}$ that are star shaped at $a$ in $I$ have a partial order by inclusion.
We shall show inductively that the lemma is true for $\mathbb{X}$ if it is true for every proper filtered subset of $\mathbb{X}$.

Notice that 
$\mathbb{X}$ may be expressed as the union of a simplex $\mathbb{S}_{\alpha}^{q}$ and a proper filtered subset $\mathbb{Y}\subset \mathbb{X}$ which is star shaped at $a$ in $I$.

Let $\sigma'= \mathbb{S}_{\alpha}^{q}\cup {\mathbb{Q}_{a}}$. Later we will give an explanation of why $\sigma'$ is defined and it will be given a proof for the assumed isomorphism $\mathrm{H}_n^I(\sigma')\equiv \mathrm{H}_n^I(\{a\})$ {(Remark \ref{remark})}.

\smallskip

Let $i:\mathbb{Y}\cap \sigma' \hookrightarrow \mathbb{Y}$, $j:\mathbb{Y}\cap\sigma'\hookrightarrow \sigma'$ and $k:{\mathbb{Q}_{a}}\hookrightarrow \mathbb{Y}\cap \sigma'$ be inclusion maps.

\smallskip

If $F_X(a) > \varepsilon$, in Mayer-Vietoris sequence of $\mathbb{X} = \mathbb{Y}\cup \sigma'$ we have by hypothesis that 
$\mathrm{H}_n^I(\mathbb{Y}) 
= \mathrm{H}_n^I(\sigma')
=\mathrm{H}_n^I(\mathbb{Y}\cap \sigma') 
= \mathrm{H}_n^I(\{a\}) 
=0$ for all $n$.
This implies that $\mathrm{H}_n^I (\mathbb{X}) = 0$ for all $n$. Thus  $i_\ast:\mathrm{H}_n^I(\{a\})\to \mathrm{H}_n^I (\mathbb{X})$ is an isomorphism for all $n$.

Then let us assume that $F_X(a)\leq \varepsilon$.
 By hypothesis, $k_\ast$ and $(i\circ k)_\ast$ are isomorphisms; hence $i_\ast$ is also an isomorphism.

 The constant map $t:\sigma'\to {\mathbb{Q}_{a}}$ is a left inverse to $j\circ k$; thus, $(j\circ k)_\ast:\mathbb{Z}\to \mathbb{Z}$ is a split monomorphism and so it is an isomorphism. Therefore, $j_\ast$ is also an isomorphism.

Then in Mayer-Vietoris sequence for $\mathbb{X}=\mathbb{Y}\cup \sigma'$: 
\begin{alignat*}{2}
\mathrm{H}_1^I(\mathbb{Y})\oplus \mathrm{H}_1^I(\sigma')
    &  \longrightarrow 
    \mathrm{H}_1^I(\mathbb{X}) 
     \stackrel{\alpha}{\longrightarrow}
     \mathrm{H}_0^I(\mathbb{Y}\cap \sigma') 
        \stackrel{i_*\oplus j_\ast}{\longrightarrow}  
	  &\quad
    \\        &
        \stackrel{i_*\oplus j_\ast}{\longrightarrow}
         \mathrm{H}_0^I(\mathbb{Y}) \oplus \mathrm{H}_0^I(\sigma')
         \stackrel{\beta}{\longrightarrow}
         \mathrm{H}_0^I(\mathbb{X})    
       &  
       \longrightarrow 0
\end{alignat*}
we deduce that $\mathrm{Im}(\alpha) = 0$ and $\mathrm{Im}(\beta)=\mathbb{Z}$,
since $i_\ast $ and $j_\ast$ are isomorphisms, $i_\ast\oplus j_\ast$ is injective.

Therefore $\mathrm{H}_1^I(\mathbb{X}) = 0$ and $\mathrm{H}_0^{I}(\mathbb{X}) = \mathbb{Z}$.

\bigskip

For $n>1$, by using the Mayer-Vietoris sequence, one has $\mathrm{H}_n^I(\mathbb{X}) =0{ = \mathrm{H}^{I}_n(\mathbb{Q}_a)}$.

\end{proof}

\begin{remark}\label{remark}
{In the proof of Lemma \ref{lemma2.2} }we needed to define $\sigma'$ to avoid defining maps on empty sets. If $\mathbb{S}_\alpha^{q}$ is not empty in   the interval $I$, then $\sigma'$ is not necessary. 
\end{remark}

We also need to show that the inclusion map 
$i:{\mathbb{Q}_{a}}\to \mathbb{S}_{\alpha}^{q}\cup {\mathbb{Q}_{a}}$ induces an isomorphism in persistent homology in $I$.

Using \cref{as2} on the pair $(\mathbb{S}_{\alpha}^{q}, \mathbb{S}_{\alpha}^{q}\cap{\mathbb{Q}_{a}})$, and remembering that the simplex $\mathbb{S}_\alpha^{q}$ is empty in some sub-interval of $I$, we have
that both ends of the sequence bellow are null for all $n$, that is $\mathrm{H}_n^I(\mathbb{S}_{\alpha}^{q}, \mathbb{S}_{\alpha}^{q}\cap \mathbb{Q}_{a})=0$.
\[
\begin{tikzcd}[row sep=1ex]
	\mathrm{H}_n^I(\mathbb{S}_{\alpha}^{q}) \arrow[r] 
		& \mathrm{H}_n^I(\mathbb{S}_{\alpha}^{q}, \mathbb{S}_{\alpha}^{q}\cap \mathbb{Q}_{a}) \arrow[r] 
		& \mathrm{H}_{n-1}^I(\mathbb{S}_{\alpha}^{q}\cap \mathbb{Q}_{a}) .
\end{tikzcd} 
\]

By \cref{as1}, for all $n$,
\[
\mathrm{H}_n^I(\mathbb{S}_{\alpha}^{q}\cup \mathbb{Q}_{a}, \mathbb{Q}_{a}) = \mathrm{H}_n^I(\mathbb{S}_{\alpha}^{q}, \mathbb{S}_{\alpha}^{q}\cap \mathbb{Q}_{a}) = 0  .
\]

Using \cref{as2} in the pair $(\mathbb{S}_{\alpha}^{q}\cup {\mathbb{Q}_{a}}, {\mathbb{Q}_{a}})$ 
with $n\geq 1$ we also have that
$\mathrm{H}_n^I(\mathbb{S}_{\alpha}^{q}\cup \mathbb{Q}_{a}) = 0$.

Then for $n=0$, the following sequence implies that
$  \mathrm{H}_0^I(\mathbb{Q}_{a}) $ is isomorphic to {${\mathrm{H}_{0}^I(\mathbb{S}_{\alpha}^{q}\cup \mathbb{Q}_{a})}$}
\[
\begin{tikzcd}[row sep=1ex]
	\mathrm{H}_{1}^I(\mathbb{S}_{\alpha}^{q}\cup \mathbb{Q}_{a}, \mathbb{Q}_{a}) 
 \arrow[r] 
		& 
  \mathrm{H}_0^I(\mathbb{Q}_{a}) 
  \arrow[r] 
		& 
  \mathrm{H}_{0}^I(\mathbb{S}_{\alpha}^{q}\cup \mathbb{Q}_{a}\mathbb{Q}_{a}) 
  \arrow[r] 
		& 
  \mathrm{H}_{0}^I(\mathbb{S}_{\alpha}^{q}\cup \mathbb{Q}_{a}, \mathbb{Q}_{a}) 
  .
\end{tikzcd} 
\]

With this, we conclude that the inclusion map 
$i:{\mathbb{Q}_{a}}\to \mathbb{S}_{\alpha}^{q}\cup{\mathbb{Q}_{a}}$ induces an isomorphism in persistent homology in $I$.

\begin{lemma}\label{lemma2.3}
Let $\mathbb{X}$  be a filtered set and $\mathbb{A}$ a star shaped filtered subset of $\mathbb{X}$ in $I$.
Then
$\partial = 0$ in the {diagram} of \cref{as2} for the pair $(\mathbb{X}, \mathbb{A})$, and the resulting short exact sequence
\[
\begin{tikzcd}
	0 \arrow[r]  
		& \mathrm{H}_n^I(\mathbb{A})\arrow[r] 
		& \mathrm{H}_n^I(\mathbb{X}) \arrow[r] 
		& \mathrm{H}_n^I(\mathbb{X}, \mathbb{A})\arrow[r] 
		& 0
\end{tikzcd}
\]
splits for every $n$.    
\end{lemma}

\begin{proof}

For $n>1$, by Lemma \ref{lemma2.2} and \cref{as2} we have the short exact sequence
\[
\begin{tikzcd}[row sep=1ex]
	\mathrm{H}_{n}^I(\mathbb{A}) \arrow[r] 
		& \mathrm{H}_n^I(\mathbb{X}) \arrow[r] 
		& \mathrm{H}_{n}^I(\mathbb{X}, \mathbb{A}) \arrow[r] 
		& \mathrm{H}_{n-1}^I(\mathbb{A}) 
\end{tikzcd}  
\]
with $\mathrm{H}_{n}^I(\mathbb{A})=0$ and $\mathrm{H}_{n-1}^I(\mathbb{A})=0$, and the result is immediate.
\smallskip

For the remaining cases, consider the exact sequence
\[
\begin{tikzcd}[column sep=4ex , row sep=1ex]
	0 \arrow[r, "i_\ast"] 
		& \mathrm{H}_1^I(\mathbb{X}) \arrow[r, "j_\ast"] 
		& \mathrm{H}_{1}^I(\mathbb{X}, \mathbb{A}) \arrow[r, "\partial"] 
		& \mathrm{H}_{0}^I(\mathbb{A}) \arrow[r, "i_\ast"] 
		& \mathrm{H}_{0}^I(\mathbb{X})\arrow[r, "j_\ast"] 
		& \mathrm{H}_{0}^I(\mathbb{X}, \mathbb{A})\arrow[r, "\partial"] 
		& 0
\end{tikzcd}  
\]
where $\mathrm{H}_{0}^I(\mathbb{A})\cong \mathbb{Z}$.

Let $a$ be the point about which $\mathbb{A}$ is star shaped, with inclusion map 
$k:\{a\}\to \mathbb{A}$. By Lemma \ref{lemma2.2}, $k_\ast$ is an isomorphism. 

Let $t:\mathbb{X}\to \{a\}$ be the collapsing map. Observe that $t$ is a left inverse to $i\circ k$, so $(i\circ k)_\ast:\mathrm{H}_0^I(\{a\})\to \mathrm{H}_0^I(\mathbb{X})$ and $i_\ast : \mathrm{H}_0^I(\mathbb{A})\to \mathrm{H}_0^I(\mathbb{X})$ are monomorphisms.

Therefore, $\partial: \mathrm{H}_1^I(\mathbb{X}, \mathbb{A})\to \mathrm{H}_0^I(\mathbb{A})$ is the zero map, providing the lemma for $n=1$ and $n=0$.

\end{proof}

\begin{proof}[Proof of Theorem\ref{theoequiv}]\

 \cref{ax1,ax2,ax3}  follow by the requiring that $\mathrm{H}^{I}$ is a functor and that $\partial$ is a natural transformation.
Axiom \ref{ax4} and \cref{as2}  are the same,
Axiom \ref{ax6} is a particular case of \cref{as3}, when $X$ is a single point set, and
Axiom \ref{ax7} is equivalent to \cref{as1}.

It remains to prove \cref{ax5}, that is, if  $f, g:\mathbb{X}\to \mathbb{Y}$ are contiguous maps in $I$, then $f_\ast^{I} = g_\ast^{I}$.

To prove this result, we first construct  a structure  analogous to a cylinder one in complexes.

Let $\mathbb{X} = (X, F_X)$ be a filtered set and assume that the vertices are ordered. The new filtered set 
$\mathbb{X}^{cyl}= (X^{cyl}, F_X^{cyl})$ is defined as follows:

The finite set $X^{cyl} = X \sqcup X'$ is the disjoint union of two copies of $X$ with the dictionary order.
If  $\sigma_i = (s_{i,1}, \ldots , s_{i,j}, s_{i,j+1}, \ldots , s_{i,k})$ is a subset of $X$ whose vertices appear on the given order, let $\sigma_{i,j}\subset X^{cyl}$ 
denoting $\sigma_{ij} = (s_{i,1}', \ldots , s_{i,j}', s_{i,j}, s_{i,j+1}, \ldots , s_{i,k})$, and
define
$F_X^{cyl}(\sigma_{ij}) :=  F_X(\sigma_i)$. 
If $\sigma\subset X^{cyl}$, $\sigma = (s_1', \ldots , s_i', s_{i+1} , \ldots , s_k )$, is a subset of some $\sigma_{ij}$, $s_i\neq s_{i+1}$, we define 
$
F_X^{cyl}(\sigma) :=  F_X(\tau) 
$,
 where $\tau = (s_1,\ldots,s_i, s_{i+1},\ldots, s_k)$. Notice that it is the lower value possible where there is a $\sigma_{i,j}$ containing $\sigma$.

On the other simplices not addressed previously, let $F_X^{cyl}(\sigma): = M$.

Let $h_0: s  \mapsto s $ be the inclusion map $X\subset X^{cyl}$, and let $h_1: s\mapsto s'$ be the inclusion map on the other copy of $X$; we will show that $h_0$ and $h_1$ induce the same map in persistence homology.

\smallskip

Assume that the subsets $\sigma_i\subset X $  are ordered in a way that all subsets of order $n$ come before all subsets of order $n+1$. We will construct $\mathbb{X}^{cyl}$ from $\mathbb{X}$ by adding all the subsets $\sigma_{i,j}$, in lexicographic
order;
that is, $\mathbb{X}_0 = \mathbb{X}$, $\mathbb{X}_1 = \mathbb{X}_0\cup \sigma_{1,1}$, $\mathbb{X}_2 = \mathbb{X}_1\cup \sigma_{2,1}$ and so on.

\smallskip

At each stage in the construction, if $\mathbb{X}_{k+1} = \mathbb{X}_k\cup \sigma_{ij}$, then $\mathbb{X}_k\cap \sigma_{i,j}$ is the union (of filtered complexes) of the faces of $\sigma_{ij}$ containing the vertex $s_{i,j}$. This is star shaped at $s_{ij}$.

\smallskip
Now let us consider the Mayer-Vietoris sequence of $\mathbb{X}_k\cup \sigma_{ij}$. Observe that for
 $n>1$, by Lemma \ref{lemma2.2} we have that 
$\mathrm{H}_{n}^I(\mathbb{X}_k\cap\sigma_{ij})=0$, 
$\mathrm{H}_n^I(\sigma_{ij}) = 0$ and  
$\mathrm{H}_{n-1}^I(\mathbb{X}_k\cap\sigma_{ij}) = 0$.
\[
\begin{tikzcd}[column sep=4ex , row sep=1ex]
	\mathrm{H}_{n}^I(\mathbb{X}_k\cap\sigma_{ij})  \arrow[r] 
		& \mathrm{H}_n^I(\mathbb{X}_k) \oplus \mathrm{H}_n^I(\sigma_{ij})  \arrow[r] 
		& \mathrm{H}_n^I(\mathbb{X}_k\cup\sigma_{ij})         \arrow[r] 
    & \mathrm{H}_{n-1}^I(\mathbb{X}_k\cap\sigma_{ij})
\end{tikzcd}  .
\]

Therefore the inclusion map $\mathbb{X}_k\subset \mathbb{X}_k\cup \sigma_{ij}$ induces an isomorphism.
\smallskip

For $n=1$, we have that 
$\mathrm{H}_{1}^I(\mathbb{X}_k\cap\sigma_{ij})=0$,
$\mathrm{H}_1^I(\sigma_{ij}) =0$,
$\mathrm{H}_{0}^I(\mathbb{X}_k\cap\sigma_{ij}) 
  \cong \mathbb{Z}$ and
$\mathrm{H}_0^I(\sigma_{ij})\cong \mathbb{Z}$.
\begin{multline*}
\mathrm{H}_{1}^I(\mathbb{X}_k\cap\sigma_{ij}) 
\longrightarrow 
\mathrm{H}_1^I(\mathbb{X}_k) \oplus  \mathrm{H}_1^I(\sigma_{ij}) 
\longrightarrow 
\mathrm{H}_1^I(\mathbb{X}_k\cup\sigma_{ij})
\longrightarrow 
\mathrm{H}_{0}^I(\mathbb{X}_k\cap\sigma_{ij}) 
\longrightarrow  
\\
\longrightarrow
 \mathrm{H}_0^I(\mathbb{X}_k) \oplus  \mathrm{H}_0^I(\sigma_{ij})
\longrightarrow
 \mathrm{H}_0^I(\mathbb{X}_k\cup\sigma_{ij})
\longrightarrow
0 .
\end{multline*}

Since $\mathrm{H}_0^I(\mathbb{X}_k\cap\sigma_{ij})\to \mathrm{H}_0^I(\sigma_{ij})$ is an isomorphism,  
$\mathrm{H}_1^I(\mathbb{X}_k\cup \sigma_{ij})\to \mathrm{H}_0^I(\mathbb{X}_k\cap\sigma_{ij})$ is the null homomorphism. Then 
$\mathrm{H}_1^I(\mathbb{X}_k)\to \mathrm{H}_1^I(\mathbb{X}_k\cup \sigma_{ij})$ is an isomorphism.
\smallskip

Applying Lemma \ref{lemma2.3} for the pairs $(\mathbb{X}_k\cup\sigma_{ij}, \sigma_{ij})$ and $(\mathbb{X}_k, \mathbb{X}_k\cap \sigma_{ij})$ and considering the natural inclusions we have the following commutative diagram:
\[
\begin{tikzcd}[column sep=4ex , row sep=4ex]
 0  \arrow[r] 
    & \mathrm{H}_n^I(\mathbb{X}_k\cap \sigma_{ij})    
        \arrow[r] \arrow[d, "(i_1)_\ast"]
    & \mathrm{H}_n^I(\mathbb{X}_k)                
        \arrow[r] \arrow[d, "(i_2)_\ast"]
    & \mathrm{H}_{n}^I(\mathbb{X}_k,\mathbb{X}_k\cap \sigma_{ij}) 
        \arrow[r]\arrow[d, "(i_3)_\ast"]
    & 0
			\\
 0  \arrow[r] 
    & \mathrm{H}_n^I(\sigma_{ij})
        \arrow[r] 
    & \mathrm{H}_n^I(\mathbb{X}_k\cup\sigma_{ij})
        \arrow[r] 
    & \mathrm{H}_{n}^I(\mathbb{X}_k\cup\sigma_{ij}, \sigma_{ij}) 
        \arrow[r] 
    & 0
\end{tikzcd}  
\]

By Lemma \ref{lemma2.2} $(i_1)_\ast$ is an isomorphism, and by \cref{as1} $(i_3)_\ast$ is also an isomorphism. Then by the Five Lemma, $(i_2)_\ast$ is an isomorphism.

Thus the inclusion map $\mathbb{X}_k\subset \mathbb{X}_{k+1}$ induces an isomorphism in persistence homology at each stage. Then the composition $h_0$ of all inclusion maps induces an isomorphism in persistence homology.
An analogous construction shows that $h_1$ also induces isomorphism in persistence {homology.}

\smallskip

By using the long exact sequences of $(\mathbb{X}, \mathbb{A})$ and $(\mathbb{X}^{cyl}, \mathbb{A}^{cyl})$, we have the following commutative diagram:
\[
\begin{tikzcd}[column sep=2ex , row sep=4ex]
 \quad  \arrow[r] 
    & \mathrm{H}_n^I(\mathbb{A})                       
        \arrow[r] \arrow[d, "(h_i)_\ast"]
    & \mathrm{H}_n^I(\mathbb{X})                
        \arrow[r] \arrow[d, "(h_i)_\ast"]
    & \mathrm{H}_{n}^I(\mathbb{X}, \mathbb{A}) 
        \arrow[r]\arrow[d,"(h_i)_\ast"]
    & \mathrm{H}_{n-1}^I(\mathbb{A}) 
		\arrow[r]\arrow[d, "(h_i)_\ast"]
		& \mathrm{H}_{n-1}^I(\mathbb{X}) 
            \arrow[r]\arrow[d, "(h_i)_\ast"]
    & \qquad
 \\
  \quad  \arrow[r] 
    & \mathrm{H}_n^I(\mathbb{A}^{cyl})                       
				\arrow[r]  
    & \mathrm{H}_n^I(\mathbb{X}^{cyl})                
        \arrow[r]  
    & \mathrm{H}_{n}^I(\mathbb{X}^{cyl}, \mathbb{A}^{cyl}) 
        \arrow[r] 
    & \mathrm{H}_{n-1}^I(\mathbb{A}^{cyl}) 
				\arrow[r]
		& \mathrm{H}_{n-1}^I(\mathbb{X}^{cyl}) 
         \arrow[r]
    & \qquad
\end{tikzcd}  
\]

By the Five Lemma we have that $(h_i)_\ast:\mathrm{H}_n^I(\mathbb{X}, \mathbb{A})\to \mathrm{H}_n^I(\mathbb{X}^{cyl}, \mathbb{A}^{cyl})$ is an isomorphism for all $n$, $i=0,1$.

\smallskip

Let $k:(\mathbb{X}^{cyl}, \mathbb{A}^{cyl})\to (\mathbb{X}, \mathbb{A})$ be defined by $k(s_{ij}) := k( s_{ij}')= s_{ij}$.
It is a filtration preserving  map and left inverse to both $h_0$ and $h_1$. Then 
\[
(h_0)_\ast = k_\ast^{-1} = (h_1)_\ast .
\]

Given two contiguous maps $f, g:(\mathbb{X}, \mathbb{A})\to (\mathbb{Y}, \mathbb{B})$, let $G:(\mathbb{X}^{cyl}, \mathbb{A}^{cyl})\to (\mathbb{Y}, \mathbb{B})$ be defined by
$G(s_{ij}) := f(s_{ij})$, $G(s_{ij}') := g(s_{ij})$.
Notice that $G$ is a well defined filtration preserving map because $f$ and $g$ are both contiguous.

Now $f = Gh_0$ and $g = Gh_1$. Then $(f)_\ast  = G_\ast (h_0)_\ast = G_\ast (h_1)_\ast = (g)_\ast$.

\end{proof}

\section{The Uniqueness Theorem}\label{sec:uniq}

In this section, we present results involving maps between persistent homology sequences and introduce the notion of coefficient group of a persistent homology theory.

\subsection{The Main Isomorphism}

Let $\mathrm{H}^I$ be a persistent homology theory. We will define new groups of $q$-chains and $q$-dimensional persistent homology.

\begin{defi}
 Let $(\mathbb{X},\mathbb{A})$ be a relative filtered set. The group of  $q$-chains of $(\mathbb{X}, \mathbb{A})$, denoted by $\mathcal{C}^{I}_q(\mathbb{X}, {A})$, is defined by 
\[
 \mathcal{C}^{I}_q(\mathbb{X}, \mathbb{A}):=\mathrm{H}^I_q
 (\mathbb{X}^{(q)}\cup \mathbb{A}, \mathbb{X}^{(q-1)}\cup \mathbb{A}).
\]

 The kernel of $\partial_q: \mathcal{C}^{I}_q(\mathbb{X}, \mathbb{A})\longrightarrow \mathcal{C}^{I}_{q-1}(\mathbb{X}, \mathbb{A})$ is called \textbf{the group of $q$-cycles of $(\mathbb{X}, \mathbb{A})$} and it is denoted by $\mathcal{Z}^{I}_q(\mathbb{X}, \mathbb{A})$. 
 The image of $\partial_{q+1}: \mathcal{C}^{I}_{q+1}(\mathbb{X}, \mathbb{A})\longrightarrow \mathcal{C}^{I}_{q}(\mathbb{X}, \mathbb{A})$ is called \textbf{the group of $q$-boundaries of $(\mathbb{X}, \mathbb{A})$ } and it is denoted by $\mathcal{B}^{I}_q(\mathbb{X}, \mathbb{A})$. 
 Since $\partial_q\circ\partial_{q+1}=0$, $\mathcal{B}^{I}_q(\mathbb{X}, \mathbb{A})$ is a subgroup of $\mathcal{Z}^{I}_q(\mathbb{X}, \mathbb{A})$.

The factor group  $\mathcal{H}^I_q(\mathbb{X}, \mathbb{A}):=\dfrac{\mathcal{Z}^{I}_q(\mathbb{X}, \mathbb{A})}{\mathcal{B}^{I}_q(\mathbb{X}, \mathbb{A})}$ is called the \textbf{relative $q$-dimensional persistent homology of $(\mathbb{X}, \mathbb{A})$} associated to $I$.

If $f:(\mathbb{X}, \mathbb{A})\longrightarrow (\mathbb{X}', \mathbb{A}')$ is a filtration preserving map, then $f^{I}_q$ maps $\mathcal{Z}^{I}_q(\mathbb{X}, \mathbb{A})$ into $\mathcal{Z}^{I}_q(\mathbb{X}', \mathbb{A}')$ and $\mathcal{B}^{I}_q(\mathbb{X}, \mathbb{A})$ into $\mathcal{B}^{I}_q(\mathbb{X}', \mathbb{A}')$, thereby inducing homomorphisms
$$ f^{I}_*: \mathcal{H}^I_q(\mathbb{X}, \mathbb{A})\longrightarrow \mathcal{H}^I_{q}(\mathbb{X}', \mathbb{A}')$$
\end{defi}

\begin{lemma}\label{lemma:homo-induced}

Let $j_q:\mathcal{C}^{I}_q(\mathbb{X}) \longrightarrow \mathcal{C}^{I}_q(\mathbb{X}, \mathbb{A})$ be the {chain map induced by the inclusion}
 $\mathbb{X}\stackrel{j}{\hookrightarrow}(\mathbb{X},\mathbb{A})$.

If we define 
$\overline{\mathcal{Z}}_q^{I}(\mathbb{X}, \mathbb{A}):=j_q^{-1}[\mathcal{Z}^{I}_q(\mathbb{X}, \mathbb{A})]$,
 $\overline{\mathcal{B}}_q^{I} (\mathbb{X}, \mathbb{A}):=j_q^{-1}[\mathcal{B}^{I}_q(\mathbb{X}, \mathbb{A})]$ 
 and $\overline{\mathcal{H}}^I_q(\mathbb{X}, \mathbb{A}):= \dfrac{\overline{\mathcal{Z}}^{I}_q(\mathbb{X}, \mathbb{A})}{\overline{\mathcal{B}}^{I}_q(\mathbb{X}, \mathbb{A})}$, then
$$\overline{\mathcal{Z}}^{I}_q(\mathbb{X}, \mathbb{A})
= \partial_q^{-1}[i_{q-1} \mathcal{C}^{I}_{q-1}(\mathbb{A})] 
\qquad \mbox{ and }\qquad
\overline{\mathcal{B}}^{I}_q(\mathbb{X}, \mathbb{A})
=\mathcal{B}^{I}_q(\mathbb{X}, \mathbb{A})+ i_q[\mathcal{C}^{I}_q(\mathbb{A})],$$

\noindent where $+$ is the operation which gives the smallest subgroup of $\mathcal{C}^{I}_q(\mathbb{X})$ containing the two groups. Further, the homomorphism $j_q$ induces isomorphism
$$\overline{j}_q:\overline{\mathcal{H}}^I_q(\mathbb{X}, \mathbb{A})\longrightarrow \mathrm{\mathcal{H}}^I_q(\mathbb{X}, \mathbb{A}).$$
\end{lemma}
\begin{proof}
According to Theorem \ref{similar-1}, the inclusion maps 
$$ \mathbb{A}\stackrel{i}{\longrightarrow} \mathbb{X} \stackrel{j}{\longrightarrow} (\mathbb{X},\mathbb{A})$$

\noindent induce an exact sequence
$$0 \longrightarrow \mathcal{C}^{I}_q(\mathbb{A}) \stackrel{i_q}{\longrightarrow} \mathcal{C}^{I}_q(\mathbb{X}) \stackrel{j_q}{\longrightarrow} \mathcal{C}^{I}_q(\mathbb{X}, \mathbb{A})\longrightarrow 0.$$

Assume $c\in\overline{\mathcal{B}}^{I}_q(\mathbb{X}, \mathbb{A})$. Then $j_q(c)\in \mathcal{B}^{I}_q(\mathbb{X}, \mathbb{A})$ and for some $b\in \mathcal{C}^{I}_{q+1}(\mathbb{X}, \mathbb{A})$ we have that $j_q(c)=\partial_{q+1}(b)$. Since the sequence is exact, there is $m\in \mathcal{C}^{I}_{q+1}(\mathbb{X})$ such that $j_{q+1}(m)=b$. Therefore
$$j_q(c-\partial_{q+1}(m))=j_q(c)-\partial_{q+1}\circ j_{q+1}(m)=j_q(c)-\partial_{q+1}(b)=0 .$$
Again, by exactness of the sequence, there is an $n\in \mathcal{C}^{I}_q(\mathbb{A})$ such that $i_q(n)=c-\partial_{q+1}(m)$, 
and it follows that $c=\partial_{q+1}(m)+i_q(n)\in \mathcal{B}^{I}_q(\mathbb{X}, \mathbb{A})
+
i_q[\mathcal{C}^{I}_q(\mathbb{A})]$.

If $c=\partial_{q+1}(m)+i_q(n)$, then $j_q(c)=j_q(\partial_{q+1}(m)+i_q(n))= j_q\circ \partial_{q+1}(m)+ j_q(i_q(n))$. Since the sequence is exact, we have that $j_q(i_q(n))=0$ and, from the fact that $j_q\circ \partial_{q+1}=\partial_{q+1}\circ j_{q+1}$, it follows that 
$$j_q(c)=\partial_{q+1}\circ j_{q+1}(m) \text{ and } j_q(c)\in{\mathcal{B}}^{I}_q(\mathbb{X}, \mathbb{A}).$$

Assuming $e\in \overline{\mathcal{Z}}^{I}_q(\mathbb{X}, \mathbb{A})$ we have that $\partial_q\circ j_q(e)=0$. Since $\partial_q\circ j_q=j_{q-1}\circ \partial_q$, $j_{q-1}\circ \partial_q(e)=0$. From exactness of the sequence, $\ker(j_{q-1})=\mathrm{Im}(i_{q-1}).$

Hence,
$$\partial_q(e)\in i_{q-1}[\mathcal{C}^{I}_{q-1}(\mathbb{A})] \text{ which means } e\in \partial^{-1}_q(i_{q-1}[\mathcal{C}^{I}_{q-1}(\mathbb{A})]).$$ 

{The last part of the lemma follows as a consequence of the isomorphism theorem.}

\end{proof}

The following lemma is a consequence of Lemma \ref{lemma:homo-induced}.

\begin{lemma}
The boundary homomorphism $\partial_q: \mathcal{C}^{I}_q(\mathbb{X}, \mathbb{A})\longrightarrow \mathcal{C}^{I}_{q-1}(\mathbb{X}, \mathbb{A})$ defines homomorphisms 
$$\overline{\mathcal{Z}}^{I}_q(\mathbb{X}, \mathbb{A})\longrightarrow i_{q-1}[\mathcal{Z}^{I}_{q}(\mathbb{A})],$$
$$\overline{\mathcal{B}}^{I}_q(\mathbb{X}, \mathbb{A})\longrightarrow i_{q-1}[\mathcal{B}^{I}_{q}(\mathbb{A})].$$

Since the kernel of $i_{q-1}$ is zero, the composition $i_{q-1}^{-1}\circ \partial_q$ defines homomorphisms
$$\overline{\mathcal{Z}}^{I}_q(\mathbb{X}, \mathbb{A})\longrightarrow \mathcal{Z}^{I}_{q-1}(\mathbb{A}),$$
$$\overline{\mathcal{B}}^{I}_q(\mathbb{X}, \mathbb{A})\longrightarrow \mathcal{B}^{I}_{q-1}(\mathbb{A})$$

\noindent and thereby induces a homomorphism 
$$\alpha^{I}: \overline{\mathcal{H}}^I_q(\mathbb{X}, \mathbb{A})\longrightarrow \mathcal{H}^I_{q-1}(\mathbb{X}, \mathbb{A}).$$
\end{lemma}

\begin{defi}
The boundary homomorphism
$$\partial_*: \mathcal{H}^I_q(\mathbb{X}, \mathbb{A})\longrightarrow \mathcal{H}^I_{q-1}(\mathbb{A})$$
is defined to be the composition $\partial_q\circ\overline{j}_q^{-1}$.
\end{defi}

\begin{theo}[{Main Isomorphism Theorem}]\label{main.iso}
The group $\mathrm{H}^I_q(\mathbb{X}, \mathbb{A})$ is isomorphic to the group
$\mathcal{H}^I_q(\mathbb{X}, \mathbb{A})$.

\end{theo}

\begin{proof} 
To define the isomorphism, we consider first the following diagram:
$$ {H}^I_q(\mathbb{X}, \mathbb{A})\stackrel{l_*}{\longleftarrow} 
{H}^I_q(\mathbb{X}^{(q)}\cup \mathbb{A}, \mathbb{A}) 
\stackrel{j_*}{\longrightarrow} 
{H}^I_q(\mathbb{X}^{(q)}\cup \mathbb{A}, \mathbb{X}^{(q-1)}\cup \mathbb{A})$$

\noindent where the inclusions $l:(\mathbb{X}^{(q)}\cup \mathbb{A}, \mathbb{A})\longrightarrow (\mathbb{X}, \mathbb{A})$ and $j: (\mathbb{X}^{(q)}\cup \mathbb{A}, \mathbb{A})\longrightarrow (\mathbb{X}^{(q)}\cup \mathbb{A}, \mathbb{X}^{(q-1)}\cup \mathbb{A})$ are relative filtration preserving maps. The following relations hold:

\begin{enumerate} [(a).]
    \item  $j_*$ is a monomorphism and its image is $\mathcal{Z}^{I}_q(\mathbb{X}, \mathbb{A})$;  
    \item $l_*$ is an epimorphism, whose kernel is $j^{-1}_*[\mathcal{B}_q^{I}(\mathbb{X}, \mathbb{A})]$.
\end{enumerate}

In fact, if $p\neq q$, then by Lemma \ref{lema5}, one has:
\begin{equation}\label{trivial-pdifq}
{\mathrm{H}}_q^I(\mathbb{X}^{(p)}\cup \mathbb{A}, \mathbb{X}^{(p-1)}\cup \mathbb{A})=0.\end{equation} We claim some properties of the homomorphism
$$i_*: {\mathrm{H}}_q^I(\mathbb{X}^{(p)}\cup \mathbb{A}, \mathbb{A})\longrightarrow {\mathrm{H}}_q^I(\mathbb{X}^{(p+1)}\cup \mathbb{A}, \mathbb{A})$$
induced by the inclusion $(\mathbb{X}^{(p)}\cup \mathbb{A}, \mathbb{A}) \stackrel{i}{\hookrightarrow} (\mathbb{X}^{(p+1)}\cup \mathbb{A}, \mathbb{A})$.

Consider the homology sequence of the triple $(\mathbb{X}^{(p+1)}\cup \mathbb{A}, \mathbb{X}^{(p)}\cup \mathbb{A}, \mathbb{A})$ given by
\begin{multline}
         {\mathrm{H}}^I_{q+1} (\mathbb{X}^{(p+1)}\cup \mathbb{A}, \mathbb{X}^{(p)}\cup \mathbb{A}) 
    \longrightarrow
     {\mathrm{H}}^I_q (\mathbb{X}^{(p)}\cup \mathbb{A}, \mathbb{A})     \stackrel{i_\ast}{\longrightarrow}  
       \\
    \stackrel{i_\ast}{\longrightarrow}  
    {\mathrm{H}}^I_q (\mathbb{X}^{(p+1)}\cup \mathbb{A}, \mathbb{A}) \longrightarrow
    {\mathrm{H}}^I_{q} (\mathbb{X}^{(p+1)}\cup \mathbb{A}, \mathbb{X}^{(p)}\cup \mathbb{A}) . 
\end{multline}

Note that, if $q\neq p$, then (\ref{trivial-pdifq}) implies that ${\mathrm{H}}^I_{q+1} (\mathbb{X}^{(p+1)}\cup \mathbb{A}, \mathbb{X}^{(p)}\cup \mathbb{A})=0$ and then by the exactness of the homology sequence of the triple we have that $i_*$ has trivial kernel.

If $q\neq p+1$, then (\ref{trivial-pdifq}) implies that ${\mathrm{H}}^I_{q} (\mathbb{X}^{(p+1)}\cup \mathbb{A}, \mathbb{X}^{(p)}\cup \mathbb{A})=0$ and then, again by the exactness of the homology sequence of the triple, we have that $i_*$ is surjective.

Assuming $q\neq p,p+1$ we have that $i_*$ is an isomorphism.

We also claim that 
\begin{equation}\label{homology-q-1esque}
{\mathrm{H}}_q^I(\mathbb{X}^{(q-1)}\cup \mathbb{A}, \mathbb{A})=0.
\end{equation}

Note that, using that previous properties about $i_*$ one has that ${\mathrm{H}}_q^I(\mathbb{X}^{(q-1)}\cup \mathbb{A}, \mathbb{A})\cong {\mathrm{H}}_q^I(\mathbb{X}^{(q-r)}\cup \mathbb{A}, \mathbb{A})$, for $r=2,\dots,q+1$. 
Since $\mathbb{X}^{-1}=\emptyset$,  then $\mathbb{X}^{-1}\cup \mathbb{A} = \mathbb{A}$ and therefore, ${\mathrm{H}}_q^I(\mathbb{X}^{(q-1)}\cup \mathbb{A}, \mathbb{A})={\mathrm{H}}_q^I(\mathbb{A}, \mathbb{A})$ which is zero by Lemma \ref{lemma:PH-trivial}.

Let us now prove our claims. Consider the following part of the homology sequence of the triple $(\mathbb{X}^{(q)}\cup \mathbb{A}, \mathbb{X}^{(q-1)}\cup \mathbb{A}, \mathbb{A})$ given by 
\[
 {\mathrm{H}}^I_{q} (\mathbb{X}^{(q-1)}\cup \mathbb{A}, \mathbb{A}) \stackrel{i_*}{\longrightarrow} 
 {\mathrm{H}}^I_q (\mathbb{X}^{(q)}\cup \mathbb{A}, \mathbb{A}) \stackrel{j_*}{\longrightarrow} 
 {\mathrm{H}}^I_q (\mathbb{X}^{(q)}\cup \mathbb{A}, \mathbb{X}^{(q-1)}\cup \mathbb{A}) \stackrel{\alpha^{I}}{\longrightarrow} 
 {\mathrm{H}}^I_{q-1} (\mathbb{X}^{(q-1)}\cup \mathbb{A}, \mathbb{A}) .
 \]
 
 From (\ref{homology-q-1esque}) we have that $\mathrm{H}_q^I (\mathbb{X}^{(q-1)}\cup \mathbb{A}, \mathbb{A})=0$
 and then, by exactness of the homology sequence of the triple $(\mathbb{X}^{(q)}\cup \mathbb{A}, \mathbb{X}^{(q-1)}\cup \mathbb{A}, \mathbb{A})$, it follows that the kernel of $j_*$ is zero.
 
 The second part of $(a)$ follows by the following commutative diagram:
 \[
 \resizebox{\linewidth}{!}{
\xymatrix
{
	{\mathrm{H}}_{q}^I(\mathbb{X}^{(q)}\cup \mathbb{A}, \mathbb{A}) \ar[rr]^{j_*} 
        &  
        & {\mathrm{H}}_{q}^I(\mathbb{X}^{(q)}\cup \mathbb{A}, \mathbb{X}^{(q-1)}\cup \mathbb{A}) \ar[dr]_{\partial_q} \ar[rr]^{\alpha^{I}} 
        &  
        & {\mathrm{H}}_{q-1}^I(\mathbb{X}^{(q-1)}\cup \mathbb{A}, \mathbb{A}) \ar[dl]^{j_*}
    \\
        &
        &
        &  \mathrm{H}^I_{q-1}(\mathbb{X}^{(q-1)}\cup \mathbb{A}, \mathbb{X}^{(q-2)}\cup \mathbb{A}).
        & 
}}
\]

Since $\mathrm{Im} j_*=\ker\alpha^{I}$ and the kernel of $j_*$ is trivial, 
$$\ker\alpha^{I}=\ker(j_*\circ\alpha^{I})=\ker\partial_q=\mathcal{Z}^{I}_q (\mathbb{X}, \mathbb{A}).$$
 
In order to prove the statement $(b)$, note that we can decompose the homomorphism $l_*:\mathrm{H}_q^I(\mathbb{X}^{(q)}\cup \mathbb{A}, \mathbb{A})\longrightarrow\mathrm{H}_q^I(\mathbb{X}, \mathbb{A})$ into 
\[\mathrm{H}_q^I(\mathbb{X}^{(q)}\cup \mathbb{A}, \mathbb{A}) \stackrel{l_{1*}}{\longrightarrow}\mathrm{H}_q^I(\mathbb{X}^{(q+1)}\cup \mathbb{A}, \mathbb{A}) \longrightarrow\cdots\stackrel{l_{r*}}{\longrightarrow} \mathrm{H}_q^I(\mathbb{X}^{(q+r)}\cup \mathbb{A}, \mathbb{A})=\mathrm{H}_q^I(\mathbb{X}, \mathbb{A})\]
\noindent where $q+r=\dim(X)$ and $l_1,\dots,l_r$ are inclusion maps. Using the properties of $i_*:\mathrm{H}_q^I(\mathbb{X}^{(p)}\cup \mathbb{A}, \mathbb{A})\longrightarrow\mathrm{H}_q^I(\mathbb{X}^{(p+1)}\cup \mathbb{A}, \mathbb{A})$  proved previously, we have that $l_{1*}$ is onto and $l_{2*},\dots,l_{r*}$ are isomorphisms.
Therefore, from $l_*=l_{r*}\circ\cdots\circ l_{1*}$ one has that $j_*$ is an epimorphism.

To conclude, let us consider the following commutative diagram:
 \[
\xymatrix
{
    \mathcal{H}_{q+1}^I(\mathbb{X}^{(q+1)}\cup \mathbb{A}, \mathbb{X}^{(q)}\cup \mathbb{A})\ar[d]^{\alpha^{I}} \ar[drr]^{\partial_{q+1}} 
        &    &   \\
    \mathcal{H}_{q}^I(\mathbb{X}^{(q)}\cup \mathbb{A}, \mathbb{A})\ar[d]^{l_{1*}} \ar[rr]^{j_*} 
        &  
        & \mathcal{H}_{q}^I(\mathbb{X}^{(q)}\cup \mathbb{A}, \mathbb{X}^{(q-1)}\cup \mathbb{A})
        \\
    \mathcal{H}_{q}^I(\mathbb{X}^{(q+1)}\cup \mathbb{A}, \mathbb{A}) 
        &  
        &
}
\]

Note that the vertical line is part of a homology sequence of a triple. Therefore, from the exactness and $\ker l_*=\ker l_{1*}$, one {concludes}
$$j_*[\ker l_*]= j_*[\ker l_{1*}]=j_*[\mathrm{Im}\alpha^{I}]=\mathrm{Im}(j_*\circ\alpha^{I})=\mathrm{Im}\partial_{q+1}=\mathcal{B}_q^{I}(\mathbb{X}, \mathbb{A}).$$

Since $\ker j_*$ is zero, then the kernel of $l_*$ is $j^{-1}_*[\mathcal{B}_q^{I}(\mathbb{X}, \mathbb{A})]$, which conclude the statement $(b)$.

To define the isomorphism, consider the diagram
\[
\mathrm{H}^I_q(\mathbb{X}, \mathbb{A})\stackrel{l_*}{\longleftarrow} \mathrm{H}^I_q(\mathbb{X}^{(q)}\cup \mathbb{A}, \mathbb{A}) \stackrel{j_*}{\longrightarrow} \mathcal{C}^{I}_q(\mathbb{X}, \mathbb{A}) \stackrel{\eta}{\longleftarrow} \mathcal{Z}^{I}_q(\mathbb{X}, \mathbb{A}) \stackrel{\tau}{\longrightarrow} \mathcal{H}^I_q(\mathbb{X}, \mathbb{A})
\]
where $j_*$ and $l_*$ are given before, $\eta$ is an inclusion map and $\tau$ the natural map. 

The isomorphism is given by $\Theta: =\tau\circ \eta^{-1}\circ j_*\circ l_*^{-1}$.
\end{proof}

\begin{theo}
$\mathcal{H}^I_q(\mathbb{X}, \mathbb{A})=0$, for $q<0$.
\end{theo}

\begin{proof}
By Remark \ref{theo-chain}, we have that $\mathcal{C}^{I}_q(\mathbb{X}, \mathbb{A})=0$ for $q<0$. This implies that $\mathcal{Z}^{I}_q(\mathbb{X}, \mathbb{A})=0$ and thus $\mathcal{H}^I_q(\mathbb{X}, \mathbb{A})=0.$
\end{proof}

\subsection{Uniqueness}

\begin{theo}[Uniqueness Theorem]
Given $\mathrm{H}$ and $\overline{\mathrm{H}}$ two persistent homology theories satisfying \cref{as1,as2,as3}
and  
$h_0^{I}:\mathrm{H}_0^I(\mathbb{P}_{\alpha})\longrightarrow \overline{\mathrm{H}}_0^I(\mathbb{P}_{\alpha})$ 
(where $\alpha\leq\varepsilon\leq\varepsilon'$) homomorphisms between their coefficient groups, there exists a unique set of homomorphisms
$$ 
\mathrm{H}^I
(q,\mathbb{X}, \mathbb{A}):\mathrm{H}_q^I(\mathbb{X}, \mathbb{A})\longrightarrow \overline{\mathrm{H}}_q^I(\mathbb{X}, \mathbb{A})
$$
defined for each relative filtered pair $((X,A),(F_X,F_A))$, for any integer $q$
one has:

\begin{enumerate}
    \item \begin{equation*}\label{cond-1}h_0^{I}=\mathrm{H}^I(0,\mathbb{P}_\alpha,\emptyset)\end{equation*}
    
    \item If $f:(\mathbb{X}, \mathbb{A})\longrightarrow(\mathbb{X}_1, \mathbb{A}_1)$ is a filtration preserving map, where     
    $((X,A),(F_X,F_A))$ and \linebreak $((X_1,A_1),(F_{X_1},F_{A_1}))$ are relative filtered pairs, then the following diagram is commutative for all $q$ and $I$:
\begin{equation*}\label{cond-2}    
\xymatrix
{
	\mathrm{H}_{q}^I(\mathbb{X}, \mathbb{A})     \ar[d]_{f_*^{I}} \ar[rr]^{\mathrm{H}^I(q,\mathbb{X}, \mathbb{A})} 
        & 
        & \overline{\mathrm{H}}_{q}^I(\mathbb{X}, \mathbb{A}) \ar[d]^{\overline{f}_*^{I}}  
    \\
	\mathrm{H}_{q}^I(\mathbb{X}_1, \mathbb{A}_1) \ar[rr]^{\mathrm{H}^I(q, \mathbb{X}_1,\mathbb{A}_1)} 
        & 
        &  \overline{\mathrm{H}}_{q}^I(\mathbb{X}_1, \mathbb{A}_1)
}
\end{equation*}

that is, $\overline{f}_*^{I}\circ \mathrm{H}^I(q,\mathbb{X}, \mathbb{A})=\mathrm{H}^I(q,\mathbb{X}_1, \mathbb{A}_1)\circ f_*^{I}$;

    \item The commutativity relation $\overline{\partial}\circ \mathrm{H}^I(q, \mathbb{X}, \mathbb{A})=\mathrm{H}^I(q-1, \mathbb{A},{\emptyset})\circ \partial$ holds in the diagram:
    
\begin{equation*}\label{cond-3}
\xymatrix
{
	\mathrm{H}_{q}^I(\mathbb{X}, \mathbb{A})     \ar[d]_{\partial} \ar[rr]^{\mathrm{H}^I(q, \mathbb{X}, \mathbb{A})} 
        & 
        & \overline{\mathrm{H}}_{q}^I(\mathbb{X}, \mathbb{A}) \ar[d]^{\overline{\partial}}
    \\
	\mathrm{H}_{q-1}^I(\mathbb{A}) \ar[rr]^{\mathrm{H}^I(q-1, \mathbb{A},\emptyset)} 
        & 
        &  \overline{\mathrm{H}}_{q-1}^I(\mathbb{A}) 
}
\end{equation*}

If $h_0:\mathrm{H}_0^I(\mathbb{P}_{\alpha})\longrightarrow \overline{\mathrm{H}}_0^I(\mathbb{P}_{\alpha})$ is an isomorphism, then each $\mathrm{H}^I(q,\mathbb{X}, \mathbb{A})$ is also an isomorphism.

\end{enumerate}
\end{theo}

\begin{proof}
Let $g\in G^I:=\mathrm{H}^I_0(\mathbb{P}_\alpha)$ and let $gA^0\cdots A^q$ be a $q$-chain in $\mathrm{H}^I_q(\mathbb{X})$. Define
$$\eta_q^{I}(gA^0\cdots A^q):=h_{0}^{I}(g)A^0\cdots A^q$$
and for each $q$ we claim that $ \eta_q^{I}:\mathcal{C}^{I}_q(\mathbb{X})\longrightarrow \overline{\mathcal{C}}^{I}_q(\mathbb{X})$ is a homomorphism.

Indeed, if $q<0$ or $q>\#X$, follows from Remark \ref{theo-chain}, that both, $\mathcal{C}^{I}_q(\mathbb{X}, \mathbb{A})$ and $\overline{\mathcal{C}}^{I}_q(\mathbb{X}, \mathbb{A})$ are zero, and then $\eta_q^{I}$ is the null homomorphism.

If $0\leq q\leq \#X$, $\eta_q^{I}$ is defined to be map that commutes the following diagram:
\[
\xymatrix
{
	G^I \ar[dr] \ar[r]^{h_0^{I}} 
        & \overline{G}^I \ar[r] 
        & \overline{\mathcal{C}}^{I}_q(\mathbb{X})   
    \\
        & \mathcal{C}^{I}_q(\mathbb{X})  \ar[ur]_{\eta_q^{I}} 
        &     \\
}
\]

From Theorem \ref{description-chain}, each $q$-chain $c$ of $(\mathbb{X}, \mathbb{A})$ can be written uniquely in the form 
\linebreak 
$ c = \sum^{s}_{i=1} g_i A^0_i \cdots A_i^q , g_i\in G^I$ and, from Theorem \ref{permutation-of-vertex}, the maps $G^I\longrightarrow \mathcal{C}^{I}_q(\mathbb{X}, \mathbb{A})$ and $\overline{G}^I\longrightarrow \overline{\mathcal{C}}^{I}_q(\mathbb{X}, \mathbb{A})$ are homomorphisms. Then, $ \eta_q^{I}:\mathcal{C}^{I}_q(\mathbb{X})\longrightarrow \overline{\mathcal{C}}^{I}_q(\mathbb{X})$ is a homomorphism.

From Theorems \ref{filt-pres-chain} and \ref{boundary-chain} it follows that $\eta_q^{I}$ commutes with $f_q^{I}$ and $\partial^{I}_q$. 

Consequently, $\eta_q^{I}$ defines homomorphisms 
$$ \mathcal{Z}^{I}_q(\mathbb{X}, \mathbb{A})\longrightarrow \overline{\mathcal{Z}}^{I}_q(\mathbb{X}, \mathbb{A}) \text{ and } \mathcal{B}^{I}_q(\mathbb{X}, \mathbb{A})\longrightarrow \overline{\mathcal{B}}^{I}_q(\mathbb{X}, \mathbb{A})$$

\noindent and thus induce homomorphisms on their respective quotient groups:

$\mathrm{H}^I(q,\mathbb{X}, \mathbb{A}): 
\mathcal{H}^I_q(\mathbb{X}, \mathbb{A})\longrightarrow \overline{\mathcal{H}}^I_q(\mathbb{X}, \mathbb{A})$ which satisfy (1)-(3).

We will prove now that the homomorphisms $\mathrm{H}^I(q,\mathbb{X}, \mathbb{A})$ which satisfy the conditions \ref{cond-1}.-\ref{cond-3}. are unique.\vspace{0.3cm}

We claim that, if $(S^q_\alpha,F_0)$ is an ordered $q$-simplex and $g\in G^{I}$ ,then 
\begin{equation}\label{prop-1-uniq}
\mathrm{H}^I (q, \mathbb{S}_\alpha^q, \dot{\mathbb{S}}_\alpha^{q}) (gS^q_\alpha) = h_0^{I}(g)S_\alpha^{q}.\end{equation}

Indeed, let us prove by induction in $q$. Assume $q=0$, since $\mathbb{S}^0_\alpha$ is a $0$-simplex, it has empty boundary  and, by \ref{cond-1}., $h_0^{I}=\mathrm{H}^I(0,\mathbb{P}_\alpha,\emptyset)$.
Then we have that 
$$\mathrm{H}^I(0,\mathbb{S}_\alpha^0,\dot{\mathbb{S}}_\alpha^0)(gS^0_\alpha)=\mathrm{H}^I(0,\mathbb{P}_\alpha,\emptyset)(gS^0_\alpha)=h_0^{I}(g)S^0_\alpha.$$

Now assume $q>0$ and that (\ref{prop-1-uniq}) holds for $q-1$.
Let $A^0<\cdots<A^q$ be the set of vertices of $\mathbb{S}_\alpha^q$ and let $\mathbb{S}_\alpha^{q-1}$ be the face whose vertices are $A^1<\cdots<A^q$. By, Corollary \ref{correspondence-ordered-simplex}, we have that  $[\mathbb{S}_\alpha^q:\mathbb{S}_\alpha^{q-1}]gS_\alpha^q=gS_\alpha^{q-1}$. 
In Definition \ref{incidence-iso}, we define $[\mathbb{S}_\alpha^q: \mathbb{S}_\alpha^{q-1}]$ to be the incidence isomorphism $\alpha^{I}$ of Theorem \ref{incidence-isomorphism}, then explicitly, it may be defined by $[\mathbb{S}_\alpha^q:\mathbb{S}_\alpha^{q-1}]:=(j_*^{I})^{-1}\circ i_*^{I}\circ\partial$,
\[ 
\mathrm{H}^I_{p} (\mathbb{S}_\alpha^q,\dot{\mathbb{S}}_\alpha^{q})
\stackrel{\partial}{\longrightarrow} 
\mathrm{H}^I_{p-1} (\dot{\mathbb{S}}_\alpha^q) 
\stackrel{i_*}{\longrightarrow} 
\mathrm{H}^I_{p-1} (\dot{\mathbb{S}}_\alpha^{q},\mathbb{C}_\alpha^{q-1})
\stackrel{j_*}{\longleftarrow} 
\mathrm{H}^I_{p-1} (\mathbb{S}_\alpha^{q-1},\dot{\mathbb{S}}_\alpha^{q-1}),\]
where $i:(\dot{\mathbb{S}}_\alpha^q)\longrightarrow (\dot{\mathbb{S}}_\alpha^{q}, \mathbb{C}_\alpha^{q-1})$ and $j:(\mathbb{S}_\alpha^{q-1}, \dot{\mathbb{S}}_\alpha^{q-1})\longrightarrow (\dot{\mathbb{S}}_\alpha^{q}, \mathbb{C}_\alpha^{q-1})$ are inclusion maps.

Then, by \ref{cond-2}. and \ref{cond-3}., we have that 
\[
\mathrm{H}^I(q-1,\mathbb{S}_\alpha^{q-1},\dot{\mathbb{S}}_\alpha^{q-1})[\mathbb{S}_\alpha^q:\mathbb{S}_\alpha^{q-1}]=\overline{[\mathbb{S}_\alpha^q:\mathbb{S}_\alpha^{q-1}]}\mathrm{H}^I(q,\mathbb{S}_\alpha^q,\dot{\mathbb{S}}_\alpha^q).
\]

Therefore,
\[
\overline{[\mathbb{S}_\alpha^q:\mathbb{S}_\alpha^{q-1}]}\mathrm{H}^I(q,\mathbb{S}_\alpha^q,\dot{\mathbb{S}}_\alpha^q)
=h_0^{I}(g)S_\alpha^{q-1}
=\overline{[\mathbb{S}_\alpha^q:\mathbb{S}_\alpha^{q-1}]}h_0^{I}(g)S_\alpha^{q}
\]
and, since $\overline{[\mathbb{S}_\alpha^q:\mathbb{S}_\alpha^{q-1}]}$ is an isomorphism by definition, the result follows.

In addition, let $gA^0\cdots A^q\in \mathrm{C}^{I}_q(\mathbb{X},\mathbb{A})$. We claim that
\begin{equation}\label{prop-2-uniq}
\mathrm{H}^I(q,\mathbb{X},\mathbb{X})(gA^0\cdots A^q)=h^{I}_0(g)A^0\cdots A^q.
\end{equation}

By Definition \ref{defi:vertices}, we have that $gA^0\dots A^q=f^{I}_*(gS_\alpha^q)$, where $\mathbb{S}_\alpha^q$ is an ordered $q$-simplex and  $f:(\mathbb{S}_\alpha^q,\dot{\mathbb{S}}_\alpha^q)\longrightarrow (\mathbb{X}, \mathbb{X})$ is the filtration preserving map of the definition. Thus, by $\ref{cond-2}.$ and $(\ref{prop-1-uniq})$, we have
\begin{align*} 
\mathrm{H}^I(q,\mathbb{X}, \mathbb{X}) (gA^0\cdots A^q)
    & =\mathrm{H}^I(q,\mathbb{X}, \mathbb{X})(f^{I}_*(gS_\alpha^q)) 
    \\
    & \stackrel{\ref{cond-2}.}{=}\overline{f}^{I}_*(\mathrm{H}^I(q,\mathbb{X}, \mathbb{X}))(gS_\alpha^q)
    \\
    & \stackrel{(\ref{prop-1-uniq})}{=}h_0^{I}(g)A^0\cdots A^q
\end{align*}

Now let $\{\mathrm{H}^I \}$ and $\{ \mathrm{H}'^I\}$ two families of homomorphisms satisfying \ref{cond-1}.-\ref{cond-3}.. 

Let $((X,A),(F_X,F_A))$ be a relative filtered pair. We will use abbreviations, one associated to the relative filtered pair $((X^{(q)},X^{(q-1)}),(F^{(q)},F^{(q-1)}))$(see \ref{qskltn}),
$$ h_1^{I}=\mathrm{H}^I(q,\mathbb{X}^{(q)}, \mathbb{X}^{(q-1)})$$

\noindent and other associated to the relative  filtered pair $((X^{(q)},A),(F^{(q)},F_A))$,
$$ h_2^{I}=\mathrm{H}^I(q, \mathbb{X}^{(q)}, \mathbb{A}).$$

Similarly, we define $(h'_1)^{I}$ and $(h'_2)^{I}$, and by (\ref{prop-2-uniq}), we have that ${h}^{I}_1=(h'_1)^{I}$.

Let $l:((X^{(q)},A),(F^{(q)},F_A))\longrightarrow ((X^{(q)},X^{(q-1)}),(F^{(q)},F^{(q-1)})$ be the inclusion map. Then we have, by (2),
$$
\overline{l}_*\circ h_2^{I}
= h_1^{I}\circ l_*
= (h'_1)^{I}\circ l_*=\overline{l}_*\circ (h_2')^{I}.$$

Since the kernel of $\overline{l}_*$ is zero, it follows that $h_2^{I}=(h_2')^{I}$. Now considering the inclusion map  $j:((X^{(q)},A),(F^{(q)},F_A))\longrightarrow ((X,A),(F_X,F_A))$, we have 
$$ 
\mathrm{H}^I(q,\mathbb{X}, \mathbb{A})\circ j_*
=\overline{j}_*\circ h_2^{I}
=\overline{j}_*\circ (h'_2)^{I}=(\mathrm{H}')^I(q,\mathbb{X}, \mathbb{A})\circ j_*.
$$

Since $j_*$ is onto, it follows that $\mathrm{H}^I(q,\mathbb{X}, \mathbb{A})=(\mathrm{H}')^I(q,\mathbb{X}, \mathbb{A})$ for any $q$. Then, $\mathrm{H}^I=(\mathrm{H}')^I$ and uniqueness has been proved.

Now let $h_0^{I}:G^I\rightarrow \overline{G}^I$ be an isomorphism and let $\overline{h}_0^{I}:\overline{G}^I\rightarrow G^I$ be the inverse isomorphism. Let $\overline{h}(q,\mathbb{X}, \mathbb{A}):\overline{\mathcal{H}}^I_q(\mathbb{X}, \mathbb{A})\longrightarrow \mathcal{H}^I_q(\mathbb{X}, \mathbb{A})$ be the homomorphism satisfying (1)-(3) relative to $\overline{h}_0$. Then
$$ \overline{h}^{I}(q,\mathbb{X}, \mathbb{A})\circ \mathrm{H}^I(q,\mathbb{X}, \mathbb{A}):\mathcal{H}^I_q(\mathbb{X}, \mathbb{A})\longrightarrow \mathcal{H}^I_q\mathbb{X}, \mathbb{A})$$
 are homomorphisms satisfying (1)-(3) relative to the identity map $G^I\longrightarrow G^I$.

Therefore, by the uniqueness property, we have $\overline{h}^{I}(q,\mathbb{X}, \mathbb{A})\circ \mathrm{H}^I(q,\mathbb{X}, \mathbb{A})$ is the identity. 

Similarly, $\mathrm{H}^I(q,\mathbb{X}, \mathbb{A})\circ \overline{h}^{I}(q,\mathbb{X}, \mathbb{A})$ is also the identity. Thus $\mathrm{H}^I(q,\mathbb{X}, \mathbb{A})$ is an isomorphism with $\overline{h}^{I}(q,\mathbb{X}, \mathbb{A})$ as its inverse. The results follow using the isomorphisms of Theorem \ref{main.iso}.
\end{proof}

\section*{Acknowledgements}\

The authors are deeply grateful to Facundo M\'emoli for many discussions and suggestions which improved this paper.

\newpage
\appendix

\section{Theories on Persistent Homology}
\label{ap1}

\subsection{ The Reduced Persistent Homology}

The following theorem gives   a homomorphism between the persistent homology sequences of $\mathbb{(X,A)}$ and of $ \mathbb{(Y,B)}.$ 

\begin{theo}\label{thm:rest-to-homo}
Given a filtration preserving map of relative filtered sets $f:\mathbb{(X,A)}\longrightarrow \mathbb{(Y,B)}$, with inclusion maps  $i: \mathbb{A}\longrightarrow \mathbb{X}$, $j :\mathbb{X}\longrightarrow \mathbb{(X,A)}$,  $i': \mathbb{B}\longrightarrow \mathbb{Y}$ and $j' :\mathbb{Y}\longrightarrow \mathbb{(Y,B)}$, and  restriction maps 
 $f_1:\mathbb{X}\longrightarrow \mathbb{Y} \text{ and } f_2:\mathbb{A}\longrightarrow \mathbb{B}$, the following diagram is commutative:
\[
\xymatrix
{
\cdots \ar[r]	& \mathrm{H}_{n+1}^I(\mathbb{X},\mathbb{A}) \ar[r]^{\partial^{I}_X }\ar[d]^{f_*}& \mathrm{H}_{n}^I(\mathbb{A})     \ar[d]^{f_{2*}} \ar[r]^{i_*} & \mathrm{H}_{n}^I(\mathbb{X}) \ar[r]^{j_*} \ar[d]^{f_{1*}} & \mathrm{H}_{n}^I(\mathbb{X},\mathbb{A}) \ar[d]^{f_*} \ar[r]  & \cdots \\
\cdots \ar[r]	& \mathrm{H}_{n+1}^I(\mathbb{Y},\mathbb{B}) \ar[r]^{\partial^{I}_Y}& \mathrm{H}_{n}^I(\mathbb{B}) \ar[r]^{i'_*} & \mathrm{H}_{n}^I(\mathbb{Y}) \ar[r]^{j'_*}& \mathrm{H}_{n}^I(\mathbb{Y},\mathbb{B}) \ar[r] & \cdots
}
\]
\end{theo}

\begin{proof}

Observe that in the considered diagram,   rows are  exact sequences of relative sets and vertical arrows are induced maps of $f, f_1$ and $f_2$, respectively. The commutativity of the diagram follows from {Axioms \ref{ax2} and \ref{ax3}}.
 
\end{proof}

\begin{coro}\label{homo-filt-iso}
Given a filtration preserving map of relative filtered sets $f:\mathbb{(X,A)}\longrightarrow \mathbb{(Y,B)}$, let us assume that  the induced maps $f_{1*}: \mathrm{H}^I_n(\mathbb{X})\longrightarrow \mathrm{H}^I_n(\mathbb{Y})$ and $f_{2*}: \mathrm{H}^I_n(\mathbb{A})\longrightarrow \mathrm{H}^I_n(\mathbb{B})$ are isomorphisms. Then $f_*: \mathrm{H}^I_n\mathbb{(X,A)}\longrightarrow \mathrm{H}^I_n\mathbb{(Y,B)}$ is an isomorphism.
\end{coro}
\begin{proof}
The result is a consequence of Theorem \ref{thm:rest-to-homo} and the Five Lemma. 

\end{proof}

\begin{theo}\label{thm:cont-to-iso}
A contiguity $f$ from $(\mathbb{X},\mathbb{A})$ onto $(\mathbb{Y},\mathbb{B})$ induces an isomorphism $f^{I}_{*}: \mathrm{H}^I_n\mathbb{(X,A)}\longrightarrow \mathrm{H}^I_n\mathbb{(Y,B)}$.  
\end{theo}

\begin{proof}
From the definition of contiguity and using {\cref{ax5}}, we have that 
$(g\circ f)_{*}^{I}=(\mathop{id}_{(\mathbb{X},\mathbb{A})})_{*}^{I}$ and 
$(f\circ g)_{*}^{I}=(\mathop{id}_{(\mathbb{Y},\mathbb{B})})_{*}^{I}$,
where   $(\mathop{id}_{(\mathbb{X},\mathbb{A})})_{*}^{I}$ and $(\mathop{id}_{(\mathbb{Y},\mathbb{B})})_{*}^{I}$ are identity homomorphisms by Axiom \ref{ax1}.

From Axiom \ref{ax2},  $(f\circ g)^{I}_{*}=f^{I}_{*}\circ g^{I}_{*}$ and $(g\circ f)^{I}_{*}=g^{I}_{*}\circ f^{I}_{*}$, which {conclude} the proof. 
\end{proof}

The next goal is to define the reduced persistent homology associated to a persistent homology theory. In order to do this we will need some definitions and results.

\begin{defi}\label{def-persist-group}
{
Let $p_0$ be a fixed point. For each $\alpha\geq 0$, let $\mathbb{P}_\alpha = (P_\alpha, F_{P_\alpha})$ be representing the filtered set where
$P_\alpha = \{p_0\}$ and $F_{P_\alpha}(P_\alpha) = \alpha$.
 Given a persistent homology theory, \textbf{the persistent coefficient group associated to $I$} and $\alpha$ is defined as the group
 $G^I:=\mathrm{H}_0^I(\mathbb{P}_\alpha)$.
}
 \smallskip

Let $\mathbb{X}=(X, F_X)$ be a filtered set.
If $x\in X$, $\alpha = F_X(\{x\})$ and $g\in G^I$, then $(gx)_X$ will denote the image of $g$ in $\mathrm{H}^I_0(\mathbb{X})$ under the homomorphism induced by the map $f:\mathbb{P}_\alpha\longrightarrow \mathbb{X}$ defined by $f(p_0):=x$. The image of $G^I$ in $\mathrm{H}_0^I(\mathbb{X})$ under $f_*^{I}$ is denoted by $(G^Ix)_X$.
\end{defi}

\begin{remark}
An important remark here is that the map $f$ is a filtration preserving map, since $F_{P_\alpha} (P_\alpha) = F_X (f(P_\alpha))$.
\end{remark}

\begin{theo}\label{thm:filt-to-onto}
If  $f:\mathbb{X}\longrightarrow \mathbb{Y}$  is a filtration preserving map, then for  $x\in X$ and $g\in G^I$,
$f^{I}_*$ maps $(G^Ix)_X$ onto  $(G^If(x))_Y$.

\end{theo}

\begin{proof}

Let $\alpha = F_X(\{x\})$ and let $\mathbb{P}_\alpha$ be the filtered set given in Definition \ref{def-persist-group}.

Consider the following commutative diagram: 
\[
\xymatrix
{
    \mathbb{P}_\alpha \ar[dr]_{T_{y}} \ar[rr]^{T_{x}} &  & \mathbb{X} \ar[dl]^{f}   \\
	& \mathbb{Y}  &  \\
}
\]
where $T_{x}:\mathbb{P}_\alpha\longrightarrow \mathbb{X}$ is defined by $T_{x}(p_0):=x$ and $T_{y}:\mathbb{P}_\alpha\longrightarrow \mathbb{Y}$ is defined by $T_{y}(p_0):=y = f(x)$. Then by Axiom \ref{ax2} the  diagram commutes:
\[
\xymatrix
{
	\mathrm{H}_*^I(\mathbb{P}_\alpha) \ar[dr]_{(T_{y})_\ast} \ar[rr]^{(T_{x})_\ast} &  & \mathrm{H}_*^I(\mathbb{X}) \ar[dl]^{f^{I}_*}   \\
	& \mathrm{H}_*^I(\mathbb{Y})  &  \\
}
\]
which implies that $f^{I}_*((gx)_X)=(gy)_Y$.
  Then we have that $f^{I}_*$ maps $(G^{I} x)_X$ onto $(G^Iy)_Y$.
\end{proof}

\begin{defi}
Let  $\mathbb{P}_\alpha$ be the filtered set where $P_\alpha$ is a single-point set with filtration map $F_{P_\alpha}$ defined by $F_{P_\alpha}(P_\alpha):=\alpha$ and let $f: \mathbb{X}\longrightarrow \mathbb{P}_\alpha$ be a filtration preserving map. We define the \textbf{reduced persistent homology group} 
$$ 
\tilde{\mathrm{H}}^I_n\mathbb{(X)}:=\ker (f^{I}_{*}: \mathrm{H}^I_n \mathbb{(X)} \longrightarrow \mathrm{H}^I_n (\mathbb{P}_\alpha)) .
$$

\begin{remark}\label{Remark-3}
For $n>0$   we have that $\mathrm{H}^I_n (\mathbb{P}_\alpha)=0$ 
and $\tilde{\mathrm{H}}^I_n(\mathbb{X})=\mathrm{H}^I_n(\mathbb{X})$.
\end{remark}

\end{defi}

\begin{theo}
If $P$ is a single-point set,
 with the filtration map $F_P$, $F_P(P) :=\alpha$, then 
$$\tilde{\mathrm{H}}^I_0(\mathbb{P})=0 \text{ and } \mathrm{H}^I_0(\mathbb{P})=(G^IP)_P .$$
\end{theo}

\begin{proof}
Considering $f: (P,F_P)\longrightarrow \mathbb{P}_\alpha$ and $h: \mathbb{P}_\alpha\longrightarrow (P,F_P)$ well defined maps between single-point sets one has that $f$ is a contiguity. Therefore,
\[
f^{I}_{*}: \mathrm{H}^I_n(\mathbb{P})\longrightarrow \mathrm{H}^I_n(\mathbb{P}_\alpha) 
\]
is an isomorphism and
$\tilde{\mathrm{H}}^I_0(\mathbb{P})=0$.
Also 
\[
(Gx)P = h^{I}_{*}(G^I) =
{h}^{I}_{*}(\mathrm{H}_0^I(\mathbb{P}_\alpha)) = 
\mathrm{H}_0^I(\mathbb{P}).
\]
\end{proof}

\begin{theo}
Let $\mathbb{X}$ be a filtered set and $x\in X$. Then $\mathrm{H}^I_0(\mathbb{X})$ decomposes into the direct sum
\[
\mathrm{H}^I_0(\mathbb{X})=\tilde{\mathrm{H}}^I_0\mathbb{(X)}\oplus (G^Ix)_X
\]
and the correspondence $g\longrightarrow (gx)_X$ maps $G^I$ isomorphically onto $(G^Ix)_X$.
\end{theo}

\begin{proof}
Consider $h= f_2\circ f_1$, where $f_1:\mathbb{P}_\alpha\longrightarrow \mathbb{X}$ is the filtration preserving map defined by $f_1(P_\alpha):=x$ and  $f_2:\mathbb{X}\longrightarrow \mathbb{P}_\alpha$ is the constant filtration preserving map.
Observe that $h$ is the identity map of $P_\alpha$, and by Axiom \ref{ax1} and Axiom \ref{ax2}, the composition of $f_{1*}^{I}:G^I\longrightarrow \mathrm{H}^I_0\mathbb{(X)}$ and $f_{2*}^{I}:\mathrm{H}^I_0\mathbb{(X)}\longrightarrow G^I$ is the identity map. 
Therefore, $f_1$ maps $G^I$ isomorphically onto $(G^Ix)_X$.

Since $\tilde{\mathrm{H}}^I_0\mathbb{(X)}=\ker f^{I}_{2*}$, one has that $(G^Ix)_X\cap \tilde{\mathrm{H}}^I_0\mathbb{(X)}=0$.
Taking $m\in \mathrm{H}^I_0\mathbb{(X)}$ and defining $m':= f_{1*}\circ f_{2*}(m)$ and $m''=m-m'$, it is straightforward that $m = m' +m''$. The result follows.

\end{proof}

\begin{theo}\label{thm:iso-onto}
If $f:\mathbb{X}\longrightarrow \mathbb{Y}$ is a filtration preserving map , $x\in X$, $y\in Y$ and $y=f(x)$, then $f^{I}_*$ maps
$\tilde{\mathrm{H}}^I_0\mathbb{(X)}$ into $\tilde{\mathrm{H}}^I_0\mathbb{(Y)}$ and maps $(G^Ix)_X$ isomorphically onto $(G^Iy)_Y$. 
\end{theo}

\begin{proof}
Let $f_1:\mathbb{P}_\alpha\longrightarrow \mathbb{X}$ be the filtration preserving map defined by $f_1(P_\alpha):=x$ and let $f_2:\mathbb{Y}\longrightarrow \mathbb{P}_\alpha$ be the constant map. Then $f_2\circ f\circ f_1$ is the identity map of $P_\alpha$ and from Axiom \ref{ax1} and Axiom \ref{ax2},  the composition 
$$  \mathrm{H}^I_0 (\mathbb{P}_\alpha) \stackrel{f^{I}_{1*}}{\longrightarrow} \mathrm{H}^I_0 \mathbb{(X)} \stackrel{f^{I}_{*}}{\longrightarrow} \mathrm{H}^I_0 \mathbb{(Y)} \stackrel{f^{I}_{2*}}{\longrightarrow} \mathrm{H}^I_0 (\mathbb{P}_\alpha) $$

\noindent and particularly
$$  G^I \stackrel{f^{I}_{1*}}{\longrightarrow} (G^Ix)_X \stackrel{f^{I}_{*}}{\longrightarrow} (G^Iy)_Y \stackrel{f^{I}_{2*}}{\longrightarrow} G^I $$

\noindent are  identity maps.

Since each homomorphism is injective and, by Theorem \ref{thm:filt-to-onto}, each one is also surjective, then we have that each homomorphism is an isomorphism, which proves that $f^{I}_*$ maps $(G^Ix)_X$ isomorphically onto $(G^Iy)_Y$.

Since  $f_2\circ f$ is a map from $\mathbb{X}$ into $\mathbb{P}_\alpha$,  by definition we have that $\tilde{\mathrm{H}}^I_0\mathbb{(X)}=\ker (f_2\circ f)_*^{I}$. Then to prove that $f^{I}_*$ maps
$\tilde{\mathrm{H}}^I_0\mathbb{(X)}$ into $\tilde{\mathrm{H}}^I_0\mathbb{(Y)}$ note that 
$$f^{I}_*(\tilde{\mathrm{H}}^I_0\mathbb{(X)})=f^{I}_*(\ker (f_2\circ f)_*^{I})=f^{I}_*(\ker (f_{2*}^{I}\circ  f_*^{I}))\subseteq \ker (f_{2*}^{I})=\tilde{\mathrm{H}}^I_0\mathbb{(Y)}.$$
\end{proof}

{The restriction of the map $f_*^{I}$ to the domain $\tilde{\mathrm{H}}^I_0\mathbb{(X)}$ and the codomain $\tilde{\mathrm{H}}^I_0\mathbb{(Y)}$  will be denoted by $\tilde{f}_\ast^{I}:\tilde{\mathrm{H}}^I_0\mathbb{(X)}\to  \tilde{\mathrm{H}}^I_0\mathbb{(Y)}$ .}

\begin{lemma}\label{lemma:PH-trivial}
For each filtered set $\mathbb{X}=(X, F_X)$ and each non-negative integer $q$, $$\mathrm{H}^I_q \mathbb{(X,X)}=0.$$
\end{lemma}

\begin{proof}
Let $i:\mathbb{X}\longrightarrow \mathbb{X}$ and $j:\mathbb{X}\longrightarrow (\mathbb{X}, \mathbb{X})$ be inclusion maps and consider the  following part of the persistent  homology exact sequence of the filtered relative set $\mathbb{(X,X)}$:
$$\mathrm{H}_{q}^I\mathbb{(X)} \stackrel{i_*}{\longrightarrow} \mathrm{H}_{q}^I\mathbb{(X)}\stackrel{j_*}{\longrightarrow} \mathrm{H}_{q}^I\mathbb{(X,X)}\stackrel{\alpha^{I}_{X}}{\longrightarrow} \mathrm{H}_{q-1}^I\mathbb{(X)} \stackrel{i_*}{\longrightarrow} \mathrm{H}_{q-1}^I\mathbb{(X)} .$$

Since each $i_*$ is an isomorphism, it follows by exactness that $0=\ker (i_*)=\mathrm{Im} (\alpha^{I}_X)$. Therefore, $\mathrm{H}^I_q\mathbb{(X,X)}=\ker (\alpha^{I}_X)$. Similarly $\mathrm{H}^I_q\mathbb{(X)}=\mathrm{Im} (i_*)=\ker (j_*)$ and $\mathrm{Im} (j_*)=0$. Also $\mathrm{Im} (j_*)=\ker (\alpha^{I}_X)$ and the result follows. 
\end{proof}

\begin{theo}
 $\alpha^{I}_A$ maps $\mathrm{H}_1^I\mathbb{(X,A)}$ into $\tilde{\mathrm{H}}_0^I\mathbb{(A)}$ in the exact sequence of the relative filtered set $\mathbb{(X,A)}$
\end{theo}

\begin{proof}
Let $f:\mathbb{(X,A)}\longrightarrow (\mathbb{P}_0,\mathbb{P}_0)$ be the constant map.
If $h\in \mathrm{H}_1^I\mathbb{(X,A)}$, then $f_*^{I}(h)\in \mathrm{H}_1^I(\mathbb{P}_0,\mathbb{P}_0)=0$, by Lemma \ref{lemma:PH-trivial}. 

From Axiom \ref{ax3}  the following diagram commutes: 
\[
\xymatrix
{
	\mathrm{H}_{1}^I\mathbb{(X,A)}     \ar[d]_{f_*} \ar[rr]^{\alpha^{I}_A} & & \mathrm{H}_{0}^I\mathbb{(A)} \ar[d]^{(f|_A)_*}  \\
	\mathrm{H}_{1}^I(\mathbb{P}_0,\mathbb{P}_0) \ar[rr]^{\alpha^{I}_{P_0}} & &  \mathrm{H}_{0}^I(\mathbb{P}_0)
}
\]

\noindent Therefore $(f|_A)_*\circ \alpha^{I}_A=\alpha^{I}_B\circ f_*(h)=0$ and then  by definition, $\alpha^{I}_A(h)\in \tilde{\mathrm{H}}_0^I\mathbb{(A)}$.
\end{proof}

\begin{defi}
Let $f:\mathbb{(X,A)}\longrightarrow (\mathbb{P}_0,\mathbb{P}_0)$ be a filtration preserving map. The reduced homology sequence of $\mathbb{(X,A)}$ is the sequence composed by the kernels of the induced maps defined in Theorem \ref{thm:rest-to-homo}.
\end{defi}

\begin{theo}\label{theo:red-seq}
The reduced homology sequence of $\mathbb{(X,A)}$ differs from the homology sequence of $\mathbb{(X,A)}$ only on the part  
$$\mathrm{H}_{1}^I\mathbb{(X,A)} \stackrel{\alpha^{I}_A}{\longrightarrow} \mathrm{H}_{0}^I\mathbb{(A)}\stackrel{i_*}{\longrightarrow} \mathrm{H}_{0}^I\mathbb{(X)}\stackrel{j_*}{\longrightarrow} \mathrm{H}_{0}^I\mathbb{(X,A)}$$
\noindent which is replaced by
$$
\mathrm{H}_{1}^I\mathbb{(X,A)} 
\stackrel{\tilde{\alpha}^{I}_A}{\longrightarrow} 
\tilde{\mathrm{H}}_{0}^I\mathbb{(A)}
\stackrel{\tilde{i}_*}{\longrightarrow} 
\tilde{\mathrm{H}}_{0}^I\mathbb{(X)}
\stackrel{\tilde{j}_*}{\longrightarrow}
\tilde{\mathrm{H}}_{0}^I\mathbb{(X,A)},
$$
\noindent where, $\tilde{i}_*$,$\tilde{j}_*$ and $\tilde{\alpha}^{I}_A$ are the maps induced by $i_*$,$j_*$ and $\alpha^{I}_A$, respectively.
\end{theo}

\begin{proof}
From Remark \ref{Remark-3}, it follows  that $\mathrm{H}^I_0 (\mathbb{P}_0)=0$ for $n>0$. Then $\ker (f^{I}_{*})=\mathrm{H}^I_n {(\mathbb{P}_0)}$. 

For  $n=0$ we have changes in the sequence, but by definition $\tilde{\mathrm{H}}^I_n\mathbb{(X)}:=\ker (f^{I}_{*}: \mathrm{H}^I_n \mathbb{(X)} \longrightarrow \mathrm{H}^I_n (\mathbb{P}_0))$ and also for the pair $(X,A)$  and for $A$, which concludes the proof.
\end{proof}

\begin{defi}A filtered set $\mathbb{X}$ is said to be \textbf{homologically trivial} if  $\tilde{\mathrm{H}}_q^I\mathbb{(X)}=0$, for all $q$.
If $A\neq\emptyset$, $\mathbb{(X,A)}$ is said to be homologically trivial if $\tilde{\mathrm{H}}_q^I\mathbb{(X,A)}=0$ for all $q$.
\end{defi}

\begin{theo}\label{thm: homotriv-1}
If $\mathbb{(X,A)}$ is homologically trivial and $A\neq\emptyset$, then the homomorphisms 
$i_*:\mathrm{H}_q^I\mathbb{(A)}\to \mathrm{H}_q^I\mathbb{(X)}$, $q>0$, and 
$\tilde{i}_*:\tilde{\mathrm{H}}_0^I\mathbb{(A)}\to \tilde{\mathrm{H}}_0^I\mathbb{(X)}$, induced by the inclusion, are isomorphisms.
\end{theo}

\begin{proof}
Since  $\mathbb{(X,A)}$ is homologically trivial, the exact persistent homology sequence is 
$$0\longrightarrow \mathrm{H}_{q}^I\mathbb{(A)}\stackrel{i_*}{\longrightarrow} \mathrm{H}_{q}^I\mathbb{(X)}\longrightarrow 0$$
\noindent and then $i_*$ is an isomorphism, for all $q$.

The second isomorphism follows from the definition of $\tilde{\mathrm{H}}^I_n\mathbb{(X)}$ and from the fact that $(\mathbb{X},\mathbb{A})$ is homologically trivial, which implies that $\tilde{\mathrm{H}}_0^I\mathbb{(X,A)}=0$.

\end{proof}

\begin{theo}\label{thm: homotriv-2}
If $\mathbb{A}$ is homologically trivial and $\emptyset\neq \mathbb{A}\subset \mathbb{X}$, 
then the homomorphisms $j_*:\mathrm{H}_q^I\mathbb{(X)}\to \mathrm{H}_q^I\mathbb{(X,A)}$, $q>0$ and $\tilde{j}_*:\tilde{H}_0^I\mathbb{(X)}\to \tilde{H}_0^I\mathbb{(X,A)}$, induced by the inclusion  $j:(\mathbb{X},\emptyset)\longrightarrow \mathbb{(X,A)}$, are isomorphisms.
\end{theo}

\begin{proof}
Since  $\mathbb{A}$ is homologically trivial, 
$\mathrm{H}_q^I\mathbb{(A)}=0$ for $q\neq 0$ and $\tilde{H}_0^I\mathbb{(X)}=0$. Therefore, the persistent homology sequence is 
\[
0\longrightarrow \mathrm{H}_{q}^I\mathbb{(X)}\stackrel{i_*}{\longrightarrow} \mathrm{H}_{q}^I\mathbb{(X,A)}\longrightarrow 0.
\]
\noindent By exactness,
$$i_*:\mathrm{H}_q^I\mathbb{(X)}\approx \mathrm{H}_q^I\mathbb{(X,A)}, \forall q.$$

For the second isomorphism, we consider the following part of the reduced persistent homology sequence:
$$\cdots\longrightarrow \tilde{\mathrm{H}}_{0}^I\mathbb{(A)}\stackrel{\tilde{i}_*}{\longrightarrow} \tilde{\mathrm{H}}_{0}^I\mathbb{(X)}\stackrel{\tilde{j}_*}{\longrightarrow} \mathrm{H}_{0}^I\mathbb{(X,A)}\longrightarrow 0 .$$

Since $\mathrm{H}_{0}^I\mathbb{(X,A)}=\tilde{\mathrm{H}}_{0}^I\mathbb{(X,A)}$ and $\tilde{\mathrm{H}}_{0}^I\mathbb{(A)}\subseteq \mathrm{H}_{0}^I\mathbb{(A)}=0$, by definition of homologically trivial, the sequence becomes 
\[
 \longrightarrow \tilde{\mathrm{H}}_{0}^I\mathbb{(X)}\stackrel{\tilde{j}_*}{\longrightarrow} \tilde{\mathrm{H}}_{0}^I\mathbb{(X,A)}\longrightarrow 0
 \]
\noindent and by exactness we obtain the required isomorphism.
\end{proof}

\begin{theo}\label{thm: homotriv-3}
Let $\mathbb{(X,A)}$ be a relative filtered set. If $\mathbb{X}$ is homologically trivial and $\mathbb{A}\neq\emptyset$, then the boundary homomorphisms 
$\alpha^{I}_{A*}:\mathrm{H}_q^I\mathbb{(X,A)}\to \mathrm{H}_{q-1}^I\mathbb{(A)}$, $q>1$, and $\tilde{\alpha^{I}}_A: \mathrm{H}_1^I\mathbb{(X,A)}\to \tilde{\mathrm{H}}_0^I\mathbb{(A)}$ are isomorphisms.
\end{theo}

\begin{proof}
The proof of this theorem is analogous to the previous two theorems.
\end{proof}

\begin{coro}\label{coro-homo-triv}
The relative filtered set $\mathbb{(X,A)}$ is homologically trivial if both $\mathbb{X}$ and $\mathbb{A}$ are homologically trivial.
\end{coro}

\begin{proof}
The proof follows from Theorems \ref{thm: homotriv-1}, \ref{thm: homotriv-2} and \ref{thm: homotriv-3}.
\end{proof}

\subsection{Filtered Triads}

\begin{defi}Assume that
the inclusion maps
$\overline{i}: \mathbb{(A,B)}\subset \mathbb{(X,B)}$ and $\overline{j}:\mathbb{(X,B)}\subset \mathbb{(X,A)}$
are filtration preserving maps. The set $((X,A,B),(F_X,F_A,F_B))$, denoted by $\mathbb{(X,A,B)}$, is called \textbf{relative filtered triple}. The inclusion maps and boundary homomorphisms associated with the relative filtered sets $\mathbb{(X,A)}$, $\mathbb{(X,B)}$ and $\mathbb{(A,B)}$ are, respectively, denoted by \smallskip

$i:\mathbb{A}\longrightarrow \mathbb{X},\hspace{0.5cm} j:\mathbb{X}\longrightarrow \mathbb{(X,A)}, \hspace{0.5cm} \alpha^{I}:\mathrm{H}_q^I\mathbb{(X,A)}\longrightarrow \mathrm{H}_{q-1}^I\mathbb{(A)}$  \smallskip

$i':\mathbb{B}\longrightarrow \mathbb{X},\hspace{0.5cm} j':\mathbb{X}\longrightarrow \mathbb{(X,B)}, \hspace{0.5cm} {\alpha'}^{I}:\mathrm{H}_q^I\mathbb{(X,B)}\longrightarrow \mathrm{H}_{q-1}^I\mathbb{(B)}$ \smallskip

$i'':\mathbb{B}\longrightarrow \mathbb{A},\hspace{0.5cm} j'':\mathbb{A}\longrightarrow \mathbb{(A,B)}, \hspace{0.5cm} {\alpha''}^{I}:\mathrm{H}_q^I\mathbb{(A,B)}\longrightarrow \mathrm{H}_{q-1}^I\mathbb{(B)}$.

The homomorphism 
$\overline{\alpha^{I}}=j_*''\alpha^{I}$ is called \textbf{the connecting  operator} of the triple $\mathbb{(X,A,B)}$:
$$ \overline{\alpha^{I}}: \mathrm{H}_q^I\mathbb{(X,A)}\longrightarrow \mathrm{H}_{q-1}^I\mathbb{(A,B)}$$
\noindent and the sequence of groups
$$  \cdots\longrightarrow \mathrm{H}^I_q \mathbb{(A,B)} \stackrel{\overline{i}_*}{\longrightarrow} \mathrm{H}^I_q \mathbb{(X,B)} \stackrel{\overline{j}_*}{\longrightarrow} \mathrm{H}^I_q \mathbb{(X,A)} \stackrel{\overline{\alpha^{I}}}{\longrightarrow} \mathrm{H}^I_{q-1} \mathbb{(A,B)}\longrightarrow\cdots $$
\noindent is called \textbf{the  persistent homology sequence of the triple } $\mathbb{(X,A,B)}$.
\end{defi}

\begin{theo}
The persistent homology sequence of a triple is exact.
\end{theo}

\begin{proof}
The proof is similar to the exactness of the relative persistent homology sequence of the pair. 
\end{proof}

\begin{theo}\label{thm:induced-triple}
A filtration preserving map $g:\mathbb{(X,A,B)}\longrightarrow \mathbb{(X',A',B')}$ induces a homomorphism from the persistent homology sequence of $\mathbb{(X,A,B)}$ into the persistent homology sequence of $\mathbb{(X',A',B')}$.
\end{theo}

\begin{proof}

A filtration preserving map $g:\mathbb{(X,A,B)}\longrightarrow \mathbb{(X',A',B')}$ defines three filtration preserving maps
$$ g_1:\mathbb{(A,B)}\longrightarrow \mathbb{(A',B')} \text{ , } g_2:\mathbb{(X,B)}\longrightarrow \mathbb{(X',B')}  \text{ and } g_3:\mathbb{(X,A)}\longrightarrow \mathbb{(X',A')}, $$

\noindent obtained from $g$. To prove the theorem, consider the following diagram
\[
\xymatrix
{
 \ar[r]
    & \mathrm{H}_{n}^I\mathbb{(A,B)} \ar[r]^{\overline{i}_*} \ar[d]^{g_{1*}}
    & \mathrm{H}_{n}^I\mathbb{(X,B)}     \ar[d]^{g_{2*}} \ar[r]^{\overline{j}_*} 
    & \mathrm{H}_{n}^I\mathbb{(X,A)} \ar[r]^{\overline{\alpha^{I}}} \ar[d]^{g_{3*}} 
    & \mathrm{H}_{n-1}^I\mathbb{(A,B)} \ar[d]^{g_{1*}} \ar[r] 
    & \\
 \ar[r]	
    & \mathrm{H}_{n}^I\mathbb{(A',B')} \ar[r]^{\overline{i'}_*}
    & \mathrm{H}_{n}^I\mathbb{(X',B')} \ar[r]^{\overline{j'}_*} 
    & \mathrm{H}_{n-1}^I\mathbb{(X',A')} \ar[r]^{\overline{{\alpha'}^{I}}}
    & \mathrm{H}_{n-1}^I\mathbb{(A',B')} \ar[r] 
    &
}
\]

\noindent where $\overline{i},\overline{j}, \overline{i'}, \overline{j'}$ are  appropriate inclusion maps, whose commutativity follows in a analogous way to the proof of Theorem \ref{thm:rest-to-homo}.
\end{proof}

\begin{theo}\label{4.14}
If $B\neq \emptyset$  and  one of the relative filtered sets $\mathbb{(X,A)}$, $\mathbb{(X,B)}$ or $\mathbb{(A,B)}$ is homologically trivial, then the persistent homology group of the remaining two relative filtered sets are isomorphic under the maps of the persistent homology sequence of $\mathbb{(X,A,B)}$.

\begin{enumerate}
    \item If $\mathrm{H}_q^I\mathbb{(X,A)}=0$ for each $q$, then $\overline{i}_*:\mathrm{H}_q^I\mathbb{(A,B)}\approx \mathrm{H}_q^I\mathbb{(X,B)}$ for each $q$.
    
    \item If $\mathrm{H}_q^I\mathbb{(X,B)}=0$ for each $q$, then $\overline{\alpha^{I}}:\mathrm{H}_q^I\mathbb{(X,A)}\approx \mathrm{H}_{q-1}^I\mathbb{(A,B)}$ for each $q$.
    
    \item If $\mathrm{H}_q^I\mathbb{(A,B)}=0$ for each $q$, then $\overline{j}_*:\mathrm{H}_q^I\mathbb{(X,B)}\approx \mathrm{H}_q^I\mathbb{(X,A)}$ for each $q$.
\end{enumerate}

Conversely, any one of the three statements implies the corresponding hypothesis.
\end{theo}
 
\begin{proof}
The proof of this theorem is analogous to the proof of Theorems \ref{thm: homotriv-1}, \ref{thm: homotriv-2} and \ref{thm: homotriv-3} by just changing the notation. Note that in that proof, we used the homologically trivial property to obtain that $\mathrm{H}_q^I\mathbb{(X,A)}=0$, $\mathrm{H}_q^I\mathbb{(X,B)}=0$ and $\mathrm{H}_q^I\mathbb{(A,B)}=0$, which we have by hypothesis.
\end{proof}

\begin{theo}
Let $\mathbb{(X,A,B)}$ be a relative filtered triple. If the inclusion $\mathbb{B}\subset \mathbb{A}$ induces an isomorphism $\mathrm{H}_q^I\mathbb{(B)}\approx \mathrm{H}_q^I\mathbb{(A)}$ for all values of $q$, then the inclusion map $\mathbb{(X,B)}\subset \mathbb{(X,A)}$ also induces an isomorphism $\mathrm{H}_q^I\mathbb{(X,B)}\approx \mathrm{H}_q^I\mathbb{(X,A)}$ for all $q$. Similarly, if $\mathbb{A}\subset \mathbb{X}$ induces $\mathrm{H}_q^I\mathbb{(A)}\approx \mathrm{H}_q^I\mathbb{(X)}$ for all $q$, then $\mathbb{(A,B)}\subset \mathbb{(X,B)}$ induces $\mathrm{H}_q^I\mathbb{(A,B)}\approx \mathrm{H}_q^I\mathbb{(X,B)}$ for all $q$.
\end{theo}

\begin{proof}
The first part of the theorem follows from the hypothesis that $\mathrm{H}_q^I\mathbb{(B)}\approx \mathrm{H}_q^I\mathbb{(A)}$ for all values of $q$ and by the exactness of the persistent homology sequence of a triple $\mathbb{(X,A,B)}$. 

Applying Corollary\ref{homo-filt-iso} to the inclusion $\mathbb{(X,B)}\subset \mathbb{(X,A)}$, we obtain the proof of the second assertion.
\end{proof}

To continue, it will be important to remember the definitions of contiguous and  contiguously equivalence (\cref{contiguous},\ref{seqcontiguous},\ref{eqcontiguous}).

\begin{theo}\label{thm:contequi-to-iso}
If $f:(\mathbb{X},\mathbb{A})\longrightarrow(\mathbb{Y},\mathbb{B})$ and $g:(\mathbb{Y},\mathbb{B})\longrightarrow(\mathbb{X},\mathbb{A})$ 
are contiguously equivalent, then 
$f$ induces an isomorphism
$f_*:\mathrm{H}_q^I\mathbb{(X,A)}$ $\longrightarrow \mathrm{H}_q^I\mathbb{(Y,B)}$ and $(f_*)^{-1}=g_*$, for all $q$.
\end{theo}

\begin{proof} From hypothesis, one has that $g\circ f$ and $\mathop{id}_{(\mathbb{X},\mathbb{A})}$ are contiguous and $f\circ g$ and $\mathop{id}_{(\mathbb{Y},\mathbb{B})}$ are contiguous. 

Therefore  
by Axiom \ref{ax5} and the Five Lemma, $(f\circ g)_*=\mathop{id}_*:\mathrm{H}_q^I\mathbb{(X,A)}\to \mathrm{H}_q^I\mathbb{(X,A)}$ and  by Axiom $2$ one has that $(f\circ g)_*=f_*\circ g_*$ and then $f_*\circ g_*=\mathop{id}_\ast$. Similarly, $g_*\circ f_*=\mathop{id}_\ast :\mathrm{H}_q^I\mathbb{(Y,B)}\to \mathrm{H}_q^I\mathbb{(Y,B)}$.

Then, $f_*:\mathrm{H}_q^I\mathbb{(X,A)}\approx \mathrm{H}_q^I\mathbb{(Y,B)}$ is an isomorphism with inverse $(f_*)^{-1}=g_*$.
\end{proof}

\begin{coro}\label{coro:contequi-to-iso}
A contiguity $f$ from $(\mathbb{X},\mathbb{A})$ to $(\mathbb{Y},\mathbb{B})$ induces an isomorphism of the ordinary and reduced persistent homology sequences of $\mathbb{(X,A)}$ with the corresponding sequence of $\mathbb{(Y,B)}$.
\end{coro}

\begin{proof}
Let $f,g$ be contiguously equivalent from $\mathbb{(X,A)}$ into $\mathbb{(Y,B)}$. Then it is defined  contiguously equivalences $f_1,g_1$ from $\mathbb{X}$ into $\mathbb{Y}$ and $f_2,g_2$ from $\mathbb{A}$ into $\mathbb{B}$.

On the one hand, Corollary \ref{coro:contequi-to-iso} implies that $f,f_1$ and $f_2$ give rise to isomorphisms $f_*:\mathrm{H}_q^I\mathbb{(X,A)}\approx \mathrm{H}_q^I\mathbb{(Y,B)}$, ${f_1}_*:\mathrm{H}_q^I\mathbb{(X)}\approx \mathrm{H}_q^I\mathbb{(Y)}$ and ${f_2}_*:\mathrm{H}_q^I\mathbb{(A)}\approx \mathrm{H}_q^I\mathbb{(B)}$. In the other hand, the same result implies that $g,g_1$ and $g_2$ give rise to inverse isomorphisms $g_*:\mathrm{H}_q^I\mathbb{(Y,B)}\approx \mathrm{H}_q^I\mathbb{(X,A)}$, ${g_1}_*:\mathrm{H}_q^I\mathbb{(Y)}\approx \mathrm{H}_q^I\mathbb{(X)}$ and ${g_2}_*:\mathrm{H}_q^I\mathbb{(B)}\approx \mathrm{H}_q^I\mathbb{(A)}$. Then, we have an isomorphism between the ordinary persistent homology sequence of $\mathbb{(X,A)}$ with the corresponding sequence of $\mathbb{(Y,B)}$.

The same argument holds for the reduced persistent homology sequence, but instead of Theorem \ref{thm:contequi-to-iso}, in that case, we would have to invoke Theorem \ref{thm:iso-onto} and Theorem \ref{theo:red-seq}.

\end{proof}

\begin{theo}\label{thm:contequi-to-point}
Every filtered set contiguously equivalent to a point is homologically trivial.
\end{theo}
\begin{proof}
A filtered set $\mathbb{X}$ is said to be homologically trivial if $\mathrm{H}_q^I\mathbb{(X)}=0$ for $q\neq 0$ and $\tilde{\mathrm{H}}_0^I\mathbb{(X)}=0$.

Since the set is contiguously equivalent to a point $P_0$, using Theorem \ref{thm:contequi-to-iso}, there exists an isomorphism between their respective ordinary and reduced persistent homology sequences. 

Also considering that $\mathrm{H}_q^I(\mathbb{P}_0)=0$ for $q\neq 0$ and $\tilde{\mathrm{H}}_0^I(\mathbb{P}_0)=0$, the result follows.
\end{proof}

\begin{defi}
Let $\mathbb{(X',A')}$ and $\mathbb{(X,A)}$ be relative filtered sets with $(X',A')\subset (X,A)$ and $F_{X'}(\sigma)\leq F_{X}(\sigma)$ {for all $\sigma\subset X'$}. The relative set  $\mathbb{(X',A')}$ is a \textbf{retract} of the relative set $\mathbb{(X,A)}$ if there exists a filtration preserving map $f:\mathbb{(X,A)}\longrightarrow \mathbb{(X',A')}$, which we call \textbf{retraction}, such that $f(x')=x'$ for each $x'\in X'$.
It is called a \textbf{deformation retract} if there is a retraction $f$ such that the composition of $f$ with the inclusion map $\mathbb{(X',A')}\subset \mathbb{(X,A)}$ and the identity map of $\mathbb{(X,A)}$are contiguous maps.
\end{defi}

\begin{lemma}\label{lemma:deformation}
If $\mathbb{(X',A')}$ is a deformation retract of $\mathbb{(X,A)}$, the inclusion map $i:\mathbb{(X',A')}\subset\mathbb{(X,A)}$ and the retraction $f$ are  contiguously equivalent maps.
\end{lemma}
\begin{proof}
In order to prove the result, we must show that
$i\circ f$ and $\mathop{id}_{\mathbb{(X,A)}}$ are contiguous maps and $f\circ i$ and $\mathop{id}_{\mathbb{(X',A')}}$ are contiguous maps.

Since $f$ is a retraction, $f$ is a filtration preserving map such that $f(x)=x$ for each $x\in X'$, then $f\circ i=\mathop{id}_{\mathbb{(X',A')}}$ and, in particular, the maps are contiguous.

By definition, deformation retract means that the composition of $f$ with $i$ is contiguous to the identity map of $\mathbb{(X,A)}$, which ends the proof.
\end{proof}

\begin{theo}
If $\mathbb{(X',A')}$ is a deformation retract of $\mathbb{(X,A)}$, the inclusion map $\mathbb{(X',A')}\subset\mathbb{(X,A)}$ induces an isomorphism from the persistent homology sequence of $\mathbb{(X',A')}$ onto that of $\mathbb{(X,A)}$.
\end{theo}
\begin{proof}
If $\mathbb{(X',A')}$ is a deformation retract of $\mathbb{(X,A)}$, by Lemma \ref{lemma:deformation}, the inclusion map $i:\mathbb{(X',A')}\subset\mathbb{(X,A)}$ and the retraction $f$ give a contiguously equivalence of $\mathbb{(X,A)}$ and $\mathbb{(X',A')}$. 

Applying Corollary \ref{coro:contequi-to-iso}, we have that this contiguously equivalence induces isomorphism of the persistent homology sequence of $\mathbb{(X',A')}$ with the corresponding sequence of $\mathbb{(X,A)}$.

\end{proof}

\begin{defi}
A $\textbf{filtered triad}$ $((X;X_1,X_2),(F_X;F_{X_1},F_{X_2}))$, denoted by $(\mathbb{X};\mathbb{X}_1,\mathbb{X}_2)$, consists of  filtered sets $\mathbb{X}$ , $\mathbb{X}_1$ and $\mathbb{X}_2$, with $\mathbb{X}_i\subset \mathbb{}$, $i=1,2$.
The filtered triad $(\mathbb{X};\mathbb{X}_1,\mathbb{X}_2)$ is called \textbf{proper} if the inclusion maps  $k_1: (\mathbb{X}_1, \mathbb{X}_1\cap \mathbb{X}_2)\subset (\mathbb{X}_1\cup \mathbb{X}_2, \mathbb{X}_2)$ and $k_2: (\mathbb{X}_2, \mathbb{X}_1\cap \mathbb{X}_2)\subset (\mathbb{X}_1\cup \mathbb{X}_2, \mathbb{X}_1)$ induce {isomorphisms} of  persistent homology groups in all dimensions.
\end{defi}

A useful observation is that if $(\mathbb{X};\mathbb{X}_1,\mathbb{X}_2)$ is a proper triad, then $(\mathbb{X}_1\cup \mathbb{X}_2;\mathbb{X}_1,\mathbb{X}_2)$ is proper as well.

\begin{theo}\label{thm:injec-rep}
A filtered triad $(\mathbb{X};\mathbb{X}_1,\mathbb{X}_2)$ is proper if and only if, the inclusion maps $i_l:(\mathbb{X}_j,\mathbb{X}_1\cap \mathbb{X}_2)\longrightarrow (\mathbb{X}_1\cup \mathbb{X}_2,\mathbb{X}_1\cap \mathbb{X}_2)$, $l=1,2$, give, for each $q$, an injective representation of  $\mathrm{H}^I_q(\mathbb{X}_1\cup \mathbb{X}_2,\mathbb{X}_1\cap \mathbb{X}_2)$ as a direct sum
{
that is,
every $u\in \mathrm{H}^{I}_q(\mathbb{X}_1\cup \mathbb{X}_2,\mathbb{X}_1\cap \mathbb{X}_2)$ can be expressed uniquely as $u=i_{1*}u_1+i_{2*}u_2$ for $u_j\in
 \mathrm{H}^{I}_q(\mathbb{X}_j,\mathbb{X}_1\cap \mathbb{X}_2)$, $j=1,2$.
}
\end{theo}
\begin{proof}
Let $j_l:(\mathbb{X}_1\cup \mathbb{X}_2, \mathbb{X}_1\cap \mathbb{X}_2)\longrightarrow (\mathbb{X}_1\cup \mathbb{X}_2,\mathbb{X}_l)$, $k_l:(\mathbb{X}_l,\mathbb{X}_1\cap \mathbb{X}_2)\longrightarrow (\mathbb{X}_1\cup \mathbb{X}_2,\mathbb{X}_{3-l})$, $l=1,2,$ be the inclusion maps.

Consider the diagram 
\[
\xymatrix
{
    \mathrm{H}_{n}^I(\mathbb{X}_1\cup \mathbb{X}_2,\mathbb{X}_1)    
        & 
        & \mathrm{H}_{n}^I(\mathbb{X}_1\cup \mathbb{X}_2,\mathbb{X}_2) 
    \\
        & \mathrm{H}_{n}^I(\mathbb{X}_1\cup \mathbb{X}_2,\mathbb{X}_1\cap \mathbb{X}_2) \ar[ur]^{j_{1*}} \ar[ul]^{j_{2*}}   
        & 
    \\
    \mathrm{H}_{n}^I(\mathbb{X}_2,\mathbb{X}_1\cap \mathbb{X}_2) \ar[uu]^{k_{1*}} \ar[ur]^{i_{2*}}  
        & 
        & \mathrm{H}_{n-1}^I(\mathbb{X}_1,\mathbb{X}_1\cap \mathbb{X}_2) \ar[uu]^{k_{2*}} \ar[ul]^{i_{1*}} 
}
\]

From the exactness of the persistent homology sequence of  triples  $(\mathbb{X}_1\cup \mathbb{X}_2,\mathbb{X}_1,\mathbb{X}_1\cap \mathbb{X}_2)$ and $(\mathbb{X}_1\cup \mathbb{X}_2,\mathbb{X}_2,\mathbb{X}_1\cap \mathbb{X}_2)$ we have that $\ker(j_{1*})=\mathrm{Im}(i_{1*})$ and $\ker(j_{2*})=\mathrm{Im}(i_{2*})$. 
Since the triad is proper, one has by definition that $k_{1*}$ and $k_{2*}$ are isomorphisms. Therefore  13.1 of \cite{foundations} implies a decomposition as a direct sum of $\mathrm{H}^I_q(\mathbb{X}_1\cup \mathbb{X}_2,\mathbb{X}_1\cap \mathbb{X}_2)$. 

Assuming the direct sum condition, we have that $\ker(i_{1*})=0$ and then, from the persistent homology sequence of the triple $(\mathbb{X}_1\cup \mathbb{X}_2,\mathbb{X}_1,\mathbb{X}_1\cap \mathbb{X}_2)$ we have $\overline{\alpha^{I}}=0$. By exactness of this sequence, $j_{1*}$ is surjective.
We must prove that $k_{1*}$ is an isomorphism. Let $m\in \mathrm{H}^I_q(\mathbb{X}_1\cup \mathbb{X}_2,\mathbb{X}_1)$, then $m=j_{1*}n$, for some $n\in \mathrm{H}^I_q(\mathbb{X}_1\cup \mathbb{X}_2,\mathbb{X}_1\cap \mathbb{X}_2)$. By hypothesis, $n=i_{1*}n_1+i_{2*}n_2$ and therefore $m=j_{1*}n=j_{1*}(i_{1*}n_1+i_{2*}n_2) = j_{1*}\circ i_{1*}(n_1)+j_{1*}\circ i_{2*}(n_2)=k_{1*}n_2$ which implies that $k_{1*}$ is onto.

Assuming now that $m\in \mathrm{H}^I_q(\mathbb{X}_2,\mathbb{X}_1\cap \mathbb{X}_2)$ and $k_{1*}m=0$, we have that $j_{1*}\circ i_{2*}(m)=0$. By exactness, there exists an $n\in \mathrm{H}^I_q(\mathbb{X}_1,\mathbb{X}_1\cap \mathbb{X}_2)$ with $i_{2*}m=i_{1*}n$. Then $i_{1*}(-n)+i_{2*}m=0$ and, by the direct sum condition, $u=0$, which implies that $k_{1*}$ is a monomorphism. 

 The same argument prove that $k_{2*}$ is also an isomorphism.
Therefore, the triad $(\mathbb{X};\mathbb{X}_1,\mathbb{X}_2)$ is proper.

\end{proof}

\begin{defi}
Let $(\mathbb{X};\mathbb{X}_1,\mathbb{X}_2)$ be a filtered proper triad. The composition of the homomorphisms
\[
\mathrm{H}^I_q (\mathbb{X},\mathbb{X}_1\cup \mathbb{X}_2) \stackrel{\alpha^{I}}{\longrightarrow} \mathrm{H}^I_{q-1} (\mathbb{X}_1\cup \mathbb{X}_2) \stackrel{l_{2*}}{\longrightarrow} \mathrm{H}^I_{q-1} (\mathbb{X}_1\cup \mathbb{X}_2,\mathbb{X}_2) \stackrel{k^{-1}_{2*}}{\longrightarrow} \mathrm{H}^I_{q-1} (\mathbb{X}_1,\mathbb{X}_1\cap \mathbb{X}_2), 
\]

\noindent where ${k_2}_\ast$ and ${l_2}_\ast$ are induced by inclusion maps, is called the \textbf{boundary operator} of the filtered proper triad and is also denoted, by $\partial$. The sequence 
\[
\cdots \longrightarrow
\mathrm{H}^I_q (\mathbb{X}_1,\mathbb{X}_1\cap \mathbb{X}_2) \stackrel{i_*}{\longrightarrow} \mathrm{H}^I_{q} (\mathbb{X},\mathbb{X}_2) \stackrel{j_*}{\longrightarrow} \mathrm{H}^I_{q} (\mathbb{X},\mathbb{X}_1\cup \mathbb{X}_2) \stackrel{\partial}{\longrightarrow} \mathrm{H}^I_{q-1} (\mathbb{X}_1,\mathbb{X}_1\cap \mathbb{X}_2)
\longrightarrow \cdots
\]

\noindent where $i_\ast$ and $j_\ast$ are induced by inclusion maps,
is called \textbf{persistent homology sequence} of the filtered proper triad.
\end{defi}

\begin{theo}
The persistent homology sequence of a filtered proper triad is exact.
\end{theo}
\begin{proof}
Looking to the following diagram
\[
\xymatrix
{
 \mathrm{H}_{n}^I(\mathbb{X},\mathbb{X}_2)
    & \mathrm{H}_{n}^I(\mathbb{X}_1\cup \mathbb{X}_2,\mathbb{X}_2)     \ar[l]_{\overline{i}_\ast} 
    & \mathrm{H}_{n}^I(\mathbb{X},\mathbb{X}_1\cup \mathbb{X}_2) \ar[dl]_{\alpha^{I}} \ar[l]_{\overline{\alpha}^{I}} 
    & \mathrm{H}_{n-1}^I(\mathbb{X},\mathbb{X}_2) \ar[l]_{\overline{j}_*}
\\
    & \mathrm{H}_{n}^I(\mathbb{X}_1, \mathbb{X}_1\cap \mathbb{X}_2) \ar[u]^{k_{2*}} \ar[ul]^{i_*} 
    & 
    &
}
\]

\noindent we can note that the persistent homology sequence of the triad $(\mathbb{X};\mathbb{X}_1, \mathbb{X}_2)$ is obtained from the persistent homology sequence of the triple $(\mathbb{X}, \mathbb{X}_1\cup \mathbb{X}_2, \mathbb{X}_2)$ by exchanging the group $\mathrm{H}_{n}^I(\mathbb{X}_1\cup \mathbb{X}_2,\mathbb{X}_2)$ by the group $\mathrm{H}_{n}^I(\mathbb{X}_1, \mathbb{X}_1\cap \mathbb{X}_2)$ under $k_{2*}$ and defining $\alpha^{I}$ such that $k_{2*}\circ \alpha^{I}=\overline{\alpha}^{I}$.

Since the triad is proper, $k_{2*}$ is an isomorphism and then the persistent homology sequences are isomorphic. From the fact that the persistent homology sequence of a triple is exact, it follows that the persistent homology sequence of a filtered proper triad is exact, as desired.

\end{proof}

Note that, if $(\mathbb{X};\mathbb{X}_1,\mathbb{X}_2)$ is a filtered triad with $\mathbb{X}_2\subset \mathbb{X}_1$, then $(\mathbb{X}; \mathbb{X}_1, \mathbb{X}_2)$ is a filtered proper triad and its persistent homology sequence reduces to the homology sequence of the triple $(\mathbb{X},\mathbb{X}_1,\mathbb{X}_2)$.

A map of filtered triads $f:(\mathbb{X};\mathbb{X}_1,\mathbb{X}_2)\to (\mathbb{Y};\mathbb{Y}_1,\mathbb{Y}_2)$ is a map $f:X\to Y$ such that $f:\mathbb{X}\to \mathbb{Y}$ and the restrictions $f_{\mathbb{X}_1}:\mathbb{X}_1\to \mathbb{Y}_1$ and $f_{\mathbb{X}_2}:\mathbb{X}_2\to \mathbb{Y}_2$ are filtration preserving maps.

\begin{theo}\label{incidence-triple}
If $(\mathbb{X};\mathbb{X}_1,\mathbb{X}_2)$ and $(\mathbb{Y};\mathbb{Y}_1,\mathbb{Y}_2)$ are filtered proper triads, and if $f:(\mathbb{X};\mathbb{X}_1,\mathbb{X}_2)\longrightarrow (\mathbb{Y};\mathbb{Y}_1,\mathbb{Y}_2)$ is a map of filtered triads, then $f$ induces a homomorphism $f_{*}$ from the persistent homology sequence of $(\mathbb{X};\mathbb{X}_1,\mathbb{X}_2)$ into that of $(\mathbb{Y};\mathbb{Y}_1,\mathbb{Y}_2)$. In particular, the boundary homomorphism for filtered  proper triads commutes with the respective induced homomorphism.
\end{theo}

\begin{proof}
Let us consider the following restriction maps $ f_1:(\mathbb{X}_1,\mathbb{X}_1\cap \mathbb{X}_2)\longrightarrow (\mathbb{Y}_1,\mathbb{Y}_1\cap \mathbb{Y}_2)$, \linebreak {$f_2:(\mathbb{X},\mathbb{X}_2)\longrightarrow (\mathbb{Y},\mathbb{Y}_2)$} and $f_3:(\mathbb{X},\mathbb{X}_1\cup \mathbb{X}_2)\longrightarrow (\mathbb{Y},\mathbb{Y}_1\cup \mathbb{Y}_2)$.

In the following diagram:
\[
\resizebox{\linewidth}{!}{
\xymatrix
{
 \ar[r]	
    & \mathrm{H}^I_q (\mathbb{X}_1,\mathbb{X}_1\cap \mathbb{X}_2) \ar[r]^{i_*} \ar[d]^{f_{1*}}
    & \mathrm{H}^I_{q-1} (\mathbb{X},\mathbb{X}_2)     \ar[d]^{f_{2*}} \ar[r]^{j_*} 
    & \mathrm{H}^I_{q-1} (\mathbb{X},\mathbb{X}_1\cup \mathbb{X}_2) \ar[r]^{\partial} \ar[d]^{f_{3*}} 
    & \mathrm{H}^I_{q-1} (\mathbb{X}_1,\mathbb{X}_1\cap \mathbb{X}_2) \ar[d]^{f_{1*}} \ar[r] 
    & \\
\ar[r]	
    & \mathrm{H}^I_q (\mathbb{Y}_1,\mathbb{Y}_1\cap \mathbb{Y}_2) \ar[r]^{i'_*} 
    & \mathrm{H}^I_{q-1} (\mathbb{Y},\mathbb{Y}_2)     \ar[r]^{j'_*} 
    & \mathrm{H}^I_{q-1} (\mathbb{Y},\mathbb{Y}_1\cup \mathbb{Y}_2) \ar[r]^{\partial'}
    & \mathrm{H}^I_{q-1} (\mathbb{Y}_1,\mathbb{Y}_1\cap \mathbb{Y}_2) \ar[r] 
    &
}}
\]
since $i_*$, $i'_*$, $j_*$ and $j'_*$ are inclusion maps,  we have the commutativity of the first two squares. Let us prove the commutativity of the last one. We can extend the last square as 
\[
\xymatrix
{
\mathrm{H}^I_{q-1} (\mathbb{X},\mathbb{X}_1\cup \mathbb{X}_2)     \ar[d]^{f_{3*}} \ar[r]^{\alpha^{I}} & \mathrm{H}^I_{q-1} (\mathbb{X}_1\cup \mathbb{X}_2,\mathbb{X}_2) \ar[r]^{k^{-1}_{*}} \ar[d]^{f_{4*}} & \mathrm{H}^I_{q-1} (\mathbb{X}_1,\mathbb{X}_1\cap \mathbb{X}_2) \ar[d]^{f_{1*}} &\\
\mathrm{H}^I_{q-1} (\mathbb{Y},\mathbb{Y}_1\cup \mathbb{Y}_2)     \ar[r]^{{\alpha'}^{I}} & \mathrm{H}^I_{q-1} (\mathbb{Y}_1\cup \mathbb{Y}_2,\mathbb{Y}_2) \ar[r]^{k^{'-1}_{*}} & \mathrm{H}^I_{q-1} (\mathbb{Y}_1,\mathbb{Y}_1\cap \mathbb{Y}_2) &
}
\]

\noindent where $f_4:(\mathbb{X}_1\cup \mathbb{X}_2,\mathbb{X}_2)\longrightarrow (\mathbb{Y}_1\cup \mathbb{Y}_2,\mathbb{Y}_2)$ is a restriction of $f$. The commutative of the left square follows using Theorem \ref{thm:induced-triple} and, using that $k$ and $k'$ are inclusions and Theorem \ref{thm:injec-rep}, we have the commutativity of the right one. Therefore  the commutativity of the last square of first diagram follows.

\end{proof}

\subsection{Persistent Homology Theory on Simplicial Complexes}

\begin{theo}\label{injective-representation}
{Let $r\geq 1$ be an integer and} let $\mathbb{X} = (X, F)$ and $(X_1,F_1),\dots,$ $(X_r,F_r),(A,\overline{F})$ be filtered sets such that
$$\mathbb{X}=(X_1\cup\dots\cup X_r\cup A,F_1\cup\dots\cup F_r\cup\overline{F})$$
\noindent and the inclusion maps $k_{i, j}:\mathbb{X}_i\cap \mathbb{X}_j\subset \mathbb{A}$ are filtration preserving maps
for $i\neq j$. For $i=1,\dots,r$, taking  $(A_i,\overline{F}_i)=(X_i\cap A,F_i\cap\overline{F})$ and $k_i:(\mathbb{X}_i, \mathbb{A}_i)\subset (\mathbb{X}, \mathbb{A})$,
then for all $q\in \mathbb{Z}$,
the homomorphism $k_{i*}:\mathrm{H}_{q}^I(\mathbb{X}_i,\mathbb{A}_i)\longrightarrow \mathrm{H}_{q}^I(\mathbb{X},\mathbb{A})$  gives  an injective representation of $\mathrm{H}_{q}^I\mathbb{(X,A)}$ as a direct sum, 
{
that is, each $u\in \mathrm{H}_{q}^I\mathbb{(X,A)}$ can be expressed uniquely as $u = \sum_{i} k_{i\ast} (u_i)$, where $u_i\in \mathrm{H}_{q}^I(\mathbb{X}_i,\mathbb{A}_i)$.
}
\end{theo}

\proof For each $q\in \mathbb{Z}$, the proof follows by induction on $r$. 

The case $r=1$ follows by Axiom \ref{exision axiom}, since $\mathbb{X} = \mathbb{X}_1\cup \mathbb{A}$ and $\mathbb{A}_1 = \mathbb{X}_1\cap \mathbb{A}$.

Assume, inductively, that the theorem has been proved for $r=n\geq 1$, and assume that $r = n+1$. Set $\mathbb{X}' = \mathbb{X}_1\cup \cdots \cup \mathbb{X}_n\cup \mathbb{A}$ and consider the inclusion maps
\[
(\mathbb{X}_i, \mathbb{A}_i)
\stackrel{{k'}_{i}}{\longrightarrow}
(\mathbb{X}', \mathbb{A})
\stackrel{j}{\longrightarrow}
(\mathbb{X}, \mathbb{A})  \qquad (i=1,\ldots , n) ,
\]
\[
(\mathbb{X}_{n+1}, \mathbb{A}_{n+1})
\stackrel{{k'}}{\longrightarrow}
(\mathbb{X}_{n+1}, \mathbb{A})
\stackrel{j'}{\longrightarrow}
(\mathbb{X}, \mathbb{A}) .
\]

${k_{i}}_\ast = j_\ast {{k'}_i}_\ast $ for $i=1, \ldots , n$, and ${k_{n+1}}_\ast = {j'}_\ast {k'}_\ast$. By Axiom \ref{exision axiom}, ${k'}_\ast$ is an isomorphism. Since $(\mathbb{X}; \mathbb{X}', \mathbb{X}_{n+1}\cup A)$ is a proper triad, with
\[
\mathbb{X}'\cup (\mathbb{X}_{n+1}\cup \mathbb{A}) = \mathbb{X}, \quad \mathbb{X}'\cap (\mathbb{X}_{n+1}\cup \mathbb{A}) = \mathbb{A},
\]
we obtain  from \ref{thm:injec-rep} that $j_\ast$ and ${j'}_\ast$ yield an injective representation of $H_q^{I} (\mathbb{X}, \mathbb{A})$ as a direct sum. 
By the inductive hypothesis, the maps ${k'_i}_\ast$, $i=1, \ldots, n$, give an injective representation of $H_q^I(\mathbb{X}', \mathbb{A})$ as a direct sum. This implies the conclusion of the theorem.
\endproof

\begin{defi}\label{def-simp-bord}
By a \textbf{$(q, \alpha)$-simplex $\mathbb{S}^q_\alpha$} we  mean a pair $(S^q_\alpha,F^q_\alpha)$, where $M\geq \alpha \geq 0$, $S^q_\alpha$ is a (ordered) set of $q+1$ points and $F^q_\alpha:\pow(S^q_\alpha)\longrightarrow{[0,M]}$ is defined by $F^q_\alpha(\sigma):=\alpha$ for all $\sigma\subset S^{q}_\alpha$. 

The boundary of the simplex, denoted by $\dot{\mathbb{S}}^q_\alpha$ is the pair $(S^q_\alpha,\dot{F}^q_\alpha)$, where $\dot{F}^q_\alpha(S^q_\alpha):=M$, and $\dot{F}^q_\alpha(\sigma):=\alpha $, otherwise.

Let $\mathbb{S}^{q-1}_\alpha$ be a $(q-1, \alpha)$-simplex satisfying $S^{q-1}_\alpha\subset S^q_\alpha$. Let $A$ be the vertex in $S^q_\alpha$ not in $S^{q-1}_\alpha$.
We will denote by $\mathbb{C}^{q-1}_\alpha$ the \textbf{closed star} of $A$ in $\dot{\mathbb{S}}^{q}_\alpha$, that is the filtered complex $(S^q_\alpha,F^c_\alpha)$, defined by 
\[
  (F^c_\alpha)(\sigma) := 
  \begin{cases}
  M , & \text { if } \sigma=S^q_\alpha \text{ or } S^{q-1}_\alpha;\\
  \alpha , & \text { otherwise. }
      \end{cases}
\]
\end{defi}

\begin{lemma}\label{homo-trivial-lemma}
The filtered sets $\mathbb{S}^q_\alpha$, $\mathbb{C}^{q-1}_\alpha$ and the pair $(\mathbb{S}^q_\alpha,\mathbb{C}^{q-1}_\alpha)$ are homologically trivial.
\end{lemma}

\begin{proof}
We have that $\mathbb{S}^q_\alpha$ and $\mathbb{C}^{q-1}_\alpha$ are contiguously equivalent to a point and by Theorem \ref{thm:contequi-to-point} are homologically trivial. 
Since $\mathbb{S}^q_\alpha$ and $\mathbb{C}^{q-1}_\alpha$ are homologically trivial, the Corollary \ref{coro-homo-triv} implies that $(\mathbb{S}^q_\alpha, \mathbb{C}^{q-1}_\alpha)$ is also homologically trivial.
\end{proof}

\begin{theo}\label{incidence-isomorphism}
The group $\mathrm{H}_p^I(\mathbb{S}^q_\alpha, \dot{\mathbb{S}}^q_\alpha)$ is isomorphic to the group $\mathrm{H}_{p-1}^I(\mathbb{S}^{q-1}_\alpha,\dot{\mathbb{S}}^{q-1}_\alpha)$.
\end{theo}

\begin{proof}
The homology sequence of the filtered proper triad $(\mathbb{S}^q_\alpha, \mathbb{S}^{q-1}_\alpha, \mathbb{C}^{q-1}_\alpha)$ is given by 
\[ 
{\longrightarrow} 
\mathrm{H}^I_{p} (\mathbb{S}^q_\alpha, \mathbb{C}^{q-1}_\alpha) \stackrel{j_*}{\longrightarrow} 
\mathrm{H}^I_{p} (\mathbb{S}^q_\alpha, \mathbb{S}^{q-1}_\alpha\cup \mathbb{C}^{q-1}_\alpha) \stackrel{\alpha^{I}}{\longrightarrow} 
\mathrm{H}^I_{p-1} (\mathbb{S}^{q-1}_\alpha, \mathbb{S}^{q-1}_\alpha\cap \mathbb{C}^{q-1}_\alpha)\stackrel{i_*}{\longrightarrow} 
\mathrm{H}^I_{p-1} (\mathbb{S}^q_\alpha, \mathbb{C}^{q-1}_\alpha)
{\longrightarrow} 
\]

We have that $\mathbb{S}^{q-1}_\alpha\cap\, \mathbb{C}^{q-1}_\alpha=\dot{\mathbb{S}}^{q-1}_\alpha$, $\mathbb{S}^{q-1}_\alpha\cup \mathbb{C}^{q-1}_\alpha=\dot{\mathbb{S}}^{q}_\alpha$ and,  by Lemma \ref{homo-trivial-lemma}, $(\mathbb{S}^q_\alpha, \mathbb{C}^{q-1}_\alpha)$ is homologically trivial. Then, the groups $\mathrm{H}^I_{p} (\mathbb{S}^q_\alpha, \mathbb{C}^{q-1}_\alpha)$ and $\mathrm{H}^I_{p-1} (\mathbb{S}^q_\alpha, \mathbb{C}^{q-1}_\alpha)$ are zero and the sequence becomes
$$ 
0
\longrightarrow \mathrm{H}^I_{p} (\mathbb{S}^q_\alpha,\dot{\mathbb{S}}^{q}_\alpha) \stackrel{\alpha^{I}}{\longrightarrow} \mathrm{H}^I_{p-1} (\mathbb{S}^{q-1}_\alpha,\dot{\mathbb{S}}^{q-1}_\alpha)\longrightarrow 0 $$

By the exactness of the sequence, we have that $\alpha^{I}$ is an isomorphism.

\end{proof}

\begin{defi}\label{incidence-iso}
The isomorphism $\alpha^{I}: \mathrm{H}^I_q(\mathbb{S}^q_\alpha,\dot{\mathbb{S}}^q_\alpha)\longrightarrow \mathrm{H}^I_{q-1}(\mathbb{S}^{q-1}_\alpha,\dot{\mathbb{S}}^{q-1}_\alpha)$ is called the \textbf{incidence isomorphism} and will be denoted by
$$ [\mathbb{S}^q_\alpha : \mathbb{S}^{q-1}_\alpha]: \mathrm{H}^I_q(\mathbb{S}^q_\alpha,\dot{\mathbb{S}}^q_\alpha)\approx \mathrm{H}^I_{q-1}(\mathbb{S}^{q-1}_\alpha,\dot{\mathbb{S}}^{q-1}_\alpha)$$

\end{defi}
\smallskip

\begin{theo}\label{pers-homo-simplex}
The persistent homology groups of $(\mathbb{S}^q_\alpha,\dot{\mathbb{S}}^q_\alpha)$ are as follows:
$$
\mathrm{H}_q^I(\mathbb{S}^q_\alpha,\dot{\mathbb{S}}^q_\alpha)\approx G^I_\alpha
$$
where $G^I_\alpha$ is defined in {Definition} \ref{def-persist-group} and 
$$\mathrm{H}_p^I(\mathbb{S}^q_\alpha,\dot{\mathbb{S}}^q_\alpha) =0, \text { for } p\neq q . $$ 
\end{theo}

\begin{proof}
By Theorem \ref{incidence-isomorphism} we have $\alpha^{I}: \mathrm{H}_p^I(\mathbb{S}^q_\alpha,\dot{\mathbb{S}}^q_\alpha)\approx \mathrm{H}_{p-1}^I(\mathbb{S}^{q-1}_\alpha,\dot{\mathbb{S}}^{q-1}_\alpha)$. Then, inductively we have
$\mathrm{H}_p^I(\mathbb{S}^q_\alpha,\dot{\mathbb{S}}^q_\alpha)
\approx
\mathrm{H}_{q-1}^I(\mathbb{S}^{p-1}_\alpha,\dot{\mathbb{S}}^{q-1}_\alpha)
\approx 
\cdots
\approx
\mathrm{H}_{p-q}^I(\mathbb{S}^{0}_\alpha,\dot{\mathbb{S}}^{0}_\alpha)=\mathrm{H}_{p-q}^I(\mathbb{S}^{0}_\alpha)$ and the result follows.

\end{proof}

\begin{coro}\label{correspondence-ordered-simplex}
For every ordered $(q, \alpha)$-simplex $\mathbb{S}^q_\alpha$, the correspondence $g\longrightarrow g S^q_\alpha$ is an isomorphism between $G^I_\alpha$ and $\mathrm{H}^I_q(\mathbb{S}^q_\alpha,\dot{\mathbb{S}}^q_\alpha)$.
\end{coro}
\begin{proof}
This result follows from Theorem \ref{thm:iso-onto} and Theorem \ref{pers-homo-simplex}.
\end{proof}

\begin{lemma}\label{commu-simplex}
Let $q>0$ and let $\mathbb{S}_{\alpha}^q$ and ${\mathbb{S}_{\alpha}'}^q$ be $(q,\alpha)$-simplices. Following the notation {from} Definition \ref{def-simp-bord}, if $f:(\mathbb{S}_{\alpha}^q,\dot{\mathbb{S}_{\alpha}}^q)\longrightarrow ({\mathbb{S}_{\alpha}'}^q,\dot{\mathbb{S}_{\alpha}'}^q)$ is such that $\mathbb{S}_{\alpha}^q$ and $\mathbb{C}_{\alpha}^{q-1}$ are mapped into ${\mathbb{S}_{\alpha}'}^q$ and ${\mathbb{C}_{\alpha}'}^{q-1}$, respectively, then {the following diagram is commutative:}
\[
\xymatrix
{
	\mathrm{H}_{q}^I(\mathbb{S}_{\alpha}^q,\dot{\mathbb{S}_{\alpha}}^{q})     \ar[d]_{f_*^{I}} \ar[rr]^{[\mathbb{S}_{\alpha}^q:\mathbb{S}_{\alpha}^{q-1}]} 
        & 
        & \mathrm{H}_{q-1}^I(\mathbb{S}_{\alpha}^{q-1},\dot{\mathbb{S}_{\alpha}}^{q-1}) \ar[d]^{f_{*}^{I}}  
    \\
	\mathrm{H}_{q}^I({\mathbb{S}_{\alpha}'}^q,\dot{\mathbb{S}_{\alpha}'}^q) \ar[rr]^{[{\mathbb{S}_{\alpha}'}^q:{\mathbb{S}_{\alpha}'}^{q-1}]} 
        & 
        &  \mathrm{H}_{q-1}^I({\mathbb{S}_{\alpha}'}^{q-1},\dot{\mathbb{S}_{\alpha}'}^{q-1})
}
\]
\end{lemma}
\begin{proof}
The incidence isomorphism coincides with the boundary operator of the filtered proper triad $(\mathbb{S}_{\alpha}^q, \mathbb{S}_{\alpha}^{q-1}, \mathbb{C}_{\alpha}^{q-1})$, in the case $p=q$. Therefore, the commutativity with induced homomorphism is already proved in Theorem \ref{incidence-triple}.
\end{proof}

\begin{defi}\label{defi:simplex}{
If $\mathbb{A}=(A, F_A)$ is a $(0,\alpha)$-simplex, with $A=\{a\}$, and $g\in G^I_\alpha$, let $gA$ denote the element $(ga)_A$ defined in Definition \ref{def-persist-group}.
}

Let $\mathbb{S}_{\alpha}^q$ be an ordered $q$-simplex with vertices $A^0 < A^1 < \cdots < A^q$. Define $\mathbb{S}_{\alpha}^k$, $k<q$, to be the ordered simplex with vertices $A^{q-k} <\cdots <A^q$. For each $g\in G^I_\alpha$, the element $gS^q_{\alpha}$ 
of
$\mathrm{H}^I_q(\mathbb{S}_{\alpha}^q,\dot{\mathbb{S}_{\alpha}}^q)$ is defined inductively by 

$$ gS^k_{\alpha}:=[\mathbb{S}_{\alpha}^k:\mathbb{S}_{\alpha}^{k-1}]^{-1}gS_{\alpha}^{k-1}.$$

\noindent Thus 
$$gS^q_{\alpha}=[\mathbb{S}_{\alpha}^q:\mathbb{S}_{\alpha}^{q-1}]^{-1}\cdots [\mathbb{S}_{\alpha}^1:\mathbb{A}^q]^{-1}gA^q.$$
\end{defi}

Using Lemma \ref{commu-simplex} we are able to prove the following result:

\begin{theo}\label{pres-order}
If ${\mathbb{S}_{\alpha}}_1$ 
and ${\mathbb{S}_{\alpha}}_2$ are ordered $(q,\alpha)$-simplices and $f:({\mathbb{S}_{\alpha}}_1,\dot{{\mathbb{S}_{\alpha}}}_1)\longrightarrow ({\mathbb{S}_{\alpha}}_2,\dot{{\mathbb{S}_{\alpha}}}_2)$ preserves both filtration and order, then $f^{I}_*(g{S_\alpha}_1)=g{S_\alpha}_2$.
\end{theo}

\begin{theo}
Let $A$ and $B$ be the vertices of the $1$-simplex $\mathbb{S}_\alpha^1$. Using a simplified notation $\mathbb{S}^0_\alpha=\dot{\mathbb{S}_\alpha}^1$, every element $h\in \mathrm{H}_0^I(\mathbb{S}_\alpha^{0})$ can be written uniquely as:
$$ h=(gA)_{S^0} + (g'B)_{S^0}, g,g'\in G^I_\alpha.$$

The element $h$ is in $\tilde{\mathrm{H}}^I_0(\mathbb{S}_\alpha^0)$ if and only if, 
$g+g' =0$ or equivalently,
$$ h= (gA)_{S^0} - (gB)_{S^0}, g\in G^I_\alpha.$$
\end{theo}

\begin{proof}

Consider the following commutative diagram:
\[
\begin{tikzcd}
 (S^0, A)  
    &  
    & (S^0, B)            
    \\
    & (S^0, \emptyset)  \arrow[ul, "j_1 "]\arrow[ur, "j_2"]
    &
    \\
	(B, \emptyset) \arrow[uu, "k_1 "]\arrow[ur, "i_2"]
    & 
    & (A, \emptyset)\arrow[uu, "k_2 "]\arrow[ul, "i_1"] 
    \\
    (P_\alpha,\emptyset)\arrow[u, "h_2"]
    &
    & (P_\alpha, \emptyset)\arrow[u, "h_1"] .
\end{tikzcd} 
\]

By the exactness of the homology sequences of the triples 
$ (\mathbb{S}_\alpha^0, \{A\}, \emptyset)$ and
$ (\mathbb{S}_\alpha^0, \{B\}, \emptyset)$,
$\mathrm{Im}\ i_\alpha = \ker j_\alpha $, $\alpha = 1,2$.
The maps $k_1$ and $k_2$ are excisions, and therefore $k_{1\ast}$ and $k_{2\ast}$ are isomorphisms.
Thus, the conditions of \cite{foundations}, 1-13.1, are satisfied and $i_{1\ast}$ and $i_{2\ast}$ yield a injective direct sum representation of $\mathrm{H}_q^{I} (\mathbb{S}_\alpha^0, \emptyset)$.
Moreover, since $h_{1\ast}$ and $h_{2\ast}$ are isomorphisms, the firs part of the theorem follows.

The second part is a consequence of theorem \ref{thm:filt-to-onto}.
\end{proof}

\begin{theo}\label{even-odd}
Let $f:(\mathbb{S}_\alpha,\dot{\mathbb{S}_\alpha})\longrightarrow (\mathbb{S}_\alpha,\dot{\mathbb{S}_\alpha})$ be a permutation of the vertices of the $(q, \alpha)$-simplex $\mathbb{S}_\alpha$ and $h\in \mathrm{H}^I_q(\mathbb{S}_\alpha,\dot{\mathbb{S}_\alpha})$. Then
\[
  f^{I}_*(h) = 
  \begin{cases}
  h , & \text { if } f \text { is even; } \\
  -h , & \text { if } f \text { is odd. } 
      \end{cases}
\]

\end{theo}

\begin{proof}
For $q=0$, the theorem is trivial. Let us consider the case $q=1$.

Denoting $\mathbb{S}_\alpha^{0}=\dot{\mathbb{S}}_\alpha^1$, it consists of the points $A$ and $B$ and let $f_0:\mathbb{S}_\alpha^0\longrightarrow \mathbb{S}_\alpha^{0}$ be the restriction of $f$. If $f$ is the identity, the result is trivial, thus we may assume $f(A)=B$ and $f(B)=A$. Then the following diagram is commutative:
\[
\xymatrix
{
	\mathrm{H}_{1}^I(\mathbb{S}_\alpha^1,\mathbb{S}_\alpha^0)     \ar[d]_{\partial} \ar[rr]^{f^{I}_*} & & \mathrm{H}_{1}^I(\mathbb{S}_\alpha^1, \mathbb{S}_\alpha^0) \ar[d]^{\partial}  \\
	\tilde{\mathrm{H}}_{0}^I(\mathbb{S}_\alpha^0) \ar[rr]^{f^{I}_{0*}} & &  \tilde{\mathrm{H}}_{0}^I(\mathbb{S}_\alpha^0)
}
\]

If $h\in \mathrm{H}^I_1(\mathbb{S}_\alpha^1, \mathbb{S}_\alpha^0)$, then by Theorem \ref{thm:rest-to-homo},
$ \partial h=(gA)_{S^0_\alpha}-(gB)_{S^0_\alpha}$
for some $g\in G^I$ and by Theorem \ref{thm:filt-to-onto}, one has that
$$\partial\circ f^{I}_*h=f^{I}_{0*}\circ \partial h=(g\circ f(A))_{S^{0}_\alpha}-(g\circ f(B))_{S^{0}_\alpha}= (gB)_{S^{0}_\alpha}-(gA)_{S^{0}_\alpha}=-\partial h .$$

Since $\mathrm{H}_{1}^I(\mathbb{S}_\alpha^1)=0$, the kernel of $\partial$ is zero. Hence $f^{I}_*h=-h$.

Assume, inductively, that the theorem is true for integers $q-1$, $q\geq 2$. Since every permutation is a product of simple permutations, it is sufficient to consider the case of a simple permutation $f$ of the vertices of $\mathbb{S}_\alpha$. Since $\mathbb{S}_\alpha$ has more than two vertices, there is a vertex $A_0$ such that $F(A_0)=A_0$. Let $\mathbb{S}_\alpha'$ be the $(q-1)$-face of $\mathbb{S}_\alpha$ opposite to the vertex $A_0$. Then f maps $\mathbb{S}_\alpha'$ onto itself and defines a permutation
$$ f':(\mathbb{S}_\alpha',\dot{\mathbb{S}}'_\alpha)\longrightarrow (\mathbb{S'}_\alpha,\dot{\mathbb{S}}'_\alpha).$$

Moreover, by Lemma \ref{commu-simplex}, the commutativity relation $[\mathbb{S}_\alpha:\mathbb{S}_\alpha']f^{I}_*=f'^{I}_*[\mathbb{S}_\alpha:\mathbb{S}_\alpha']$ holds.
\[
\xymatrix
{
	\mathrm{H}_{q}^I(\mathbb{S}_\alpha ,\dot{\mathbb{S}}_\alpha )     \ar[d]_{f_*^{I}} \ar[rr]^{[\mathbb{S}_\alpha :\mathbb{S}_\alpha']} 
        & 
        & \mathrm{H}_{q-1}^I(\mathbb{S}_\alpha',\dot{\mathbb{S}}'_\alpha) \ar[d]^{{f'}_{*}^{I}}  
    \\
	\mathrm{H}_{q}^I(\mathbb{S}_\alpha,\dot{\mathbb{S}}_\alpha) \ar[rr]^{[\mathbb{S}_\alpha:\mathbb{S}_\alpha']} 
        & 
        &  \mathrm{H}_{q-1}^I({\mathbb{S}_\alpha'},\dot{\mathbb{S}}'_\alpha)
}
\]

Since $f'$ is a simple permutation of the vertices of $\mathbb{S}_\alpha'$  and the theorem is assumed for the dimensions $q-1$, we have for $h\in \mathrm{H}^I_q(\mathbb{S}_\alpha,\dot{\mathbb{S}}_\alpha)$,
$$ f'^{I}_*[\mathbb{S}_\alpha:\mathbb{S}_\alpha']h=-[\mathbb{S}_\alpha:\mathbb{S}_\alpha']h.$$

This implies that
$$ [\mathbb{S}_\alpha:\mathbb{S}_\alpha']f^{I}_*h=-[\mathbb{S}_\alpha:\mathbb{S}_\alpha'](-h)$$

\noindent and, since $[\mathbb{S}_\alpha:\mathbb{S}_\alpha']$ is an isomorphism, $f^{I}_*h=-h$.
\end{proof}

\begin{theo}\label{two-ordered-simplex}
Let ${\mathbb{S}_\alpha}_1$ and ${\mathbb{S}_\alpha}_2$ be two ordered $(q,\alpha)$-simplices both mapped by the same unordered $(q,\alpha)$-simplex $\mathbb{S}_\alpha$. Then
$$ g{S_\alpha}_1=\pm g{S_\alpha}_2$$

\noindent according as the order of ${\mathbb{S}_\alpha}_2$ differs by an even permutation or odd permutation from the order of ${\mathbb{S}_\alpha}_1$.
\end{theo}

\begin{proof}
Let us denote by $f$ the filtration preserving map from $({\mathbb{S}_\alpha}_1, {\dot{\mathbb{{S}}}_\alpha}{}_1)$ to 
$({\mathbb{S}_\alpha}_2,{\dot{\mathbb{S}}_\alpha}{}_2)$ which maps ${{S}_\alpha}_2= \{x_0, x_1, \ldots, x_q \}$ onto ${S_\alpha}_1= \{ x_{\sigma(0)}, x_{\sigma(1)}, \ldots, x_{\sigma(q)} \} $, $f (x_{i}) = x_{\sigma(i)}$ $i = 1, \ldots, q$.

By Theorem \ref{pres-order} we have $f^{I}_*(gS_2)=gS_1$. Furthermore, Theorem \ref{even-odd} implies
$f^{I}_*(gS_2)=\pm gS_2$, and consequently $gS_1=\pm gS_2$.
\end{proof}

\begin{theo}
Let ${\mathbb{S}_\alpha}$ be an ordered simplex with vertices $A^0<\cdots <A^{q}$ and let ${\mathbb{S}_\alpha}_k$ be the face obtained by omitting the vertex $A^k$ and not changing the order of the others. Then
$$ [{\mathbb{S}_\alpha}:{\mathbb{S}_\alpha}_k]gS_\alpha =(-1)^kg{S_\alpha}_k.$$
\end{theo}
\begin{proof}
For the case $k=0$, the formula is exactly Definition \ref{defi:simplex}. The general case can be reduced to the case $k = 0$ as follows:
we obtain from ${\mathbb{S}_\alpha}$ a new ordered simplex $\overline{\mathbb{S}}_\alpha$ by moving the vertex $A^k$ in front of all the others. Then ${\mathbb{S}_\alpha}_k={\overline{\mathbb{S}}_\alpha}_0$ and $g{S_\alpha}_k=g{\overline{S}_\alpha}_0$, while using Theorem \ref{two-ordered-simplex} we have $gS_\alpha=(-1)^kg\overline{S}_\alpha$, concluding
$$[{\mathbb{S}_\alpha}:{\mathbb{S}_\alpha}_k]gS_\alpha=[\overline{\mathbb{S}}_\alpha:{\overline{\mathbb{S}}_\alpha}_0](-1)^kg\overline{S}_\alpha
=(-1)^kg{\overline{S}_\alpha}_0=(-1)^kg {S_\alpha}_k.$$
\end{proof}

\begin{defi}\label{qskltn}
Let $\mathbb{X}$ be a filtered set. The \textbf{$q$-dimensional skeleton} of $\mathbb{X}$, that will be denoted by $\mathbb{X}^{(q)} = (X^{(q)},F_X^{(q)})$, is the filtered set where the set $X^{(q)}=X$ and $F_X^{(q)}$ is defined by 
\[
  F_X^{(q)}(\sigma) := 
  \begin{cases}
  F_X(\sigma) , & \text { if } \# (\sigma)\leq q;\\
  M , & \text { otherwise. }
      \end{cases}
\]
\end{defi}

\begin{lemma}\label{lema5}
The relative filtered set 
$(\mathbb{X}^{(q)}\cup \mathbb{A}, \mathbb{X}^{(q-1)}\cup \mathbb{A})$
obtained from $((X,A),(F_X,F_A))$  has the property that if $p\neq q$, then $\mathrm{H}^I_p (\mathbb{X}^{(q)}\cup \mathbb{A}, \mathbb{X}^{(q-1)}\cup \mathbb{A})=0$.
\end{lemma}

\begin{proof}

 Let $({S_{\alpha_1}}, f_1), (S_{\alpha_{2}}, f_2), \ldots , (S_{\alpha_{r}}, f_r)$ be the 
$q$-simplices
in $(X, F_X)$, $S_{\alpha_i}\not\subset A$, with inclusion maps $k_i:S_{\alpha_i}\to (X^{(q)}\cup A)$, $i=1,\ldots, r$.
Then considering $\cup$ and $\cap$ operations on filtered sets one has
$$ \left({S}_{\alpha_1} \cup \cdots \cup {S}_{\alpha_r} \cup ({X}^{(q-1)}\cup {A}), f_1\cup\cdots\cup f_r\cup (F_X^{(q-1)}\cup F_A) \right)
= ({X}^{(q)}\cup {A}, F_X^{(q)}\cup F_A).$$

The filtered sets $(S_{\alpha_1}, f_1),  (S_{\alpha_2}, f_2), \ldots , (S_{\alpha_r}, f_r)$,
 together with the filtered set 
 \linebreak 
{$ (X^{(q-1)}\cup A, F_X^{(q-1)}\cup F_A)$} satisfies the conditions of Theorem \ref{injective-representation}, that is, the homomorphisms 
$k_{i*}:\mathrm{H}_{q}^I(\mathbb{S}_{\alpha_i}, \mathbb{S}_{\alpha_i} \cap (\mathbb{X}^{(q-1)}\cup \mathbb{A}))\longrightarrow \mathrm{H}_{q}^I(\mathbb{X}^{(q)}\cup \mathbb{A})$ 
is an injective representation of $\mathrm{H}_{q}^I(\mathbb{X}^{(q)}\cup \mathbb{A})$ as a direct sum for each $i$.

Notice that $\dot{\mathbb{S}}_{\alpha_i}=\mathbb{S}_{\alpha_i} \cap (\mathbb{X}^{(q-1)}\cup \mathbb{A})$. Then  all we need to show is that 
$$\mathrm{H}^I_p(\mathbb{S}_{\alpha_i},\dot{\mathbb{S}}_{\alpha_i})=0$$

\noindent for $p\neq q$, and the lemma follows.
As $\mathbb{S}_{\alpha_i}$ is a simplex, then this equality follows by Theorem \ref{pers-homo-simplex}.
\end{proof}

\begin{defi}
{The group of  $q$-chains of $(\mathbb{X}, \mathbb{A})$}, denoted by $\mathcal{C}^{I}_q(\mathbb{X}, \mathbb{A})$, is defined as
$$\mathcal{C}^{I}_q(\mathbb{X}, \mathbb{A}):=
\mathrm{H}^I_q (\mathbb{X}^{(q)}\cup \mathbb{A}, \mathbb{X}^{(q-1)}\cup \mathbb{A}).$$

If $f:(\mathbb{X}, \mathbb{A})\longrightarrow (\mathbb{X}', \mathbb{A}')$ is a filtration preserving map, $f$ induces homomorphism
$$ f_q^{I}:\mathcal{C}^{I}_q(\mathbb{X}, \mathbb{A})\longrightarrow \mathcal{C}^{I}_q(\mathbb{X}', \mathbb{A}').$$
\end{defi}

\begin{remark}\label{theo-chain}
For the case $q<0$ or $q>\#X$, $\mathcal{C}^{I}_q(\mathbb{X}, \mathbb{A})=0$ is a consequence of Lemma \ref{lemma:PH-trivial}
\end{remark}

\begin{defi}\label{defi:vertices}
Let $A^0,\dots,A^q$ be a finite sequence of vertices of a filtered set $\mathbb{X}$. For each $g\in G^I$ (see \ref{def-persist-group}), we define the element $gA^0\dots A^q\in \mathcal{C}^{I}_q(\mathbb{X}, \mathbb{A})$ as follows:

Let $\mathbb{S}_\alpha$ be an ordered $(q,\alpha)$-simplex with vertices $B^0<\cdots<B^q$ with $\alpha= F_X(\{A^0,\ldots, A^q\})$, and let $f:(\mathbb{S}_\alpha ,\dot{\mathbb{S}}_\alpha )\longrightarrow (\mathbb{X}, \mathbb{X})$ be a filtration preserving  map defined by $f(B^i):=A^i$. Then
$$ g A^0\dots A^q=f^{I}_*(g S_\alpha^q).$$
\end{defi}

\begin{theo}\label{permutation-of-vertex}
The association of g to $gA^0\dots A^q$ defines a homomorphism $G^I\longrightarrow \mathcal{C}^{I}_q(\mathbb{X},\mathbb{A})$, i.e. for $g_1, g_2\in G^I$
$$ (g_1+g_2)A^0\dots A^q=g_1A^0\dots A^q +g_2A^0\dots A^q.$$

If $i_0,\dots,i_q$ is a permutation of the array $0,\dots,q$, then
$$ gA^{i_0}\dots A^{i_q}=\pm gA^0\dots A^q$$

\noindent according if the permutation is even or odd. If some vertex occurs at least twice in $A^0,\dots ,A^q$, then $gA^0\dots A^q=0$. If $A^0,\dots,A^q$ are all in $A$, then $gA^0\dots A^q=0$.
\end{theo}

\begin{proof}
The first part of this theorem follows from Corollary \ref{correspondence-ordered-simplex} and the second part from Theorem \ref{two-ordered-simplex}.
\end{proof}

\begin{theo}\label{description-chain}
Let {$\mathbb{S}_{\alpha_1}^{q} ,\dots,\mathbb{S}_{\alpha_{s}}^{q} $} be the collection of simplices of $(\mathbb{X}, \mathbb{A})$ of order $q$ such that $S_{\alpha_{i}}^{q}\not\subset A$. Assume that for each $S_{\alpha_i}^{q}$, an order of its vertices $A_i^0<\cdots<A_i^q$ has been chosen. Then each $q$-chain $c$ of $(\mathbb{X}, \mathbb{A})$ can be written uniquely in the form
$$ c=\sum^{s}_{i=1} g_iA^0_i\cdots A_i^q,\qquad g_i\in G^I.$$
\end{theo}

\begin{proof}
Consider the inclusion $i_m:(\mathbb{S}_{\alpha_m}^{q},\dot{\mathbb{S}}_{\alpha_m}^{q})\subset (\mathbb{X}^{(q)}\cup \mathbb{A}, \mathbb{X}^{(q-1)}\cup \mathbb{A})$. By Corollary \ref{injective-representation} and from the fact that $\mathbb{X}^{(q)}\cup \mathbb{A} 
= \mathbb{S}_{\alpha_1}^{q}\cup \cdots \cup \mathbb{S}_{\alpha_s}^{q}\cup \mathbb{X}^{(q-1)}\cup \mathbb{A}$ and $\dot{\mathbb{S}}_{\alpha_m}^{q}=\mathbb{S}_{\alpha_m}^{q} \cap (\mathbb{X}^{(q-1)}\cup \mathbb{A})$ we have a unique representation
$$ c=\sum^{s}_{m=1} i_{m*}h_m.$$

For each $h_m$ we also have a correspondence by Theorem \ref{correspondence-ordered-simplex}, and one can write it uniquely as $h_m=g_m {S}_{\alpha_m}^{q}$. By Definition \ref{defi:vertices}, $i_{m*}(g_m S_{\alpha_m}^{q})=g_m A^0_m\dots A^q_m$. Hence 
\[
c = \sum^{s}_{m=1} i_{m*} h_m 
    =\sum^{s}_{m=1} i_{m*} (g_m S_{\alpha_m}^{q})
    =\sum^{s}_{m=1} g_m A^0_m \cdots A_m^q .
\]

\end{proof}

\begin{theo}\label{filt-pres-chain}
If $f:(\mathbb{X}, \mathbb{A})\longrightarrow (\mathbb{X}_1,\mathbb{A}_1)$ is a filtration preserving map, $g\in G^I$ and $gA^0\cdots A^q$ is a $q$-chain of $(\mathbb{X}, \mathbb{A})$, then
$$ f_q(gA^0\cdots A^q)=gf(A^0)\cdots f(A^q).$$
\end{theo}
\bigskip

In what follows, we will present results by using $\mathcal{C}^{I}_q(\mathbb{X}, \mathbb{A}) = \mathrm{H}^I_q (\mathbb{X}^{(q)}\cup \mathbb{A}, \mathbb{X}^{(q-1)}\cup \mathbb{A})$, which were proved previously  for the case of relative filtered sets. This adaptation is necessary to prove the main theorem of this section.

\begin{theo}\label{similar-1}
The sequence
\begin{align*}
0 \longrightarrow 
\mathrm{H}^I_q (\mathbb{A}^{(q)}, \mathbb{A}^{(q-1)})
\stackrel{i_q}{\longrightarrow} 
\mathrm{H}^I_q (\mathbb{X}^{(q)} , \mathbb{X}^{(q-1)} )
\stackrel{j_q}{\longrightarrow}
\mathrm{H}^I_q (\mathbb{X}^{(q)}\cup \mathbb{A}, \mathbb{X}^{(q-1)}\cup \mathbb{A})
\longrightarrow 0
\end{align*}
\noindent is exact and the image of $i_q$ is a direct summand of $\mathrm{H}^I_q (\mathbb{X}^{(q)}, \mathbb{X}^{(q-1)})$.
\end{theo}

\begin{defi}
\textbf{The boundary operator} for chains 
$$ \partial_q: \mathcal{C}^{I}_q(\mathbb{X}, \mathbb{A})\longrightarrow \mathcal{C}^{I}_{q-1}(\mathbb{X}, \mathbb{A})$$

\noindent is defined to be the boundary operator of the triple $(\mathbb{X}^{(q)}\cup \mathbb{A}, \mathbb{X}^{(q-1)}\cup \mathbb{A}, \mathbb{X}^{(q-2)}\cup \mathbb{A})$.

Explicitly, $\partial_q$ is the composition of the homomorphisms 
$$
\mathrm{H}^I_q (\mathbb{X}^{(q)}\cup \mathbb{A}, \mathbb{X}^{(q-1)}\cup \mathbb{A})
\stackrel{{\alpha^{I}}}{\longrightarrow} 
\mathrm{H}^I_{q-1}(\mathbb{X}^{(q-1)} \cup \mathbb{A})
\stackrel{{j}^{''}_\ast} {\longrightarrow}
\mathrm{H}^I_{q-1}(\mathbb{X}^{(q-1)}\cup \mathbb{A}, \mathbb{X}^{(q-2)}\cup \mathbb{A}).
$$
\end{defi}

\begin{theo}\label{boundary-chain}
The boundary operator has the following properties:

\begin{enumerate}
    \item $\partial_q\circ\partial_{q+1}=0$. 
    \item If $f:(\mathbb{X}, \mathbb{A})\longrightarrow(\mathbb{X}', \mathbb{A}')$ is a filtration preserving map, then $f^{I}_{q-1}\partial_q=\partial_q f^{I}_q$.
    \item The boundary operator for $q$-chains $\partial_q$ satisfies
$$\partial_q(gA^0\cdots A^q)=\sum_{k=0}^{q}(-1)^kgA^0\cdots \hat{A}^k\cdots A^q.$$
\end{enumerate}
\end{theo}

\section{Additional Basic Definitions}\label{apdxb}

\subsection{The category \texorpdfstring{$\mathscr{RF}$}{} and functor 
\texorpdfstring{$H$}{}}

 A \textbf{filtered set} is a pair $\mathbb{X}=(X, F_X)$, where  $F_X$ is a filtration over $X$. A map $f: \mathbb{X}\longrightarrow \mathbb{Y}$ is a \textbf{filtration preserving map} if $F_{X}\geq F_{Y}\circ f$.

A \textbf{relative filtered set} $(\mathbb{X},\mathbb{A}):=((X,A),(F_X,F_A))$ is a pair $(X,A)$ of finite sets, where $A\subset X$,  equipped with a pair of maps $(F_X,F_A)$ defined in the following way: the maps $F_X: \pow(X)\longrightarrow\mathbb{R}$ and $F_A:\pow(A)\longrightarrow \mathbb{R}$ are filtrations over ${X}$ and over ${A}$, respectively, and satisfy $F_A(\sigma)\geq F_X(\sigma)$ for all $\sigma\in \pow(A)$.
By a \textbf{filtration preserving map of relative filtered sets} $f: (\mathbb{X},\mathbb{A}) \longrightarrow (\mathbb{Y},\mathbb{B})$ we mean that 
    \begin{itemize}
        \item $f$ is a map of pairs $f:(X, A)\longrightarrow (Y, B)$, that is, $f:X\longrightarrow Y$ such that $F(A)\subset B$; 
        \item $f:(X, F_X) \longrightarrow (Y, F_Y)$ is a filtration preserving map and
        \item the restriction $f|_A:(A,F_A) \longrightarrow (B,F_B)$ is also a filtration preserving map.
\end{itemize}

\label{rf}
The category $\mathscr{R}\mathscr{F}$ will be defined as follows:

\begin{itemize}
    \item The elements of $\mathscr{R}\mathscr{F}$ are indexed pairs of the form $((\mathbb{X}, \mathbb{A}); I)$, where $(\mathbb{X}, \mathbb{A})$ are relative filtered pairs of finite sets and $I$ is a element of the poset ${\mathrm{Int}(\mathbb{R})} := \{ [\varepsilon, \varepsilon'], \, \varepsilon \leq \varepsilon']\}$;
    
    \item The collection of morphisms $\mathscr{R}\mathscr{F} \left( ((\mathbb{X}, \mathbb{A}); I), ((\mathbb{Y}, \mathbb{B}); J)\right)$ is empty if $I\neq J$ and otherwise it will consist of maps 
    $f: ((\mathbb{X}, \mathbb{A}); I)\longrightarrow ((\mathbb{Y}, \mathbb{B}); I)$, where $f$ represent a filtration preserving map
    $f^{I}: (\mathbb{X}, \mathbb{A})\longrightarrow (\mathbb{Y}, \mathbb{B})$.
\end{itemize}

Given a object $((\mathbb{X}, \mathbb{A}); I)$ in $\mathscr{RF}$, the identity map $\mathrm{id}: ((\mathbb{X}, \mathbb{A}); I)\longrightarrow ((\mathbb{X}, \mathbb{A}); I)$ represents the identity map of filtered pairs $\mathrm{id}:(\mathbb{X}, \mathbb{A})  \longrightarrow (\mathbb{X}, \mathbb{A})$. The composition law
\[
\circ:
\mathscr{RF} \left( ((\mathbb{X}, \mathbb{A}); I), ((\mathbb{Y}, \mathbb{B}); I)\right) 
\times 
\mathscr{RF} \left( ((\mathbb{Y}, \mathbb{B}); I), ((\mathbb{Z}, \mathbb{C}); I)\right) 
\longrightarrow 
\mathscr{RF} \left( ((\mathbb{X}, \mathbb{A}); I), ((\mathbb{Z}, \mathbb{C});I)\right)
\]
comes from the composition law for filtered maps of pairs which satisfies the associative law and the unity laws.
\smallskip

The relative persistent homology functor $\mathrm{H}:\mathscr{RF}\longrightarrow \mathscr{A}b$
associates to each object $ \left((\mathbb{X}, \mathbb{A}); I\right)$ to an abelian group $\mathrm{H}_\ast^{I}(\mathbb{X}, \mathbb{A})$
and to each map
$ f^{I}:(\mathbb{X}, \mathbb{A})\longrightarrow (\mathbb{Y}, \mathbb{B})$ a homomorphism $ f_\ast^{I}: \mathrm{H}_\ast^{I}(\mathbb{X}, \mathbb{A})\longrightarrow \mathrm{H}_\ast^{I}(\mathbb{Y}, \mathbb{B})$
as defined in \cref{defhomrel}.
\smallskip

Given an object $\left((\mathbb{X}, \mathbb{A}); I\right)$ and the identity morphism $\mathrm{id}: ((\mathbb{X}, \mathbb{A}); I)\longrightarrow ((\mathbb{X}, \mathbb{A}); I)$,
$\mathrm{id}_\ast^{I}([d]) = [\mathrm{id}^{\varepsilon'}_{\#}(d)] =[d]$, that is, $\mathrm{id}^{I}_\ast:\mathrm{H}_\ast^{I}(\mathbb{X}, \mathbb{A})\longrightarrow \mathrm{H}_\ast^{I}(\mathbb{Y}, \mathbb{B})$ is the identity map.
\smallskip

Given $f:((\mathbb{X}, \mathbb{A}); I)\longrightarrow ((\mathbb{Y}, \mathbb{B});I)$ and $g:((\mathbb{Y}, \mathbb{B});I)\longrightarrow ((\mathbb{Z}, \mathbb{C});I)$ one has
$ (g\circ f)^{I}_\ast ([d]) 
= [(g\circ f)^{\varepsilon'}_\#(d)] 
= [(g^{\varepsilon'}\circ f^{\varepsilon'})_\#(d)] 
= [(g^{\varepsilon'}_\#\circ f^{\varepsilon'}_\#)(d)]
= [g^{\varepsilon'}_\#(f^{\varepsilon'}_\#(d))]
= g^{I}_\ast ([f^{\varepsilon'}_\#)(d)])
= g^{I}_\ast\circ f^{I}_\ast ( [d] )$.

\subsection{Simplicial Complex}\label{simpcomp}

A \emph{simplicial complex} $K$ is a pair $K=(V, \mathcal{S})$ of a set $V$ together with a collection $\mathcal{S}$ of subsets of $V$ called \emph{simplices} with the following properties:
\begin{enumerate}[\bf i.]
    \item for all $v\in V$, $\{v\}\in \mathcal{S}$;
    \item if $\alpha\subset \sigma \in \mathcal{S}$, then $\alpha\in \mathcal{S}$.
\end{enumerate}
We call sets $\{v\}$  \emph{vertices} of $K$.
We also say that $\sigma\in \mathcal{S}$ is an $n$-\emph{simplex} or a \emph{simplex of dimension $n$} if $|\sigma |=n+1$.
If $\alpha\subset\sigma$, $\alpha$ is a \emph{face} of $\sigma$ and $\sigma$ is a \emph{coface} of $\alpha$. 
An orientation of $\sigma = \{v_0, v_1, \ldots, v_n\}$ is an equivalence class given by $(v_0, v_1, \ldots, v_n) \sim (v_{\tau(0)}, v_{\tau(1)}, \ldots, v_{\tau(n)})$ if $\tau$ is an even permutation.

A $n$-simplex may be realized geometrically as a convex hull of $n+1$  affinely independent points  ${\displaystyle u_{0},\dots ,u_{k}}$ in $\mathbb{R}^{d}$, $d\geq n$, which means that the $n$ vectors ${\displaystyle u_{1}-u_{0},\dots ,u_{k}-u_{0}}$ are linearly independent.. A realization provides us the familiar low-dimensional simplices: \emph{vertices}, \emph{edges}, \emph{triangles} and \emph{tetrahedron}. In a realized complex, the simplices must meet along common faces. A \emph{subcomplex} $L\subset K$ is a complex such that if we denote $L=(V_L, \mathcal{S}_L)$ and $K=(V_K, \mathcal{S}_K)$, then  $V_{L}\subset V_{K}$ and $\mathcal{S}_{L}\subset\mathcal{S}_{K}$.


\begin{thebibliography}{00}

\bibitem{Alexandroff}
    \textsc{Pavel Alexandroff}.
    \emph{Une d\'efinition des nombres de Betti pour un ensemble ferm\'e quelconque}, 
    C. R. Acad. Sei. Paris, 184, 317-319, 1927.

\bibitem{frosini8}
    \textsc{Michele d'Amico, Patrizio Frosini and Claudia Landi}. 
    \emph{Using matching distance in Size Theory: a survey}, 
    International Journal of Imaging Systems and Technology, 16(5):154-161, 2006.

\bibitem{S-B} 
    \textsc{Serguei Barannikov}.
    \emph{The Framed Morse Complex and its invariants},
    Adv.Soviet Math., 21:93-115, Amer. Math. Soc., Providence, RI, 1994.


\bibitem{frosini10}
    \textsc{Silvia Biasotti, Daniela Giorgi, Michela Spagnuolo, Bianca Falcidieno}. 
    \emph{Size functions for comparing 3D models}, 
    Pattern Recognitionb 41:2855-2873, 2008.

\bibitem{Ferri} 
    \textsc{Francesca Cagliari, Massimo Ferri and Paola Pozzi}. 
    \emph{Size functions from a categorical viewpoint}, 
    Acta Applicandae Mathematica, 67:225-235, 2001.

\bibitem{gunnar} 
    \textsc{Gunnar E. Carlsson}. 
    \emph{Topology and data}, 
    Bulletin of the American Mathematical Society, 46(2):255-308, 2009.

\bibitem{Gunnar-Facundo1} 
    \textsc{ Gunnar E. Carlsson and Facundo Memoli}. 
    \emph{Characterization, stability and convergence of hierarchical clustering methods}, 
    J. Mach. Learn, 2010.

\bibitem{Gunnar-Facundo2} 
    \textsc{ Gunnar E. Carlsson and Facundo Memoli}. 
    \emph{Classifying clustering schemes}, 
    Foundations of Computational Mathematics 13, 2013.

\bibitem{gunnar-Homotopy} 
    \textsc{Gunnar E. Carlsson}. 
    \emph{Persistent homology and applied homotopy theory}, 
    Handbook of Homotopy Theory, Chapman and Hall/CRC, 2020.

\bibitem{gunnarZomo} 
    \textsc{Gunnar E. Carlsson and Afra Zomorodian}. 
    \emph{The theory of multidimensional persistence}, 
    Discrete \& Computational Geometry, 42(1):71-93, 2009.

\bibitem{Cech}
    \textsc{Eduard \v{C}ech}. 
    \emph{Th\'eorie g\'en\'erale de l'homologie dans un espace quelconque}, 
    Fund. Math., 19, 149-183, 1932.

\bibitem{frosini9}
    \textsc{Andrea Cerri, Massimo Ferri, Daniela Giorgi}. 
    \emph{Retrieval of trademark images by means of size functions},
    Graphical Models, 68:451-471, 2006.

\bibitem{Chazal}
    \textsc{Frederic Chazal, Vin de Silva, Marc Glisse and Steve Oudot}. 
    \emph{The structure and stability of persistence modules}, 
    Vol. 10. Berlin: Springer, 2016.

\bibitem{Cohen2007} 
    \textsc{David Cohen-Steiner, Herbert Edelsbrunner and John Harer}.
    \emph{Stability of persistence diagrams}, 
    Discrete \& Computational Geometry Mathematical Society, 1:103-120, Springer, 2007.

\bibitem{marcothesis} 
    \textsc{Marco Contessoto}. 
    \emph{ Some Persistent Cohomology Invariants and an Axiomatic Version of Persistent Homology}, 
    PHD thesis, UNESP - S\~ao Jos\'e do Rio Preto, 2021.

\bibitem{Marco-F-L-A} 
    \textsc{Marco Contessoto, Facundo Memoli, Anastasios Stefanou, and Ling Zhou}. 
    \emph{Persistent Cup-length},
    Proceedings of the 38th International Symposium on Computational Geometry SoCG, 2022.

\bibitem{dawson1} 
    \textsc{Robert J. MacG Dawson}. 
    \emph{A simplification of the Eilenberg-Steenrod axioms for finite simplicial complexes}, 
    Journal of Pure and Applied Algebra, 53(3):257-265, 1988.

\bibitem{frosini7}
    \textsc{Fran\c coise Dibos, Patrizio Frosini and Denis Pasquignon}. 
    \emph{The use of size functions for comparison of shapes through differential invariants}, 
    Journal of Mathematical Imaging and Vision, 21(2):107-118, 2004.

\bibitem{J-D} 
    \textsc{Jean Dieudonn\'e}. 
    \emph{A History of Algebraic and Differential Topology}, 
    1900-1960, Modern Birkh\"auser Classics, DOI 10.1007/978-0-8176-4907-41, \copyright Birkh\"auser Boston, a part of Springer Science+Business Media, LLC, 2009.

\bibitem{Edelsbrunner} 
    \textsc{Herbert Edelsbrunner and John Harer}. 
    \emph{Persistent homology - a survey}, 
    Contemporary mathematics, 453: 257-282, 42(1):71-93, Providence, RI: American Mathematical Society, 2008.

\bibitem{E-L-Z}
    \textsc{Herbert Edelsbrunner, David Letscher, and Afra Zomorodian}. 
    \emph{Topological persistence and simplification},
    Discrete \& computational geometry 28:511-533, 2002.

\bibitem{Eilenberg}
    \textsc{Samuel Eilenberg and Norman Steenrod}. 
    \emph{Axiomatic approach to homology theory}, 
    Proceedings of the National Academy of Sciences of the United States of America, 31 (4):117-120, 1945.

\bibitem{foundations} 
    \textsc{Samuel Eilenberg and Norman Steenrod}. 
    \emph{Foundations of Algebraic Topology}, 
    Princeton University Press, Princeton, 1952.

\bibitem{frosini4}
    \textsc{Patrizio Frosini and Claudia Landi}. 
    \emph{Size functions and morphological transformations}, 
    Acta Applicandae Mathematicae, 49(1):85-104, 1997.

\bibitem{Frosini-Landi}
    \textsc{Patrizio Frosini and Claudia Landi}. 
    \emph{Size theory as a topological tool for computer vision}, 
    Pattern Recognition and Image Analysis, 9.4:596-603, 1999.

\bibitem{Frosini1990}
    \textsc{Patrizio Frosini}. 
    \emph{A distance for similarity classes of submanifolds of a Euclidean space}, 
    Bulletin of the Australian Mathematical Society, 42(3):407-416, 1990.

\bibitem{Frosini1992} 
    \textsc{Patrizio Frosini}. 
    \emph{Measuring shapes by size functions}, 
    Proc. SPIE 1607 Intelligent Robots and Computer Vision X: Algorithms and Techniques, 1 February, 1992.

\bibitem{Hatcher} 
    \textsc{Allan Hatcher}. 
    \emph{Algebraic Topology}, 
    Cambridge Univ. Press, 2000.

\bibitem{Lefschetz}
    \textsc{Solomon Lefschetz}. 
    \emph{Algebraic Topology}, 
    Amer. Math. Soc. Coll. Publ. No. 27, Providence, R. I., 1942.

\bibitem{Sunhyuk-Facundo} 
    \textsc{Sunhyuk Lim, Facundo Memoli and Osman Berat Okutan}. 
    \emph{Vietoris-Rips persistent homology, injective metric spaces, and the filling radius}, 
    Algebraic \& Geometric Topology, 24(2), pp.1019-1100.

\bibitem{KimFacundo} 
    \textsc{Woojin Kim and Facundo Memoli}. 
    \emph{Generalized persistence diagrams for persistence modules over posets}, 
    Journal of Applied and Computational Topology, 5:533-581, 2021.

\bibitem{Facundo-PHmetric} 
    \textsc{Facundo Memoli and Zhou Ling}. 
    \emph{Persistent homotopy groups of metric spaces}, 
    Journal of Topology and Analysis, 1-62, 2024.

\bibitem{Facundo-L-A} 
    \textsc{Facundo Memoli, Anastasios Stefanou, and Ling Zhou}.
    \emph{Persistent cup product structures and related invariants}, 
    Journal of Applied and Computational Topology, 8.1, 2024.

\bibitem{Morse}
    \textsc{Marston Morse}. 
    \emph{The Calculus of Variations in the Large}, 
    American Mathematical Society Colloquium Publication, 18, 1934.

\bibitem{morse}
    \textsc{Marston Morse}. 
    \emph{The foundations of a theory of the calculus of variations in the large in m-space}, 
    Transactions of the American Mathematical Society, 32(4), pp.599-631,  1930.

\bibitem{Munkres} 
    \textsc{James R. Munkres}. 
    \emph{Elements of algebraic topology}, 
    Addison-Wesley, Menlo Park, CA, 1984.

\bibitem{Amit-Generalized} 
    \textsc{Amit Patel}. 
    \emph{Generalized persistent diagrams}, 
    Journal of Applied and Computational Topology, 3: 397-419, Springer, 2018.

\bibitem{Poincare}
    \textsc{Henry Poincar\'e}. 
    \emph{Oeuvres}, 
    vol. VI, Gauthier-Villars, Paris, 1953.

\bibitem{Robins} 
    \textsc{Vanessa Robins}. 
    \emph{Towards computing homology from finite approximations}, 
    Topology Proceedings, 24:503-532, 1999.

\bibitem{Spanier} 
    \textsc{Edwin H. Spanier}. 
    \emph{Algebraic topology}, 
    Springer-Verlag, New York, 1981.

\bibitem{vuff} 
    \textsc{Alessandro Verri, Claudio Uras, Patrizio Frosini and Massimo Ferri}. 
    \emph{On the use of size functions for shape analysis},
    Biological Cybernetics, 70:99-107, 1993. 

\bibitem{frosini5}
    \textsc{Alessandro Verri and Claudio Uras}. 
    \emph{Metric-topological approach to shape representation and recognition}, 
    Image Vision Comput., 14:189-207, 1996.

\bibitem{frosini6}
    \textsc{Alessandro Verri and Claudio Uras}. 
    \emph{Computing size functions from edge maps}, 
    Internat. J. Comput. Vision, 23(2):169-183, 1997.

\bibitem{Thurey1998} 
    \textsc{Volker Thurey}. 
    \emph{A simplification of the Eilenberg-Steenrod axioms, 2}. 
    Journal of Pure and Applied Algebra, 129(2):207-213, 1998.

\bibitem{Vietoris}
    \textsc{Leopold Vietoris}. 
    \emph{\"Uber die Homologiegruppen der Vereinigung zweier Komplexe},
    Monatsh. f\"ur Math. u. Phys., 37, 159-162, 1930.

\bibitem{Afra-compPH} 
    \textsc{Afra Zomorodian and Gunnar E. Carlsson}. 
    \emph{Computing persistent homology}, 
    Discrete \& Computational Geometry, 33(2):249-274, Springer, 2005.

\end{thebibliography}
\end{document}